\documentclass[letterpaper, 10pt,reqno]{amsart}
\usepackage[usenames,dvipsnames]{xcolor}

\usepackage[portrait,margin=3cm]{geometry}
\usepackage{mathrsfs}

\usepackage{mdframed}
\usepackage{amssymb,amsthm,amsmath,amsfonts,amsbsy,latexsym}
\usepackage[colorlinks,citecolor=blue,urlcolor=blue,linkcolor=blue]{hyperref}
\usepackage{color}
\usepackage{graphicx}
\usepackage{bbm}
\usepackage{comment}
\usepackage{microtype}
\usepackage{graphicx, eurosym, comment, epstopdf, units, algorithmic, algorithm, textcomp, fancyhdr, url, color, soul, enumerate, tikz}
\usepackage[textsize=small]{todonotes}

\definecolor{MyDarkblue}{rgb}{0,0.08,0.50}
\definecolor{Brickred}{rgb}{0.65,0.08,0}

\hypersetup{
	colorlinks=true,       
	linkcolor=MyDarkblue,          
	citecolor=Brickred,        
	filecolor=red,      
	urlcolor=cyan           
}

\newtheorem*{theorem*}{Theorem}
\newtheorem{theorem}{Theorem}[section]
\newtheorem{lemma}[theorem]{Lemma}

\newtheorem{proposition}[theorem]{Proposition}
\newtheorem{corollary}[theorem]{Corollary}

\newtheorem{definition}[theorem]{Definition}
\newtheorem{assumption}[theorem]{Assumption}
\newtheorem{remark}[theorem]{Remark}
\newtheorem{claim}[theorem]{Claim}

\renewcommand{\P}{\mathbb{P}}

\newcommand{\Pv}{\mathbb{P}}

\newcommand{\Ev}{\mathbb{E}}

\newcommand{\EE}{\mathcal{E}}
\newcommand{\CC}{\mathcal{C}}

\newcommand{\eps}{\varepsilon}

\newcommand{\cG}{\mathcal{G}}

\newcommand{\sss}{\scriptscriptstyle}

\newcommand{\CE}{{\mathcal{E}}}

\newcommand{\e}{{\mathrm e}}

\numberwithin{equation}{section}

\newcommand{\R}{\mathbb{R}}
\newcommand{\N}{\mathbb{N}}
\newcommand{\Z}{\mathbb{Z}}

\newcommand{\intersect}{\cap}

\newcommand{\Poi}{\mathrm{Poi}}

\renewcommand{\emptyset}{\varnothing}

\newcommand{\CA}{\mathcal {A}}

\newcommand{\CG}{\mathcal {G}}
\newcommand{\CI}{\mathcal {I}}

\newcommand{\CL}{\mathcal {L}}

\newcommand{\CT}{\mathcal {T}}

\newcommand{\CV}{\mathcal {V}}

\newcommand*{\wih}{\widehat}
\newcommand*{\wt}{\widetilde}
\newcommand*{\vr}{\varrho}
\newcommand*{\la}{\lambda}

\newcommand*{\de}{\delta}

\newcommand*{\ve}{\varepsilon}

\newcommand*{\al}{\alpha}

\newcommand*{\be}{\begin{equation}}
\newcommand*{\ee}{\end{equation}}
\newcommand*{\ba}{\begin{aligned}}
	\newcommand*{\ea}{\end{aligned}}
\newcommand*{\barr}{\begin{array}{c}}
	\newcommand*{\earr}{\end{array}}

\def \toindis  {\buildrel {d}\over{\longrightarrow}}
\def \toas     {\buildrel {a.s.}\over{\longrightarrow}}

\newcommand*{\wit}{\widetilde}

\newcommand*{\ind}{\mathbbm{1}}
\def\namedlabel#1#2{\begingroup
	#2%
	\def\@currentlabel{#2}%
	\phantomsection\label{#1}\endgroup
}

\newcommand{\Emin}{\mathrm{EGIRG}_{W,L}(1-\xi_n)}
\newcommand{\Eplus}{\mathrm{EGIRG}_{W,L}(1+\xi_n)}
\newcommand{\Elambda}{\mathrm{EGIRG}_{W,L}(\lambda)}
\newcommand{\Eupper}{\overline{\mathrm{EGIRG}}_{W,L}(\lambda)}

\newcommand{\Eone}{\mathrm{EGIRG}_{W,L}(1)}
\newcommand{\BGIRG}{\mathrm{BGIRG}_{W,L}(n)}
\newcommand{\Xdn}{\widetilde{\mathcal{X}}_d(n)}

\DeclareMathOperator*{\argmax}{arg\,max}
\DeclareMathOperator*{\arginf}{arg\,inf}

\begin{document}
	
	\title[Explosion in Geometric Inhomogeneous Random Graphs]{Explosion in weighted Hyperbolic Random Graphs and Geometric Inhomogeneous Random Graphs}
	
	\date{\today}
	
	\author[Komj\'athy]{J\'ulia Komj\'athy}
	\author[Lodewijks]{Bas Lodewijks}
	\address{JK: Department of Mathematics and Computer Science, Eindhoven University of Technology, P.O.\ Box 513, 5600 MB Eindhoven, The Netherlands.}
	\address{BL: Department of Mathematical Sciences,
University of Bath,
Claverton Down,
Bath,
BA2 7AY,
United Kingdom.} 
	\email{j.komjathy@tue.nl, b.lodewijks@bath.ac.uk}

	\maketitle
 	
	\newtheorem{conjecture}[theorem]{Conjecture}

	\newcommand{\bes}{\begin{equation}}
	\newcommand{\ees}{\end{equation}}
	
	\renewcommand{\P}[1]{\mathbb{P}\left(#1\right)}
	\newcommand{\E}[1]{\mathbb{E}\left[#1\right]}
	\newcommand{\1}[1]{\mathbbm{1}\left(#1\right)}
	\renewcommand{\empty}{\varnothing}
	\newcommand{\M}{M_{d,\tau,\alpha}}
	\newcommand{\Mt}{\widetilde{M}_{d,\tau,\alpha}}
	\newcommand{\Ml}{\widetilde{M}_l}
	
	\newcommand{\np}{$N^P_y\left(x\right)$}
	\newcommand{\nb}{$N^B_y\left(x\right)$}
	\newcommand{\bki}{B_k^{\sss\left(i\right)}}
	\newcommand{\cki}{c_k^{\sss\left(i\right)}}
	\newcommand{\BL}[1]{\todo[inline]{Bas: #1}}

	\renewcommand{\wt}[1]{\widetilde{#1}}
	\newcommand{\GIRGu}{\text{GIRG}_{W,X,L}^{\text{upper}}}
	\newcommand{\xd}{\mathcal{X}_d}
	\numberwithin{equation}{section}
	
	\newcommand{\cb}{\color{black}}
	\newcommand{\crr}{\color{red}}
\newcommand{\cz}{\color{black}}
	
	\begin{abstract}
	In this paper we study weighted distances in scale-free spatial network models: hyperbolic random graphs,  geometric inhomogeneous random graphs and scale-free percolation. In hyperbolic random graphs, $n=\Theta(\e^{R/2})$ vertices are sampled independently from the hyperbolic disk with radius $R$ and two vertices are connected either when they are within hyperbolic distance $R$, or independently with a probability depending on the hyperbolic distance.  In geometric inhomogeneous random graphs, and in scale-free percolation,
	each vertex is given an independent weight and location from an underlying measured metric space and $\Z^d$, respectively, and two vertices are connected independently with a probability that is a function of their distance and their weights. 
	We assign independent and identically distributed (i.i.d.)\ weights to the edges of the obtained random graphs, and  
	investigate the weighted distance (the length of the shortest weighted path) between two uniformly chosen vertices, called \cb typical weighted distance\cz. In scale-free percolation, we study the weighted distance from the origin of vertex-sequences with norm tending to infinity.
	
	In particular, we study the case when the parameters are so that the degree distribution in the graph follows a power law with exponent $\tau\in(2,3)$ (infinite variance), and the edge-weight distribution is such that it produces an \emph{explosive} age-dependent branching process with power-law offspring distribution, that is, the branching process produces infinitely many individuals in finite time.  We show that in all three models, typical distances within the giant/infinite component \emph{converge in distribution}, that is, no re-scaling is necessary. This solves an open question in \cite{HofKom2017}. 
	
	The main tools of our proof are to develop elaborate couplings of the models to infinite versions, to follow the shortest paths to infinity and then to connect these paths by using weight-dependent percolation on the graphs, that is, we delete edges attached to vertices with higher weight with higher probability. We realise the percolation using the edge-weights: only very short edges connected to high weight vertices are allowed to stay, hence establishing arbitrarily short upper bounds for connections. 
\end{abstract}
	\section{Introduction}
	Many complex systems in our world can be described in an abstract way by networks and are studied in various fields. Examples are social networks, the Internet, the World Wide Web (WWW), biological networks like ecological networks, protein-protein interaction networks and the brain, infrastructure networks and many more \cite{KalKolGastBlas10,Newman03}. We often do not comprehend these networks at all in terms of their topology, mostly due to their enormous size and complexity. A better understanding of the global structure of the network of neurons in the brain could, for example, help understand the spread of the tau protein through the brain, causing Alzheimer's disease \cite{Alz3}. We do know that many real-world networks exhibit the following three properties:
	
	(1) The small-world property: distances within a network are logarithmic or double logarithmic (ultrasmall-world) \cite{DorMend}. Examples include social networks \cite{BackBolRosUgaVig12,NewmPark03,Milg67}, food webs \cite{MontSol02}, electric power grids and movie actor networks \cite{Watts99}.
	
	(2) The scale-free property: the number of neighbours or connections \cb of vertices in the network statistically follows \cz a power-law distribution, i.e. $\P{\text{degree}>k}\approx k^{-(\tau-1)}$, for some $\tau>1$ \cite{AlbBar02,AlbJeoBar99,Newman03}. Examples include the World Wide Web and the Internet, \cite{Faloutsos}, global cargo ship movements \cite{KalKolGastBlas10}, and food webs \cite{MontSol02}.
	
	(3) High clustering. Local clustering is the formation of locally concentrated clusters in a network and it is measured by the average clustering coefficient: the proportion of triangles present at a vertex versus the total number of possible triangles present at said vertex, averaged over all vertices \cite{WattsStrog98}. Many real-life networks show high clustering, e.g.\ biological networks such as food webs, \cb technological \cz networks as the World Wide Web, and social networks \cite{CouLel14,Man67, New09,NewmPark03,SerBog06}.

	The study of real-world networks could be further improved by the analysis of network models. In the last three decades, network science has become a significant field of study, both theoretically and experimentally. One of the first models that was proposed for studying real-world networks was the Watts-Strogatz model \cite{WattsStrog98}, that has the small-world property and shows clustering, but does not exhibit the just at that time discovered scale-free property, however. Barab\'asi and Albert proposed that power-law degree distributions are caused by a `preferential attachment' mechanism \cite{BarAlb99}. The Barab\'asi-Albert model, also known as the preferential attachment model, produces power-law degree distributions intrinsically by describing the growth mechanism of the network over time. Static models, on the contrary,  capture a `snapshot' of a network at a given time, and generally, their parameters can be tuned extrinsically.  The configuration model and variations of inhomogeneous random graph models such as the Norros-Reitu or Chung-Lu models \cite{Boll1980,BollJansRior07, ChungLu02.1,NorRei06}, serve as popular null-models for applied studies as well as accommodate rigorous mathematical analysis. The degree distribution and typical distances are extensively studied for these models  \cite{BriDeijMart06,ChungLu02.1,ChungLu02.2,HofHoogMiegh04,HofHoogZnam05,Janson08,NorRei06}.  Typical distances in these models show the same qualitative behaviour: the phase transition from small-world to ultrasmall-world takes place when the second moment of the degree (configuration model) or weight distribution (inhomogeneous random graphs) goes from being finite to being infinite. This also follows for the preferential attachment model \cite{DomHofHoog10} and it is believed that it holds for a very large class of models.
	
	The one property that most of these models do not exhibit, however, is clustering. As the graph size tends to infinity, most models have a vanishing average clustering coefficient. Therefore, among other methods to increase clustering such as the hierarchical configuration model \cite{HofLeeuwStege16}, spatial random graph models have been introduced, including the spatial preferential attachment model (SPA) \cite{AieBonCooJanss08}, power-law radius geometric random graphs \cite{Hirs17, Yuk06}, the hyperbolic random graph (HRG) \cite{BodFouMul15,BogPapaKriou10, KriPap10}, and the models that we study in this paper, scale-free percolation on $\Z^d$ (SFP) \cite{DeijHofHoog2013}, their continuum analogue on $\R^d$ \cite{DepWut13} and the geometric inhomogeneous random graph (GIRG) \cite{BriKeuLen15}, respectively. The SFP and GIRG can be seen as the spatial counterparts of the Norros-Reitu and the Chung-Lu graph, respectively. Most of these spatial models are equipped with a long-range parameter $\alpha>0$ that models how strongly edges are predicted in relation to their distance. 
	For the SPA, HRG and GIRG, the average clustering coefficient is bounded from below (for the SPA this is shown only when $\tau<3$) \cite{BriKeuLen17,CooFriePral12,GugPanaPeter12,JacMor15}, they possess power-law degrees \cite{AieBonCooJanss08,BriKeuLen17,DeijHofHoog2013,GugPanaPeter12}  and typical distances in the SFP and GIRG behave similar to other non-spatial models in the infinite variance degree regime \cite{BriKeuLen17, DeijHofHoog2013}. In the finite variance degree regime, geometry starts to play a more dominant role and depending on the long-rage parameter $\alpha$, distances interpolate from small-world to linear distances (partially known for the SFP) \cite{DeijHofHoog2013,DeHaWu15}.
	
	Understanding the network topology lays the foundation for studying random processes on these networks, such as random walks, or information or activity diffusion. These processes can be found in many real-world networks as well, e.g., virus spreading, optimal targeting of advertisements, page-rank and many more. Mathematical models for information diffusion include the contact process, bootstrap percolation and first passage percolation (FPP).
	Due to the novelty of spatial scale-free models, the mathematical understanding of processes on them is rather limited. Random walks on scale-free percolation was studied in \cite{HeyHulJor16},
	bootstrap percolation on hyperbolic random graphs and on GIRGs in \cite{CanFou16, KocLen16}, and FPP on SFP in some regimes in \cite{HofKom2017}.	 In this paper we intend to add to this knowledge by studying FPP on GIRG and also on SFP, solving an open problem in \cite{HofKom2017}. 
	
	First passage percolation stands for allocating an independent and identically distributed (i.i.d.) random edge weight to each edge from some underlying distribution and study how metric balls grow in the random metric induced by the edge-weights.  The edge-weights can be seen as transmission times of information, and hence weighted distances correspond to the time it takes for the information to reach one vertex when started from the other vertex. 
	FPP was originally introduced as a method to model the flow of a fluid through a porous medium in a random setting \cite{HammWels65}, and has become an important mathematical model in probability theory that has been studied extensively in grid-like graphs, see e.g.\ \cite{HammWels65,HowDoug04,SmyWier78}, and on non-spatial (random) graph models as well such as the complete graph \cite{HofHoogMiegh01, EckGoodHofNar12,EckGoodHofNar15.1,EckGoodHofNar15.2,Janson99}, the Erd{\H o}s-R\'enyi graph \cite{BhaHofHoog10}, the configuration model \cite{BhaHofHoog10.2,BhaHofHoog12,HofHoogZnam05} and the inhomogeneous random graph model \cite{KolKom15}. FPP in the `true' scale-free regime on the configuration model, (i.e. when the asymptotic variance of the degrees is infinite), was studied recently in  \cite{AdrKom17,BaroHofKom15,BaroHofKom16}, revealing a dichotomy of distances caused by the dichotomy of `explosiveness' vs `conservativeness' of the underlying branching processes (BP).	
	In particular, the local tree-like structure allows for the use of age-dependent branching processes when analysing FPP. Age-dependent BPs, introduced by Bellman and Harris in \cite{BellHar52}, are branching processes where each individual lives for an i.i.d.\ lifetime and upon death it gives birth to an i.i.d. number of offspring. When the mean offspring is infinite, for some lifetime distributions, it is possible that the BP creates infinitely many individuals in finite time. This phenomenon is called \emph{explosion}. For other lifetime distributions this is not the case, in which case the process is called \emph{conservative}. A necessary and sufficient criterion for explosion was given in a recent paper \cite{AmiDev13}.
	
	\emph{Our contribution} is that  we show that the weighted distances in GIRGs and SFP in the regime where the degrees have infinite variance, converge in distribution when the edge-weight distribution produces explosive BPs with infinite mean offspring distributions. We further identify the distributional limit. 
	The case when the edge-weight distribution is conservative has been studied in the SFP 
	model in \cite{HofKom2017}. Based on this, and the result in \cite{AdrKom17}, we expect that a similar result would hold for the GIRG as well. 
	We formulate our main result without the technical details in the following meta-theorem. \cb Let $F_L$ denote the probability distribution function of the non-negative random variable $L$, and equip every edge $e$ in the GIRG (resp.\ SFP) graph with an i.i.d.\ edge-weight $L_e$, a copy of $L$. The weight of a path in the network is defined as the sum of the edge-weights of the edges in the path. Let $d_L(u,v)$ denote \cb the weight of the least-weight path, called weighted distance, \cb between two vertices $u,v$ in the model \cb(see Def.~\ref{def:distances} below for a precise definition). \cz
\medskip

\begin{mdframed}
\begin{theorem}[Meta-Theorem]\label{meta-theorem}	Consider GIRG on $n$ vertices or SFP on $\Z^d$ with i.i.d.\ power-law vertex-weights, and i.i.d.\ edge-weights with \cb probability distribution function $F_L$\cz. Let the parameters of the model be so that the degree distribution follows a power-law with exponent $\tau\in(2,3)$, and $L$ as lifetime forms an \emph{explosive} age-dependent branching process with infinite mean power-law offspring distributions. \cb Let $u,v$ be uniformly chosen vertices in GIRG. Then $d_L(u,v)$  converges in distribution to an a.s.\ finite random variable, conditioned that $u,v$ are in the giant component of GIRG.\cz

\cb In the SFP, let  $u:=0$, the origin of $\Z^d$, and $v$ the closest vertex to $n\underline e$ for some arbitrary unit vector $\underline e$. Then $d_L(u,v)$  converges in distribution to an a.s.\ finite random variable, conditioned that $u,v$ are in the infinite component of SFP\cz. 
\end{theorem}
\end{mdframed}

\medskip
\cb The criterion on the edge-weight distribution $L$ seems somewhat vague, but it is explicit. 
In fact, the necessary and sufficient criterion for a BP to be explosive appeared first in \cite{AmiDev13} in great generality and for power-law offspring distributions it simplifies to an explicitly computable sum only involving the distribution function $F_L$, see \cite{KomCMJ} for a proof. Namely, $L$ forms an explosive age-dependent BP with power-law offspring distributions if and only if 
\be \sum_{k=1}^\infty F_L^{(-1)}(\mathrm e^{-\mathrm e^k})<\infty.\ee \cz
	In \cite[Theorem 7]{BriKeuLen15} it is shown that \emph{hyperbolic random graphs} (HRG) are a special case of GIRGs, i.e. for every set of parameters in a hyperbolic random graph, there is a set of parameters in GIRG that produce graphs that are equal in distribution. \cb This suggest that our result for GIRGs carries through for hyperbolic random graphs, however, it is not obvious that the parameter setting of HRG as a GIRG satisfies the conditions of Theorem \ref{meta-theorem}. We devote Section \ref{s:hyperbolic} to show that HRGs (when considered as GIRGs) satisfy all the (hidden) conditions in Theorem \ref{meta-theorem}, and hence we obtain the following corollary:\cz
	\begin{corollary} The statement of Theorem \ref{meta-theorem} for GIRG remains word-for-word valid for hyperbolic random graphs as well. \end{corollary}
	
	\subsection*{Notation}
	We write rhs and lhs for right-hand side and left-hand side, respectively, wrt for with respect to, rv for random variable, i.i.d.\ for independent and identically distributed, pdf for probability distribution function. Generally, we write $F_X$ for the probability distribution function of a rv $X$, and $F^{(-1)}_X$ for its generalised inverse function, defined as $F^{(-1)}_X(y):=\inf\{t\in\mathbb{R}:F_X(t)\geq y\}$, and we write $\mathcal{L}eb$ for the Lebesgue measure.
	An event $E$ happens almost surely (a.s.) when $\Pv(E)=1$ and and a sequence of events $(E_n)_{n\in\mathbb{N}}$ holds with high probability (whp) when $\lim_{n\rightarrow \infty}\Pv(E_n)=1$. For rvs $(X_n)_{n\in\mathbb{N}},X, Y$, we write $X_n\overset{d}{\rightarrow} X$ and $X\overset{d}{\geq} Y$ when the $X_n$ converges in distribution to $X$, and when $X$ stochastically dominates $Y$, \cb that is, for all $x \in \mathbb{R}: F_X(x)\leq F_Y(x)$\cz, respectively. The rvs $(X_n)_{n\in\N}$ are tight, if for every $\eps>0$ there exists a $K_\eps>0$ such that $\P{|X_n|> K_\eps}<\eps$ for all $n$. For $n\in\mathbb{N}$, let $[n]:=\{1,\ldots,n\}$, and for two vertices  $u,v$, let $u\leftrightarrow v$ denote the event that $u, v$ are connected by an edge $e=(u,v)$. $\|x-y\|$ denotes the Euclidean norm between $x$ and $y$. We denote by $\lfloor x\rfloor,\lceil x\rceil$ the lower and upper integer part of $x\in\mathbb{R}$, respectively, while, when $x\in\mathbb{R}^d$, $\lfloor x\rfloor,\lceil x\rceil$ denotes taking the upper/lower integer part element-wise.  We write $x\wedge y:=\min\{x,y\}$. We generally denote vertex/edge sets by $\mathcal V, \mathcal E$ while $|\cdot|$ denotes the size of the set. 
	
	\section{Model and results}\label{sec:Modelintroduction}

	We begin with introducing the Geometric Inhomogeneous Random Graph model (GIRG) from \cite{BriKeuLen15}.
	
	\begin{definition}[Geometric Inhomogeneous Random Graph]\label{def:GIRG} \cb Let $W^{(n)}\ge 1, L\ge 0$ be random variables. \cz
		Let $V:=\left[n\right]$, and consider a \cb measure space \cz $(\mathcal{X}, \nu)$ with $\nu(\mathcal{X})=1$. Assign to each vertex \cb$i\in [n]$ \cz an i.i.d.\ position vector $x_i\in \mathcal{X}$ sampled from the measure $\nu$, and a vertex-weight $W_i^{(n)}$, \cb an i.i.d.\ copy of  \cb $W^{(n)}$. Then, conditioned on $(x_i, W_i^{(n)})_{i\in [n]}$  edges are present  independently and for any $u,v\in [n]$,
		\be\label{eq:gn-intro}
		\mathbb{P}\big(u\leftrightarrow v \text{ in } \mathrm{GIRG}_{W,L}(n)\mid (x_i, W_i^{(n)})_{i\in [n]}\big):=g_n\big(x_u,x_v,(W_i^{(n)})_{i\in [n]}\big),
		\ee
		where $g_n: \mathcal{X} \times \mathcal{X}\times \mathbb{R}^n_+ \rightarrow [0,1]$ \cb measurable\cz. 
		Finally, assign to each present edge $e$ an edge-length $L_e$, an i.i.d.\ copy of $L$.\cz
		We denote the resulting graph by $\mathrm{GIRG}_{W,L}(n)$.
	\end{definition} 
	In \cite{BriKeuLen15}, when $\mathcal{X}=\mathcal X_d:=[0,1]^d$, with $d$ standing for the parameter of the dimension, the following bounds were assumed on $g_n$. 
	There is a parameter $\alpha>1$, and $0<c_1\leq C_1 \leq 1$, such that for all $n$ \cb and all $u,v\in[n]$\cz,
	\begin{equation}
	\label{GIRG}
	c_1\le\frac{g_n\big(x_u,x_v,(W_i^{(n)})_{i\in [n]}\big)}{1\wedge \|x_u-x_v\|^{-\alpha d}(W_u^{(n)} W_v^{(n)}/\sum_{i=1}^n W_i^{(n)})^\alpha} \le C_1.
	\end{equation}
	In this paper, we use a slightly different assumption that captures a larger class of $g_n$ when $\Ev[W^{(n)}]<\infty$.
	Let us introduce two functions, $\overline{g}, \underline{g}:\R^d\times \R_+\times \R_+\to [0,1]$, having parameters $\al, \gamma, \underline a_1,\overline a_1, a_2\in \R_+$:
	\be\label{eq:gunder}\ba
	\overline{g}(x,w_1,w_2)&:=1\wedge \overline a_1\|x\|^{-\alpha d}\left(w_1 w_2\right)^\alpha,\\
	\underline{g}(x,w_1,w_2)&:=\e^{-a_2((\log w_1)^\gamma+(\log w_2)^\gamma)}\wedge \underline a_1 \|x\|^{-\alpha d}\left(w_1 w_2\right)^\alpha,
	\ea
	\ee
	\begin{assumption}\label{assu:GIRGgen}
		There exist parameters $\al>1, \gamma<1$, $\overline a_1, \underline a_1$, $a_2 \in \R_+$ and $0<c_1\leq C_1 <\infty$, such that for all $n$ \cb and all $u,v\in[n]$\cz, $g_n$ in \eqref{eq:gn-intro} satisfies 
		\begin{equation}\label{GIRGgeneral}
		c_1 \cdot \underline{g}\big(n^{1/d}(x_u-x_v),W_u^{(n)},W_v^{(n)}\big) \le g_n\big(x_u,x_v,(W_i^{(n)})_{i\in [n]}\big)\le   C_1 \cdot \overline{g}\big(n^{1/d}(x_u-x_v),W_u^{(n)},W_v^{(n)}\big).
		\end{equation}
	\end{assumption}
	Comparing the upper and lower bounds in \eqref{GIRGgeneral} to those in \eqref{GIRG},  by multiplying the spatial difference $x_u-x_v$ by $n^{1/d}$ in $\overline g, \underline g$, we have replaced $\sum_{i=1}^n W_i^{(n)}$ in \eqref{GIRG} by a constant times $n$ in \eqref{GIRGgeneral}. For large $n$ this does not make a difference when $\Ev[W^{(n)}]<\infty$, since in this case the sum is asymptotically $n\Ev[W^{(n)}]$ by the Law of Large Numbers. A more important change is that we have altered the first argument of the minimum in the lower bound in \eqref{GIRG}. The reason for the extension of the lower bound in \eqref{GIRG} to \cb the weaker \cz form in \eqref{GIRGgeneral} is that models satisfying \eqref{GIRGgeneral} still show (asymptotically) the same qualitative behaviour as the ones satisfying \eqref{GIRG}, and when applying weight-dependent percolation, it becomes necessary to allow edge probabilities to satisfy only the lower bound in \eqref{GIRGgeneral} but not \eqref{GIRG}. We discuss this in more detail in Section \ref{sec:boxinginfinitecomponent}.
	\begin{figure}[h]
		\includegraphics[height=5.15cm]{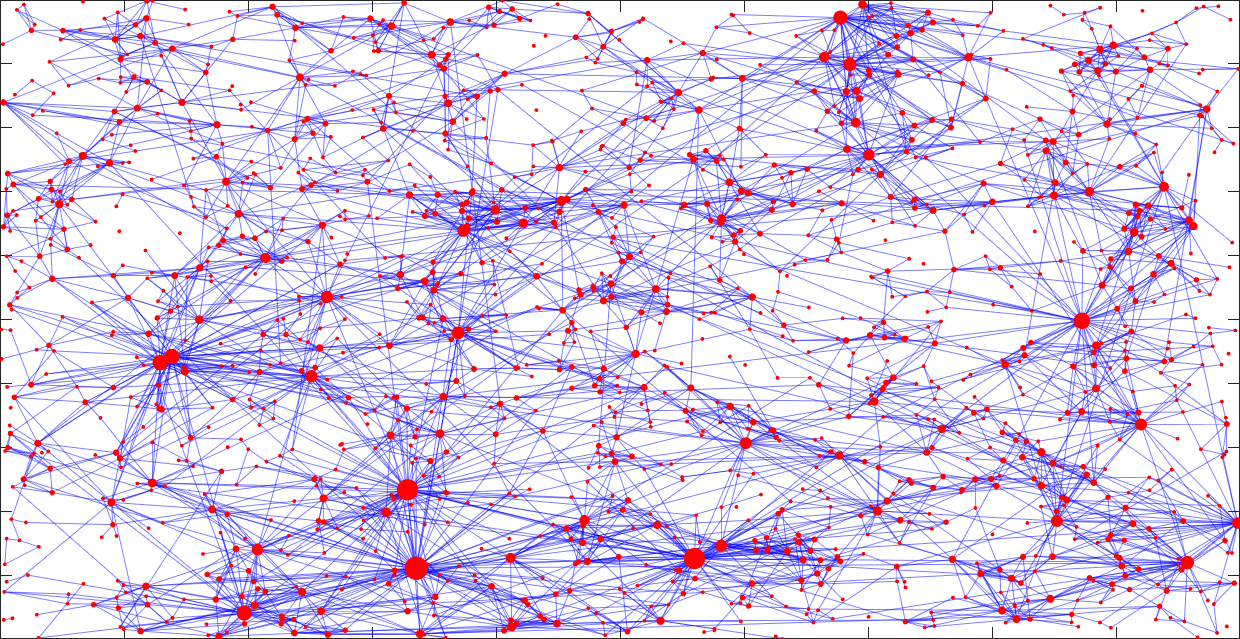}\caption{A realisation of the GIRG, for $\alpha=1.95,\tau=2.8525,d=2,n=1000$.}\label{Figure:GIRG}	
	\end{figure}
A similar model is scale-free percolation (SFP), introduced by Deijfen, van der Hofstad and Hooghiemstra in \cite{DeijHofHoog2013}.  In Section \ref{sec:Blowuppoissonization} \cb we discuss the relation \cz of GIRG and SFP in more detail. 	A realisation of (part of) the GIRG model with power-law weights can be seen in Figure \ref{Figure:GIRG}.
	
	\begin{definition}[Scale-free percolation]\label{def:SFP}Let $d\ge 1$ be an integer and $W\ge 1$ a random variable. We assign to each vertex $i\in\mathbb{Z}^d$ a random weight $W_i$, \cb an i.i.d.\ copy of $W$\cz. For a parameter $\widetilde{\alpha}>0$ and a percolation parameter $\lambda>0$, conditioned on $(W_i)_{i\in\Z^d}$, we  connect any two non-nearest neighbours $u,v\in\mathbb{Z}^d$ independently with probability 
		\be\label{SFP}
		\mathbb{P}\left(u\leftrightarrow v\mid \|u-v\|>1, \left(W_i\right)_{i\in\mathbb{Z}^d}\right)=1-\exp\left\{-\lambda\|u-v\|^{-\widetilde{\alpha}} W_u W_v\right\}.
		\ee
		\cb Vertices with distance at most one \cz are connected with probability $1$. Finally, we assign to each edge $e$ an edge-length $L_e$, \cb an i.i.d.\ copy of  $L$\cz. We denote the resulting random graph by $\mathrm{SFP}_{W,L}$.
	\end{definition}
	\subsubsection*{Vertex-weight distribution}
	To be able to produce power-law degree distributions, in both the GIRG  and SFP, it is generally interesting to study the case when the weight variables asymptotically follow a power-law distribution. \cb In the literature, (see e.g.\ Chung-Lu, or Norros Reitu model), generally the same vertex-weight distribution $W^{(n)}\equiv W$ is assumed for all values $n\ge 1$. For our results to carry through to hyperbolic random graphs it is necessary to allow \emph{$n$-dependent weight distributions} that converge to some limiting distribution.  \cb Hence we pose the following assumption on the weight distributions $W^{(n)}$\cz.  We say that a function varies slowly at infinity when $\lim_{x\rightarrow \infty} \ell\left(cx\right)/\ell\left(x\right)=1$ for all $c>0$.
 
	\begin{assumption}[Power-law limiting weights]\label{assu:weight} The vertex-weight distributions $(W^{(n)})_{n\ge 1}$ satisfy the following: 
	\begin{enumerate}
	\item
 $W^{(n)}\ge 1$ a.s. for all $n\ge 1$, 
 \item For all $n\ge 1$ there exists an $M_n\in \R_+$ with the property that $\Pv(W^{(n)}>M_n) =o(n^{-1})$ and that for all $x\in[1, M_n]$, 
		\begin{equation}\label{powerlaw-n}
		\Pv(W^{(n)}>x)=x^{-\left(\tau-1\right)}\ell^{(n)}(x) \quad \text{ and } \quad \underline \ell(x)\le \ell^{(n)}(x)\le \overline \ell(x),
		\end{equation}
		for some functions $\underline \ell(x), \overline \ell(x)$ that vary slowly at infinity. \item There exists a distribution $W$ and a function $\ell$ that varies slowly at infinity, such that
		\begin{equation}\label{powerlaw}
		\P{W>x}=\ell\left(x\right)x^{-\left(\tau-1\right)},
		\end{equation}
	and $W^{(n)}\ {\buildrel d \over \rightarrow}\ W$.
	\end{enumerate}
	\end{assumption}
	By possibly adjusting $g_n$, we have assumed that $\P{W^{(n)}\geq 1}=1$ to capture all $W^{(n)}$ with support separated from $0$. Note also that we do not require that $\ell^{(n)}(x)$ is slowly varying, it is enough if its limit $\ell(x)$ has this property. The reason for this rather weak assumption is that it allows for slightly truncated power-law distributions, and it is also necessary for HRG: there, $\ell^{(n)}(x)$ contains a term $(x/n)^\beta$ for some $\beta>0$, so $\ell^{(n)}(x) $ in itself is not slowly varying, only its limit, when this term vanishes as $n\to \infty$.
	
	\cb Recall that a joint distribution $(\widehat W^{(n)}, \widehat W)$ on $\R\times \R$ is a coupling of $W^{(n)}$ and $W$, if its first and second marginals have the law of $W^{(n)}$ and $W$, respectively. The coupling error is defined as $\Pv(\widehat W^{(n)}\neq\widehat W)$.
	 Recall that the total variation distance equals
	\be d_{\mathrm{TV}}(W^{(n)}, W) := \sup_{A \subset \R} | \Pv(W^{(n)}\in A) - \Pv(W \in A) |=\inf_{(\widehat W^{(n)}, \widehat W) \text{ coupling of }  W^{(n)}, W } \Pv(\widehat W^{(n)} \neq \widehat W).  \ee
	the total variation distance between $W^{(n)}$ and $W$.  By Assumption \ref{assu:weight}, $d_{\mathrm{TV}}(W^{(n)}, W) \to 0$ as $n\to \infty$. Further, whenever $d_{\mathrm{TV}}(W^{(n)}, W)\to 0$, (by definition of the infimum, see also \cite[Theorem 4.2]{Jans10})
 it is true that there exists a sequence of \emph{couplings} of the rvs $(W^{(n)})_{n\ge 1}, W$ such that the coupling error tends zero.
	Let us thus take such a sequence of couplings $(\widehat W^{(n)},\widehat W)$ and define 
	\be \label{eps-tv} \epsilon_{\mathrm TV}(n):= \Pv(\widehat W^{(n)}\neq\widehat W) \to 0.\ee \cz

	\subsection*{Results}
Theorems \ref{Th:GIRGexplosive} and \ref{Th:SFPexplosive} below, our main results,  
 are the precise versions of Theorem \ref{meta-theorem} above. Before formulating the results we introduce some extra assumptions, and models that help us state the distributional limits in our theorems.
	Let us express
	\be g_n(x_u, x_v, (W_i^{(n)})_{i\in[n]})=:\wit g_n(x_u, x_v, W_u^{(n)}, W_v^{(n)}, (W_{i}^{(n)})_{i\in[n]\setminus\{u,v\}}),\ee i.e. the notation $\wit g_n$ emphasises the weights of the vertices $u,v$ in $g_n$ in \eqref{eq:gn-intro}. Heuristically, the next assumption ensures that $g_n$, the edge-connection probability function, converges to some limiting function $h$ that only depends on the spatial distance between the two vertices and on their weights. 
 	\begin{assumption}[Limiting connection probabilities exists]\label{assu:extendable}\cb
		Set $(\mathcal{X}, \nu):=([-1/2, 1/2]^d, \mathcal{L}eb)$ for some integer $d\ge 1$. We assume that 
		\begin{enumerate}
		\item there exists an event $\CL_n$ measurable wrt to the $\sigma$-algebra generated by the weights $(W^{(n)}_i)_{i\in[n]}$ that satisfies $\lim_{n\to \infty}\Pv(\CL_n) =1$,
		\item there exist intervals $I_\Delta(n)\subseteq \R_+, I_w(n)\subseteq [1,\infty)$ and a sequence $\epsilon(n)\in \R_+$, such that $I_\Delta(n)\to (0, \infty), I_w(n)\to [1,\infty), \epsilon(n)\to 0$ as $n\to \infty$,
		\item there exists a function  $ h:  \R^d\times \R_+\times\R_+\to [0,1]$, 
		\end{enumerate}
		  such that whenever the (fixed) triplet $(\Delta,w_u,w_v)\in\R^d\times \R_+\times\R_+$ satisfies that $\|\Delta\|\in I_\Delta(n), w_u, w_v\in I_w(n)$, on the event $\CL_n$, $g_n$ in \eqref{eq:gn-intro} satisfies for $\nu$-almost every $x \in  \mathcal{X}$, \cb all $u,v\in[n]$  that 
		\be \label{eq:h-intro}
		\frac{|\wit g_n(x, x+\Delta/n^{1/d}, w_u, w_v, (W_i^{(n)})_{i\in [n]\setminus\{u,v\}})- h(\Delta, w_u, w_v)|}{h(\Delta, w_u, w_v)}
		\le  \epsilon(n).
		\ee \cz\end{assumption}
	\cb The function $h$ is the limit of the connection probability function $g_n$, when the two vertices involved are distance $\Delta/n^{1/d}$ away from each other. Since $\sum_{i=1}^nW_i^{(n)}=\Theta(n)$, the scale $\Theta(n^{1/d})$ is the scale for the distance between two vertices when $g_n$ is of constant order in \eqref{GIRG}. The function $h$ is important since it captures the limiting connection probabilities, that allows us to extend the model to the whole $\R^d$.  The arguments of $h$ represent the distance between two vertices $(\Delta/n^{1/d})$ and their weights, respectively.
 The need for the event $\CL_n$ is coming from the randomness of the $(W_i^{(n)})_{i\le n}$, a typical choice for $\CL_n$ is e.g.\ $\CL_n:=\{\sum_{i\le n} W_i^{(n)} - n \Ev[W^{(n)}]\in (-n^\beta, n^\beta) \}$ for some $\beta<1$. 	
 
Importantly, Assumption \ref{assu:extendable} implies that the connection probabilities converge whp to the function $h(\Delta, w_u, w_v)$, and this  limiting connection function only depends on the vertex-weights and the spatial distance of the vertices involved. Importantly, the limiting probability is \emph{translation invariant}.  The condition requires that in the range where the connection probabilities are not vanishing (i.e. the spatial difference between the two vertices is of order $n^{-1/d}$), the relative error between the connection probability and its limit is uniformly bounded by some function $\epsilon(n)$, on intervals that tend to the full possible range. This condition is satisfied when for instance $g_n$ equals either the upper or the lower bound in \eqref{GIRG} or in \eqref{GIRGgeneral} and the weights are i.i.d.\ with $\Ev[W^{(n)}]=\Ev[W]<\infty$. The bound \eqref{eq:h-intro} with $\epsilon(n)=n^{-\kappa(\tau)}, I_\Delta(n)=(0, \infty), I_w(n)=[1, \infty)$ for an appropriate $\kappa(\tau) >0$ can be shown using that the fluctuations of $\sum_{i\le n} W_i^{(n)}$ around $n\Ev[W^{(n)}]$ are of order $n^{1/2}$ whenever $\tau>3$ and are $n^{1/(\tau-1)}$ when $\tau\in(2,3)$, and that for any two positive numbers $a,b$, $|(1\wedge a)-(1\wedge b)\le 1\wedge |a-b|$.
Hyperbolic random graphs satisfy this assumption with the choices $\epsilon(n)=n^{-1}$ and $I_\Delta(n)=[n^{-\beta(\al)}, n^{\beta(\al)}), I_w(n)=[1,n^{\beta(\al)}]$ for some $\beta(\al)$, that we show in Section \ref{s:hyperbolic}.
	\begin{claim}\label{claim:hg}
		Suppose $g_n$ satisfies Assumption \ref{assu:GIRGgen}, with some $\al, \gamma$ and Assumption \ref{assu:extendable}. Then with the \emph{same} $\al, \gamma$ and for some $\underline c_1, \overline C_1$ with $0< \underline c_1\le c_1\le C_1\le \overline C_1<\infty$, the limit $h$ in \eqref{eq:h-intro} satisfies 
		\be  \underline c_1 \underline g(\Delta, w_u, w_v)\le h(\Delta, w_u, w_v)\le \overline C_1 \overline g(\Delta, w_u, w_v).  
		\ee
	\end{claim}
	\begin{proof}
		Since $I
		_\Delta(n)\to (0, \infty), I_w(n)\to [1,\infty)$, using \eqref{GIRGgeneral}, rearranging \eqref{eq:h-intro} yields that
		\be \ba 
		h(\Delta,w_u,w_v)&\le (1-\epsilon(n))^{-1} g_n(x,x+\Delta/n^{1/d},(W_i^{(n)})_{i\in[n]}) 
		\leq (1-\epsilon(n))^{-1}C_1\overline{g}(\Delta,w_u,w_v),\\
		h(\Delta,w_u,w_v) &\geq (1+ \epsilon(n))^{-1}c_1\underline{g}(\Delta,w_u,w_v),
		\ea\ee
		for all \cb sufficiently \cz large $n$, thus the result follows by adjusting the constants $c_1, C_1$.
	\end{proof}
	To be able to define the distributional limit of weighted distances in GIRG, we need to define an infinite random graph, that we think of as the \emph{extension of $\mathrm{GIRG}_{W,L}(n)$ to $\R^d$}. This model is a generalisation of \cite{DepWut13}, and it can be shown that the continuum percolation model in \cite{FouMul18}, the (local weak) limit of hyperbolic random graphs, can be transformed to this model with specific parameters that we identify in Section \ref{s:hyperbolic} below.
	\begin{definition}[Extended GIRG]\label{def:EGIRG} Let $h:\R^d\times \R_+\times\R_+\to [0,1]$ be a function, $W\ge1, L\ge 0$ random variables and $\la>0$. We define the infinite random graph model $\Elambda$ as follows. 
		Let $\CV_\lambda$ be a homogeneous Poisson Point Process (PPP) on $\R^d$ with intensity $\la$, forming the positions of vertices. For each $x\in \CV_\lambda$ draw a random weight $W_x$, \cb an i.i.d.\ non-negative rv with pdf $F_W$. Then, conditioned on $(x, W_x)_{x\in \CV_\lambda}$,  edges are present  independently with probability 
		\be\label{eq:EGIRG-prob}  \Pv(y \leftrightarrow z \text{ in } \Elambda \mid (x,W_x)_{x\in \CV_\la}):= h( y-z, W_{y}, W_z).  \ee
		Finally, we assign to each present edge $e$ an edge-length $L_e$, an i.i.d.\ copy of $L$\cz.
		We write $(\CV_\la, \CE_\la)$ for the vertex and edge set of the resulting graph, that we denote by $\mathrm{EGIRG}_{W,L}(\lambda)$. \end{definition}
	Note that in the extension it is necessary to use the limiting function $h$ instead of $g_n$, since $g_n$ is not defined outside $\mathcal X_d$. We also use the limiting weight distribution $W$ instead of $W^{(n)}$ for the weights of vertices.
	
	To be able to connect $\mathrm{GIRG}_{W, L}(n)$ to the extended infinite model $\Elambda$, we `blow-up' $\mathrm{GIRG}_{W, L}(n)$ on $(\mathcal X, \nu)=([-1/2, 1/2]^d, \mathcal{L}eb)$, by a factor $n^{1/d}$ so that the average density of vertices becomes $1$. This is the next model we define.
	\begin{definition}[Blown-up $\mathrm{GIRG}_{W, L}(n)$]\label{def:BGIRG} Let $W^{(n)}\ge 1, L\ge 0$, and $g_n$ be 
	from Definition \ref{def:GIRG}.
		Let $(\wit{\mathcal{X}}_d(n), \nu_n):=([- n^{1/d}/2, n^{1/d}/2 ]^d, \mathcal {L}eb)$. 
		Assign to each vertex \cb$i\in [n]$ \cz an i.i.d.\ position vector $x_i\in \wit{\mathcal{X}}_d(n)$ sampled from the measure $\mathcal{L}eb$, and a vertex-weight $W_i^{(n)}$, \cb an i.i.d.\ copy of  \cb $W^{(n)}$. Then, conditioned on $(x_i, W_i^{(n)})_{i\in [n]}$  edges are present  independently and for any $u,v\in [n]$,
		\be\label{eq:gn-intro2}\ba
		\Pv(y \leftrightarrow z \text{ in } \BGIRG \mid (x_i, W_i^{(n)})_{i\in [n]})&:=g_n\big(x_u/n^{1/d},x_v/n^{1/d},(W_i^{(n)})_{i\in [n]}\big)\\
		&=:g_n^{\mathrm B}\big(x_u,x_v,(W_i^{(n)})_{i\in [n]}\big)\ea
		\ee
		Finally, assign to each present edge $e$ an edge-length $L_e$, an i.i.d.\ copy of $L$.\cz
		We write $\CV_{B}(n), \CE_{B}(n)$ for the vertex and edge-set of the resulting graphs, that we  denote by $\mathrm{BGIRG}_{W,L}(n)$. 	\end{definition}
		Naturally, $\mathrm{BGIRG}_{W,L}(n)\  {\buildrel d \over =}\  \mathrm{GIRG}_{W,L}(n)$ since it can be obtained as a deterministic transformation of the model defined in Definition \ref{def:GIRG}, namely, mapping every vertex $x_u\in [-1/2, 1/2]^d$ to $n^{1/d} x_u$ to obtain a vertex in $\BGIRG$. The advantage of working with the blown-up version is that it naturally yields a limiting infinite graph on $\R^d$, which is $\Eone$.
By \eqref{eq:gn-intro2} and \eqref{GIRGgeneral}, Assumption \ref{assu:GIRGgen} for $\BGIRG$ reduces to 
	\be\label{assu:BGIRG-prob} c_1 \underline g(\| x_u-x_v\|, W_u^{(n)}, W_v^{(n)}) \le g_n^{\mathrm{B}}(x_u, x_v, (W_i^{(n)})_{i\in [n]}) \le C_1 \overline g(\| x_u-x_v\|, W_u^{(n)}, W_v^{(n)}). \ee
	Next we define weighted and graph distance, and introduce notation for corresponding metric balls and their boundaries, and then define the \emph{explosion time} of the origin in $\mathrm{EGIRG}_{W,L}(\la)$.
	\begin{definition}[Distances and metric balls]\label{def:distances}
		For a path $\pi$, we define its length and $L$-length as $|\pi|=\sum_{e\in\pi}1$, 
		and $|\pi|_L:=\sum_{e\in\pi}L_e$, respectively. Let $\Omega_{u,v}:=\left\{\pi: \; \pi \text{ is a path from $u$ to $v$}\right\}$. For two sets $A, B\subseteq V$, we define the $L$-distance and the graph distance as
		\be\label{eq:distance-def}
		d_L\left(A, B\right):=\inf_{u\in A}\inf_{v\in B} \inf_{\pi \in \Omega_{u,v}}|\pi|_L, \qquad d_G\left(A,B\right):=\inf_{u\in A}\inf_{v\in B} \inf_{\pi \in \Omega_{u,v}}|\pi|.
		\ee
		with $d_q(A,B)=0$ when $A \cap B\neq \emptyset$ and $d_q(A,B)=\infty$ when there is no path from $A$ to $B$ for $q\in
		\{L, G\}$. For Euclidean distance, graph distance and $L$-distance, respectively, let
		\begin{equation}\label{metricballs}
		\begin{aligned}
		B^2\left(u,r\right):=&\left\{v\in V:\; \|x_u-x_v\|\leq r\right\},\\
		B^G\left(u,r\right):=&\left\{v\in V:\; d_G\left(u,v\right)\leq r\right\},\\
		B^L\left(u,r\right):=&\left\{v\in V:\; d_L\left(u,v\right)\leq r\right\},\\
		\end{aligned}
		\end{equation}
		denote the balls of radius r with respect to these distances, where $x_u,x_v$ denote the possibly random positions of vertices $u$ and $v$, respectively. For an integer $k\geq1$ we write \cb$\partial B^G\left(u,k\right):=B^G\left(u,k\right)\backslash B^G\left(u,k-1\right)$ \cz as the set of vertices that are at graph distance $k$ from $u$, and $ B^{L,G}\left(u,k\right)$ and \cb $\partial B^{L,G}\left(u,k\right)$ \cz for those vertices $v$ where the path realising $d_L(u,v)$ uses at most and precisely $k$ edges, respectively. We add a respective subscript $\la$ and $n$ to  these quantities when the underlying graph is $\Elambda$ and $\BGIRG$, defined in Definitions \ref{def:EGIRG}, \ref{def:BGIRG}, respectively.
	\end{definition} 
\begin{definition}[Explosion time]\label{def:explosiontime}
		Consider $\mathrm{EGIRG}_{W,L}(\la)$ as in Definition \ref{def:EGIRG}, and let $v\in \CV_\la$.  Let $\tau_\la^{\mathrm{E}}(v,k):=\inf\{t: |B^L(v,t)|=k \text{ in } \mathrm{EGIRG}_{W,L}(\la)\}$. Let the \emph{explosion time} of vertex $v$ be defined as the (possibly infinite) almost sure limit:
		\be Y_\la^{\mathrm{E}}(v):=\lim_{k\to \infty} \tau_\la^{\mathrm{E}}(v,k).\ee 
		The explosion time of the origin, $Y_\la^{\mathrm{E}}(0)$, is defined analogously when $\CV_\la$ is conditioned on having a vertex at $0\in \R^d$. We call $\Elambda$ \emph{explosive} if $\Pv(Y_\la^{\mathrm{E}}(0)< \infty)>0$. \cb We call any infinite path $\pi$ with $|\pi|_L<\infty$ \emph{an explosive path}\cz. 

	\end{definition}
\cb	 $B^L(v,t)$, defined in \eqref{metricballs}, stands for the vertices that are available within $L$-distance $t$. Hence, $\tau_\lambda^E(v,k)$ is the smallest radius $t$, in terms of the $L$-distance, such that the metric ball $B^L(v,t)$ around vertex $v\in\mathcal{V}_\lambda$ contains $k$ vertices\footnote{When considering $t$ as time, this is the time that $B^L(v, \cdot)$ ``swallows" the $k$th vertex as one increases $t$}. Then, $Y^E_\lambda(v)$ is the almost sure limit of $\tau_\lambda^E(v,k)$ as $k$ tends to infinity. In most common scenarios considering infinite graphs, this limit is infinite. $Y_\lambda^{E}(v)$ becomes finite when it is possible to reach infinitely many vertices within finite $L$-distance from $v$. In particular, if there is an infinite path of vertices $\pi=(\pi_0=v, \pi_1, \pi_2\dots)$ with total $L$-length $|\pi|_L<\infty$, then $Y_\lambda^{E}(v)$ is finite. In reverse, if the graph is locally finite, that is, all degrees are less than infinity, then one can show (e.g., by arguing by contradiction, as done below in \eqref{eq:intersect}) that $Y_\la^{\mathrm{E}}(v)<\infty$ implies the existence of a path $\pi$ emanating from $v$ with finite total $L$-length.
In particular, one can define
\be \label{eq:piopt-1} \cb\pi_{\mathrm{opt}}^1(v):=\arginf_{\pi: \pi_0=u, |\pi|=\infty}\{|\pi|_L\}\ee
 be the `shortest' explosive path to infinity from $u$ in $\mathrm{EGIRG}_{W,L}(1)$, that is, the path that realises the shortest $L$-length of all paths with infinitely many edges from $u$. 
 In case this path is not unique (which may happen if $L$ is not absolutely continuous) then any path realising the infimum can be chosen. It can generally be shown that
\be\label{eq:piopt} |\pi_{\mathrm{opt}}^1(u)|_L=Y_{1}^{\mathrm{E}}(u).\ee	 	 
	 This phenomenon is called explosion, and it was first observed in branching processes. See Harris \cite{Harris02} for one of the first conditions derived for a continuous-time branching processes to be explosive.\cz
	
	Our first result, analogous to \cite[Theorem 1.1]{HofKom2017} states that $\Elambda$  is not  explosive when $\Ev[W^2]<\infty$, and characterise the set of edge-length distributions $L$ that give explosivity when $\tau\in(2,3)$ in Assumption \ref{assu:weight}. We denote the distribution function of $L$ by $F_L$.
	\begin{theorem}\label{thm:exp-charact}
		Consider $\Elambda$ in Definition \ref{def:EGIRG}. Then, \cb $\Elambda$ is not explosive whenever $\Ev[W^2]<\infty$. That is, $Y_\lambda^{\mathrm{E}}(0)=\infty$ almost surely when $\Ev[W^2]<\infty$. \cz
		Let
		\be\label{Lassumption}
		\mathbf{I}(L):=\int_1^\infty F_L^{\left(-1\right)}\left(\exp\left\{-\mathrm{e}^y\right\}\right)\mathrm{d}y.
		\ee
		When  $W$ satisfies \eqref{powerlaw} with $\tau\in(2,3)$, then $\Elambda$ is explosive \emph{if and only if} $\mathbf{I}(L)<\infty$.
	\end{theorem}
	We are ready to state the two main theorems. By Bringmann \emph{et al}.\ in \cite[Theorem 2.2]{BriKeuLen17}, whp there is a unique linear-sized giant component $\mathcal{C}_{\max}$ in $\mathrm{GIRG}_{W,L}(n)$ when $W$ satisfies \eqref{powerlaw} with $\tau\in(2,3)$.
	\begin{theorem}[Distances in GIRG with explosive edge-lengths]\label{Th:GIRGexplosive}
		Consider $\text{GIRG}_{W,L}(n)$, satisfying Assumptions \ref{assu:GIRGgen}, \ref{assu:weight}, \ref{assu:extendable}  with parameters $\tau\in\left(2,3\right), \al>1$, $d\geq 1$, \cb and assume that $\mathbf{I}(L)<\infty$. Let $v_n^1, v_n^2\in \mathcal{C}_{\max}$ be two vertices chosen uniformly at random\cz. Then, as $n$ tends to infinity,
		\be\label{convergenceGIRG}
		d_L\left(v_n^1, v_n^2\right)\xrightarrow{d} Y^{(1)}+Y^{(2)},
		\ee
		where $Y^{(1)},Y^{(2)}$ are i.i.d.\ copies of  $Y_1^{\mathrm{E}}(0)$ from Definition \ref{def:explosiontime} conditioned to be finite. 
	\end{theorem}
	\begin{corollary}\label{corr:tight}If all assumptions in Theorem \ref{Th:GIRGexplosive} hold except Assumption \ref{assu:extendable},  then $(d_L(v_n^1,v_n^2))_{n\ge 1}$ forms a tight sequence of random variables. 
	\end{corollary}
	\begin{proof}[Proof of Corollary \ref{corr:tight} subject to Theorem \ref{Th:GIRGexplosive}]
		Consider a model $\underline{\mathrm{GIRG}}_{W, L}(n)$ where $c_1\underline g(\cdot)$ is used instead of $g_n(\cdot)$ on the same vertex positions and vertex-weights $(x_i, W_i^{(n)})_{i\in [n]}$. Under Assumption \ref{assu:GIRGgen}, a coupling can be constructed where the edge sets  $\mathcal{E}(\underline{\mathrm{GIRG}}_{W, L}(n)) \subseteq \mathcal{E}(\mathrm{GIRG}_{W, L}(n))$ a.s.\ under the coupling measure. 
		Note that $c_1\underline g(\cdot)$ satisfies Assumption \ref{assu:extendable}, and so \eqref{convergenceGIRG} is valid for $\underline{\mathrm{GIRG}}_{W, L}(n)$ and thus Theorem \ref{Th:GIRGexplosive} holds for $\underline d_L(v_n^1, v_n^2)$, the $L$-distance  in  $\underline{\mathrm{GIRG}}_{W, L}(n)$. The proof of tightness is finished by noting that $d_L(v_n^1, v_n^2) \le\underline d_L(v_n^1, v_n^2)$ a.s.\ under the coupling.
	\end{proof}
	A similar theorem holds for the SFP model. One direction of this theorem has already appeared in \cite[Theorem 1.7]{HofKom2017}, where it was conjectured that its other direction also holds. For a vertex $x\in \Z^d$ define $\tau^{\mathrm{S}}(x,k):=\inf\{t: |B^L(x,t)|=k \text{ in } \mathrm{SFP}_{W,L}\}$ and $Y^{\mathrm{S}}(x):=\lim_{k\to \infty} \tau^{\mathrm{S}}(x,k)$ as the explosion time of vertex $x$ in SFP.
	
	\begin{theorem}[Distances in SFP with explosive edge-lengths]\label{Th:SFPexplosive}
		Consider $\text{SFP}_{W,L}$, with $W$ satisfying \eqref{powerlaw} with parameter $\wit \tau>1$, $\wit \alpha>d$ and $\gamma=\wit \alpha\left(\wit \tau-1\right)/d\in\left(1,2\right)$ and edge-length \cb distribution $F_L$ \cz with $\mathbf{I}(L)<\infty$. Fix a unit vector $\underline{e}$. Then, as $m\to \infty$,
		\be\label{eq:dist-111}
		d_L\left(0,\lfloor m\underline{e}\rfloor\right) -\left(Y^{\mathrm{S}}(0)+Y^{\mathrm{S}}(\lfloor m\underline{e}\rfloor)\right)\xrightarrow{\ \mathbb{P}\ }0,
		\ee
		and
		\be\label{eq:dist-indep-111}
		\left(Y^{\mathrm{S}}(0),Y^{\mathrm{S}}(\lfloor m\underline{e}\rfloor)\right)\xrightarrow{d}\left(Y_{(1)}^{\mathrm{S}},Y^{\mathrm{S}}_{(2)}\right),
		\ee
		where $Y_{(1)}^{\mathrm{S}},Y^{\mathrm{S}}_{(2)}$ are two \emph{independent} copies of the explosion time of the origin $Y^{\mathrm{S}}(0)$ in $\mathrm{SFP}_{W,L}$.
	\end{theorem}
	
	\subsection{Discussion and open problems}\label{sec:discussion}
	\cb Understanding the behaviour of distances and weighted distances on spatial network models is a problem that is still widely open, when the graph has a power-law degree distribution\cz. 
	Graph distances are somewhat better understood, at least in the \emph{infinite variance degree} regime, where a giant/infinite component exists whp and typical distances are doubly-logarithmic. For hyperbolic random graphs (HRG), typical distances and the diameter were recently studied  in \cite{AbdBodFou16} and \cite{MulSta17}, respectively, and in GIRGs in \cite{BriKeuLen17}. 
		For scale-free percolation and its continuum-space analogue, typical distances were studied in \cite{DeijHofHoog2013, HofKom2017} and \cite{DeHaWu15}. In the first two papers, doubly-logarithmic distances were established in the infinite variance degree regime, while lower bounds on graph distances were given in other regimes.

Typical distances in the \emph{finite variance regime} remain open - even the existence of a giant component in this case is open for GIRGs, and was studied for HRG in \cite{BodFouMul15, FouMul18}. We conjecture that this result would carry through for GIRGs as well, i.e. a unique linear sized giant component exists when $\tau>3$ and the edge-density is sufficiently high. An indication for this is that the limiting continuum percolation model identified in \cite{FouMul18} is a special case of the limit of GIRGs, as we show in Section \ref{s:hyperbolic}. 
For scale-free percolation (SFP), the phase transition for the existence of a unique infinite component for sufficiently high edge-densities was shown in \cite{DeijHofHoog2013} and for the continuum analogue in \cite{DepWut13}. In  \cite{DeHaWu15}, it is shown that distances grow linearly with the norm of vertices involved, when $\wit \alpha>2d$ and degrees have finite variance. The $\wit \al\in(d,2d)$ case remains open, and they are believed to grow poly-logarithmically with an exponent depending on $\wit\al, d, \tau$ based on the analogous results about \emph{long-range percolation} \cite{Bisk04, BisLin17}.

Scale-free percolation is the only scale-free spatial model for which results on \emph{first passage percolation} are known \cite{HofKom2017}. In this paper we extend this knowledge by studying first passage percolation in the infinite variance regime on GIRG, HRG,  and SFP, when the edge-length distribution produces explosive branching processes. The non-explosive case was treated in \cite{HofKom2017} for SFP, and we expect that the results there would carry over GIRG and HRG as well, in an $n$-dependent form that is e.g.\ stated for the configuration model in \cite{AdrKom17}.

\cb Thus, we see the following \cz as main \emph{open questions} regarding (weighted) distances in spatial scale-free network models:  Is there universality in these models \emph{beyond} the infinite variance degree regime? More specifically, in the finite variance degree regime: Can we give a description of the interpolation between log-log and linear distances in SFP? Is it similar to long-range percolation? Can we observe the same phenomenon in GIRG and HRG, in terms of the long-rage parameter $\al$ and $T_H$,  respectively? For instance, do distances become linear when the long-range parameter is sufficiently high? Does the Spatial Preferred Attachment model behave similarly to these models?

		\subsubsection*{Comparison of GIRG and SFP}
	The GIRG on $(\mathcal X_d, \nu)=([0,1]^d, \mathcal{L}eb)$ uses $n$ vertices distributed uniformly at random. When blown up to $[-n^{1/d}/2, n^{1/d}/2]^d$, the density of the points becomes unit, and, in case the connection probabilities behave somewhat regularly, the model can be extended to $\R^d$ as we did by introducing the $\Elambda$ model. \cb The main difference between GIRG and SFP is thus that their vertex set differs\cz.  In the SFP the vertex set is $\mathbb{Z}^d$, whereas it is a PPP on $\R^d$ in $\Elambda$. We conjecture that the local weak limit of $\mathrm{GIRG}_{W,L}(n)$ satisfying Assumption \ref{assu:extendable} exists and equals $\Eone$. 
	Beyond this, assuming that GIRG satisfies Assumption \ref{assu:GIRGgen}, the two models are differently parametrised. It can be easily shown that the connection probability of the SFP model in \eqref{SFP} satisfies the bounds in \eqref{GIRG}, when setting $\alpha=\wt{\alpha}/d$, $\tau-1=(\wt{\tau}-1)\alpha$ and implying that $\gamma=\wt{\alpha}(\wt{\tau}-1)/d=\tau-1$. It follows that where $\gamma$ is the relevant parameter in the SFP model, $\tau-1$ is the relevant parameter in the GIRG model. The relation $\alpha=\wt{\alpha}/d$ directly shows why we work with $\alpha>1$ for GIRGs: For the SFP $\wt{\alpha}>d$ ensures that  degrees are finite a.s., while $\wit \alpha<d$ implies a.s.\ infinite degrees \cite[Theorem 2.1]{DeijHofHoog2013}. This translates to $\alpha>1$ for the GIRG model.	
	
	\subsection{Overview of the proof}\label{sec:proofstructure}	We devote this section to discussing the strategy of the proof of Theorem \ref{Th:GIRGexplosive}.
	We start by mapping the underlying space, $\mathcal{X}_d:=[-1/2,1/2]^d$, and the vertex locations to $\wit{\mathcal{X}}_d:=[-n^{1/d}/2,n^{1/d}/2]^d$ using the blow-up method described in Definition \ref{def:BGIRG}. So, instead of  $\mathrm{GIRG}_{W, L}(n)$, we work with the equivalent $\mathrm{BGIRG}_{W, L}(n)$ that has the advantage that the vertex-density stays constant as $n$ increases. Under Assumptions \ref{assu:extendable}, we expect that this model is close to $\mathrm{EGIRG}_{W, L}(1)$ restricted to $\wit{\mathcal{X}}_d(n)$, and thus the shortest paths leaving a large neighbourhood of $u$ and $v$ have length that are distributed as the explosion time of those vertices in  $\mathrm{EGIRG}_{W, L}(1)$. Since $u,v$ are typically `far away', these explosion times become asymptotically independent. We show that this heuristics is indeed valid and also that these explosive rays can be \emph{connected to each other} within $\wit{\mathcal{X}}_d(n)$.
	
	The details are quite \cb tedious\cz.  We choose a parameter $\xi_n\downarrow 0$ such that $\mathrm{EGIRG}_{W, L}(1-\xi_n)$, (resp. $\mathrm{EGIRG}_{W, L}(1+\xi_n)$) has, for all large $n$, less (resp. more) vertices than $n$ in $\Xdn$, \cb  as proved in Claim \ref{claim:Poissonization}\cz. In Claim \ref{cl:vertex-containment} below we construct a coupling such that  the vertex sets within $\Xdn$ are subsets of each other, i.e. for all $n$ large enough,
	\be\label{eq:vertex-containment} \mathcal V_{1-\xi_n}\cap \Xdn \subseteq \mathcal V_B(n) \subseteq \mathcal V_{1+\xi_n}\cap \Xdn,\ee
	and for the edge-sets, $\CE_{\la_1} \subseteq \CE_{\la_2}$ whenever $\la_1\le \la_2$.
	By Def.\ \ref{def:EGIRG}, we use the \emph{limiting} edge probabilities $h$ and weights $W$ when determining the edges in $\mathrm{EGIRG}_{W, L}(\la)$, while  in $\mathrm{BGIRG}_{W, L}(n)$ we use $g_n$ and $W^{(n)}$. As a result, we cannot hope that a relation similar to \eqref{eq:vertex-containment} holds for the edge set of the three graphs as well. Nevertheless, we construct a coupling in Claim \ref{claim:edge-containment} below, using the error bound in \eqref{eq:h-intro} (that hold on the good event $\CL_n$ such that for any set of $k$ edges $E=(e_1, \dots, e_k)$ where $e_i$ connects vertices  in $\mathcal V_B(n)$  
	\be\label{eq:path-bound-1} \Pv\big(\cb \big(E \cap \mathcal E_B(n)\big) \ominus \big(E \cap \mathcal E_{1+\xi_n}\big) \neq \emptyset\cz \big)\to 0, \ee
	\cb where $\ominus$ denotes symmetric difference\cz, as long as the number of vertices involved is not too large.
	So, any constructed edge-set in one model is present in the other model as well as long as the corresponding \cb lhs \cz in \eqref{eq:path-bound-1} stays small. \cb Using the same edge-length $L_e$ on edges that are present in both models, we can relate the  $L$-distances in one model to the other.
	
	\cb\emph{Lower bound on $d_L(v_n^1,v_n^2)$}\cz. 
	Recall from Def.~\ref{def:distances} that we add subscript $\la$ and $n$ to the metric balls and their boundaries when the underlying model is $\Elambda$ and $\BGIRG$, respectively, and the coupling of the vertex sets in \eqref{eq:vertex-containment}.
	Suppose we can find \cb an increasing sequence \cz $k_n\to \infty$ such that the graph-distance balls $B^G_{1+\xi_n}\left(v_n^1,k_n\right) \cap B^G_{1+\xi_n}\left(v_n^2, {k_n}\right)=\emptyset$ and that both $B^G_{1+\xi_n}\left(v_n^1, {k_n}\right), B^G_{1+\xi_n}\left(v_n^2, {k_n}\right)\subset \wit{\mathcal X}_d(n).$ Then, any path connecting $v_n^1,v_n^2$  in $\Eplus$ must intersect \cb$\partial B^G_{1+\xi_n}\left(v_n^1, {k_n}\right)$ \cz and \cb$\partial B^G_{1+\xi_n}\left(v_n^2, {k_n}\right)$\cz, hence,\cb	\be\label{overviewlowerbound1}
	d_L^{1+\xi_n}\left(v_n^1,v_n^2\right)\geq  d_L^{1+\xi_n}\left(v_n^1,\partial B^G_{1+\xi_n}\left(v_n^1, {k_n}\right)\right)+d_L^{1+\xi_n}\left(v_n^2,\partial B^G_{1+\xi_n}\left(v_n^2, {k_n}\right)\right),
	\ee \cz
	Let $\cb A_{k}$ be the event that $\cb B^G_n(v_n^1, k)$ and $\cb B^G_{1+\xi_n}(v_n^1, k)$ as well as $\cb B^G_{1}(v_n^1, k)$ (the graph distance balls within graph distance $k$) coincide as graphs with vertex-weights (i.e. they have identical vertex and edge-sets and vertex-weights) \cb and that $B^G_{1+\xi_n}(v_n^1, k)$ and $B^G_{1+\xi_n}(v_n^2, k)$ are disjoint, that is,
	\be\ba\label{eq:ak}
	A_{k}:=\bigcap_{q\in\{1,2\}}\Big\{&\CV(B^G_n(v_n^q, k))=\CV(B^G_{1+\xi_n}(v_n^q, k))=\CV(B^G_{1}(v_n^q, k)),\\
	&\CE(B^G_n(v_n^q, k))=\CE(B^G_{1+\xi_n}(v_n^q, k))=\CE(B^G_{1}(v_n^q, k)),\\
	&\mathcal{W}(B^G_n(v_n^q, k))=\mathcal{W}(B^G_{1+\xi_n}(v_n^q, k))=\mathcal{W}(B^G_{1}(v_n^q, k)) \Big\}\\
	&\bigcap\big\{B^G_{1+\xi_n}(v_n^1, k)\intersect B^G_{1+\xi_n}(v_n^2, k)=\emptyset\big\},
	\ea\ee 
	where $\mathcal{W}(A)$ denotes the vertex-weights of vertices within a set $A$.	We show  in Proposition \ref{Prop:lowerbound} below \cz that on $A_{k_n}$, the shortest path connecting $v_n^q$ to \cb$\partial B^G_{1+\xi_n}\left(v_n^q, {k_n}\right)$ \cz is also the shortest in $\BGIRG$ and in $\Eone$ as well.
	Thus, on $A_{k_n}$, the lhs of \eqref{overviewlowerbound1} can be switched to  $d_L(v_n^1, v_n^2)$ ($L$-distance in $\BGIRG$), and $1+\xi_n$ can be changed to $1$ on the rhs. Then, it holds that, \cb
	\be\label{overviewlowerbound2}
	d_L\left(v_n^1,v_n^2\right)\geq  \ind_{A_{k_n}}\left(d_L^{1}\left(v_n^1,\partial B^G_1\left(v_n^1, {k_n}\right)\right)+d_L^{1}\left(v_n^2,\partial B^G_1\left(v_n^2, {k_n}\right)\right)\right),
	\ee\cz
	\cb To find $k_n$, for $q \in \{1,2\}$, we bound the maximum displacement within  $B^G_{1+\xi_n}\left(v_n^q,k\right)$ in Proposition \ref{Prop:BerBRW}\cz. Namely, we couple the exploration of the neighbourhoods of $v_n^1,v_n^2$ to two  dominating branching random walks (BRW) on the vertices of $\mathrm{EGIRG}_{W, L}(1+\xi_n)$. That is, the vertex set of $B^G_{1+\xi_n}\left(v_n^q, {k_n}\right)$ is a subset of the location of individuals in generation at most $k_n$ in the BRWs.
	\cb Then, we analyse the spatial growth of the BRW. \cb This leads us to show in Proposition \ref{Prop:lowerbound}, that there are two spatially disjoint boxes $\mathrm{Box}(v_n^q), q\in \{1,2 \}$, and an increasing sequence $k_n\to \infty$ such that\cz
	\be\label{BRW-bound-1} B^G_{1+\xi_n}\left(v_n^q, {k_n}\right)\subset \cup_{i\le k_n} \mathcal G_{i}(v_n^q)\subset \mathrm{Box}(v_n^q), \ee
	where $\mathcal G_{i}(v_n^q)$ denotes the set of generation $i$ individuals in the BRW started from $v_n^q$, and the first containment follows by  the coupling between the BRWs and the exploration of the neighbourhoods, see also Figure \ref{Figure:lowerbound}.
	
	\begin{figure}
		\centering
		\includegraphics[width=8cm]{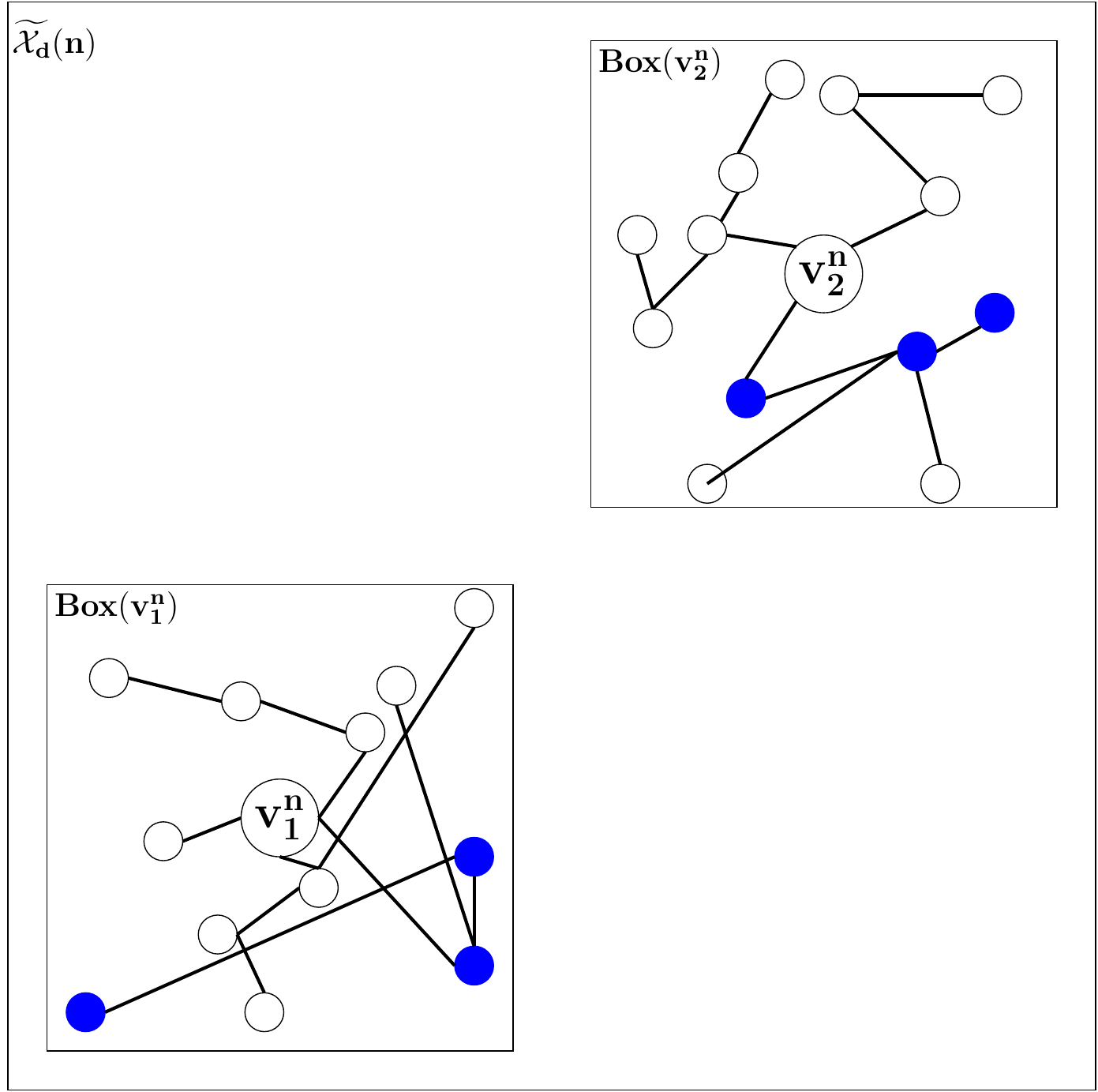}
		\caption{The exploration of the neighbourhood of $v_n^1,v_n^2$ up to graph distance $3$, contained in the boxes Box$(v_n^1)$, Box$(v_n^2)$, with the shortest paths in blue. The idea of the lower bound is: As long as the two boxes are disjoint, and each vertex within graph distance $k$ from $v_n^q$ stays within these boxes, the shortest path connecting $v_n^1, v_n^2$ is longer than the total length of the two shortest paths leading to vertices at graph distance $k$. We establish the bound by being able to choose $k=k_n$ that tends to infinity.}\label{Figure:lowerbound}
	\end{figure}
	
	This guarantees that the required conditions are satisfied \cb in \eqref{overviewlowerbound1}\cz. Since the $B^G_{1+\xi_n}\left(v_n^q, {k_n}\right)$ are defined within disjoint boxes, they are independent and thus the two terms on the rhs of \eqref{overviewlowerbound1} are independent. Further, \cb by possibly taking a new sequence $k_n$ that increases at a slower rate than the previous one\cz, \eqref{eq:path-bound-1} ensures that the event $A_{k_n}$ holds whp in \eqref{overviewlowerbound2}. 
	We finish the proof by showing that the variables on the rhs of \eqref{overviewlowerbound2} tend to two i.i.d. copies of the explosion time of the origin in $\mathrm{EGIRG}_{W,L}(1)$, a non-trivial part in itself.
	
	\emph{Upper bound on $d_L(v_n^1,v_n^2)$.}
	For the upper bound, we use the model $\Emin$. 
	Under the assumption that both $v_n^1, v_n^2$ are in the giant component of $\mathrm{BGIRG}_{W,L}(n)$, \cb we show in Claim \ref{claim:bg-to-eone} \cz that they are whp in the infinite component in the $\mathrm{EGIRG}_{W,L}(1)$. \cb Recall from Definition \ref{def:explosiontime} that a path $\pi$ is called explosive if $|\pi_L|<\infty$, and that on the event that the explosion time $Y^{\mathrm{E}}_1(v)<\infty$, there are explosive paths in the model.
Then, recall $ \cb\pi_{\mathrm{opt}}^1(v)$ from \eqref{eq:piopt-1}, the `shortest' explosive path to infinity from $v$ in $\mathrm{EGIRG}_{W,L}(1)$. \cb In case this path is not unique (which may happen if $L$ is not absolutely continuous) then any path realising the infimum can be chosen. The advantage of this path is that  $|\pi_{\mathrm{opt}}^1(v)|_L=Y_{1}^{\mathrm{E}}(v)$ (see \eqref{eq:piopt})
i.e. this path realises the explosion time of $v$. Thus, to obtain an upper bound on the $L$-distance between $v_n^1$ and $v_n^2$, our goal is to show that the initial segments of the paths  $\pi_{\mathrm{opt}}^1(v_n^1), \pi_{\mathrm{opt}}^1(v_n^2)$ are part of $\BGIRG$ and that we can  \emph{connect} these segments within $\BGIRG$ with a connecting path of length at most $\ve$, say, for arbitrarily small $\ve>0$. 
The first part, namely, that the initial segments of these paths are part of $\BGIRG$, holds whp since it is a consequence of the event $A_{k_n}$: any path in $\Eone$ that leaves $v_n^q$ and has at most $k_n$ edges is part of $\BGIRG$ as well on $A_{k_n}$. To be able to connect these segments, we need to find a vertex on $\pi_{\mathrm{opt}}^1(v_n^1)$ and $\pi_{\mathrm{opt}}^1(v_n^2)$, respectively, that has high enough weight. 

 \cz
  For a constant $K$, define now $\wt v(K)$ as the first vertex on $\cb\pi_{\mathrm{opt}}^1(v)$ with weight larger than $K$ and $T_K(v):=d_L^1(v, \wt v(K))$, where $d_L^1$ denotes the $L$-distance in
	$\Eone$. \cb Observe that $T_K(v)$ is the length of the segment between $v, \wt v(K)$ on the path $\pi_{\mathrm{opt}}^1(v)$, otherwise the optimality of $\pi_{\mathrm{opt}}^1(v)$ would be violated.
	\cb We show in Corollary \ref{lem:expgen} \cz that $\wt v(K)$ and $T_K(v)$ \emph{exist}: explosion is impossible in the subgraph of $\mathrm{EGIRG}_{W,L}(\la)$ restricted to vertices with weight $\le K$, irrespective of the value of $\lambda$.	
	Thus, any infinite path with finite total length must leave this subgraph.
	
We need to make sure that $\wit v_n^1(K)$ and $\wit v_n^2(K)$ are part of $\BGIRG$. For this we argue as follows:
By translation invariance, for any fixed $K<\infty$ the spatial distance $\| u-\wt u(K)\|$ is a copy of the random variable $\| 0-\wt 0(K)\|$, i.e. a proper random variable, \cb where $0$ denotes the origin of $\R^d$\cz, \cb and $\wit 0(K)$ denotes the first vertex with weight at least $K$ on the shortest path from $0$ to infinity in $\mathrm{EGIRG}_{W,L}(1)$, conditioned that $0$ is in the infinite component\cz. Thus, for fixed $K$, for $n$ sufficiently large, $ \wt v_n^q(K)\in \wit{\mathcal X}_d(n)$ will be satisfied with probability tending to $1$. 
	Further, \cb we show in Lemma \ref{lem:ank} \cz that for any fixed $K$, the event $\wit A_{n,K}$ that 
	the vertices and edges of $\cb\pi^1_{\text{opt}}(v_n^q)$ on  the section leading to $\wit v_n^q(K)$ are present in $\mathrm{BGIRG}_{W, L}(n)$ as well as in $\Eone$ happens whp. 
	By \eqref{eq:piopt}, \emph{jointly}, fixing the vertices $v_n^1, v_n^2$ and letting $K\to \infty$, 
	\be\label{eq:explosion-conv} \big(T_K(v_n^1), T_K(v_n^2)\big) \toas (Y^\mathrm{E}_1({v_n^1}),Y^\mathrm{E}_1({v_n^2})),  \ee
	the joint distribution of the explosion times of $v_n^1, v_n^2$ in $\mathrm{EGIRG}_{W,L}(1)$. Observe that $(T_K(v_n^1), T_K(v_n^2))$ is monotone increasing in $K$ that will be relevant later on.
	
	\cb Next, we show in Proposition \ref{Prop:upperbound} \cz the whp existence of a path connecting  $\wt{v}_n^1(K),\wt v_n^2(K)$ within $\wit{\mathcal X}_d(n)$ that has L-distance at most $\ve_K$ within $\BGIRG$. We do this using a subgraph created by \emph{weight-dependent percolation}. That is, we keep edges $e=(u,v)\in\CE_B(n)$ if and only if their edge-length $L_e \le \cb \mathrm{thr}(W_u^{(n)}, W_v^{(n)})$, \cb for a function $\mathrm{thr}$ of $W_u^{(n)}, W_v^{(n)}$, that satisfies $\mathrm{thr}(x,y)=F_L^{(-1)}(p(x,y))$, where we set $p(x,y)\geq\exp\{-(\log x)^\gamma-(\log y)^\gamma\},\gamma\in(0,1)$ It follows that we keep an edge $e=\{u,v\}$ with probability at least $\exp\{-(\log W_u^{(n)})^\gamma-(\log W_v^{(n)})^\gamma\}$\cz. Let us denote the remaining subgraph by $\cb G^{\mathrm{p}}$. \cb The advantage of choosing the percolation function like this is, as we \cb show in Claim \ref{claim:percolation}, \cz that $\cb G^{\mathrm{p}}$, in distribution, is again a BGIRG with a new weight distribution $W_{\mathrm{p}}^{(n)}$ and new edge probabilities $\cb g_n^{\mathrm{p}}$, that again satisfy Assumption \ref{assu:GIRGgen}, \ref{assu:weight} and \ref{assu:extendable}, and the new parameters $\cb \al^{\mathrm{p}}, \tau^{\mathrm{p}}$ are unchanged, i.e.  $\cb \al^{\mathrm{p}}=\al, \tau^{\mathrm{p}}=\tau$.
	
	Using this distributional identity, \cb we describe a (doubly exponentially growing) boxing structure around $\wt{v}_n^1(K),\wt v_n^2(K)$ in Lemma \ref{lemma:weightconnectioncenters} \cz that uses vertices and edges within $\cb G^{\mathrm{p}}$. Using this boxing structure, \cb we construct a connecting path \cz and control the weights of vertices along the connecting path between $\wt{v}_n^1(K),\wt v_n^2(K)$, as shown in Figures \ref{Fig:boxing} and \ref{Figure:overviewupperbound}.
	\begin{figure}[h]
		\centering
		\includegraphics[width=8cm]{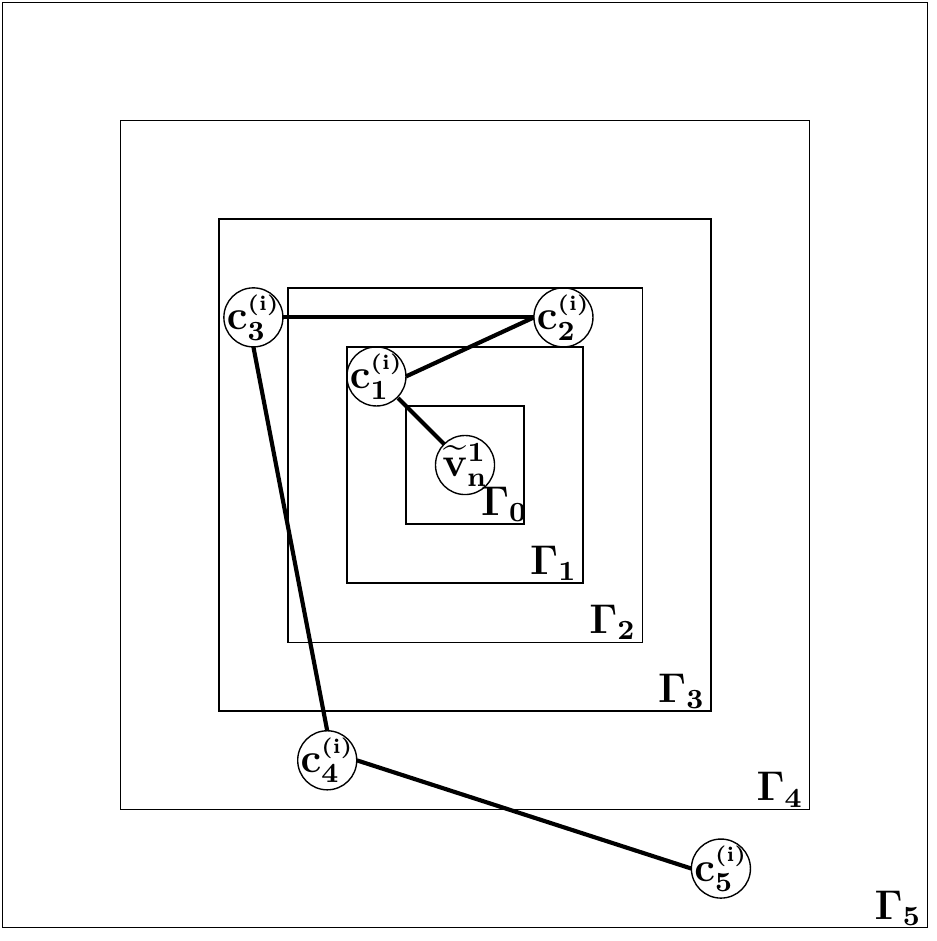}
		\caption{A schematic overview of the boxing structure around $\wt v^n_1$. The radius of the annuli increases doubly-exponentially with their index. We drew a path connecting centres of subboxes within consecutive annuli, but we did not draw the subboxes. In the extended model, this construction can be continued indefinitely, yielding an infinite path. This is the idea of proving the existence of an infinite component.}\label{Fig:boxing}
	\end{figure}

	\begin{figure}[h]
		\centering
		\includegraphics[width=8cm]{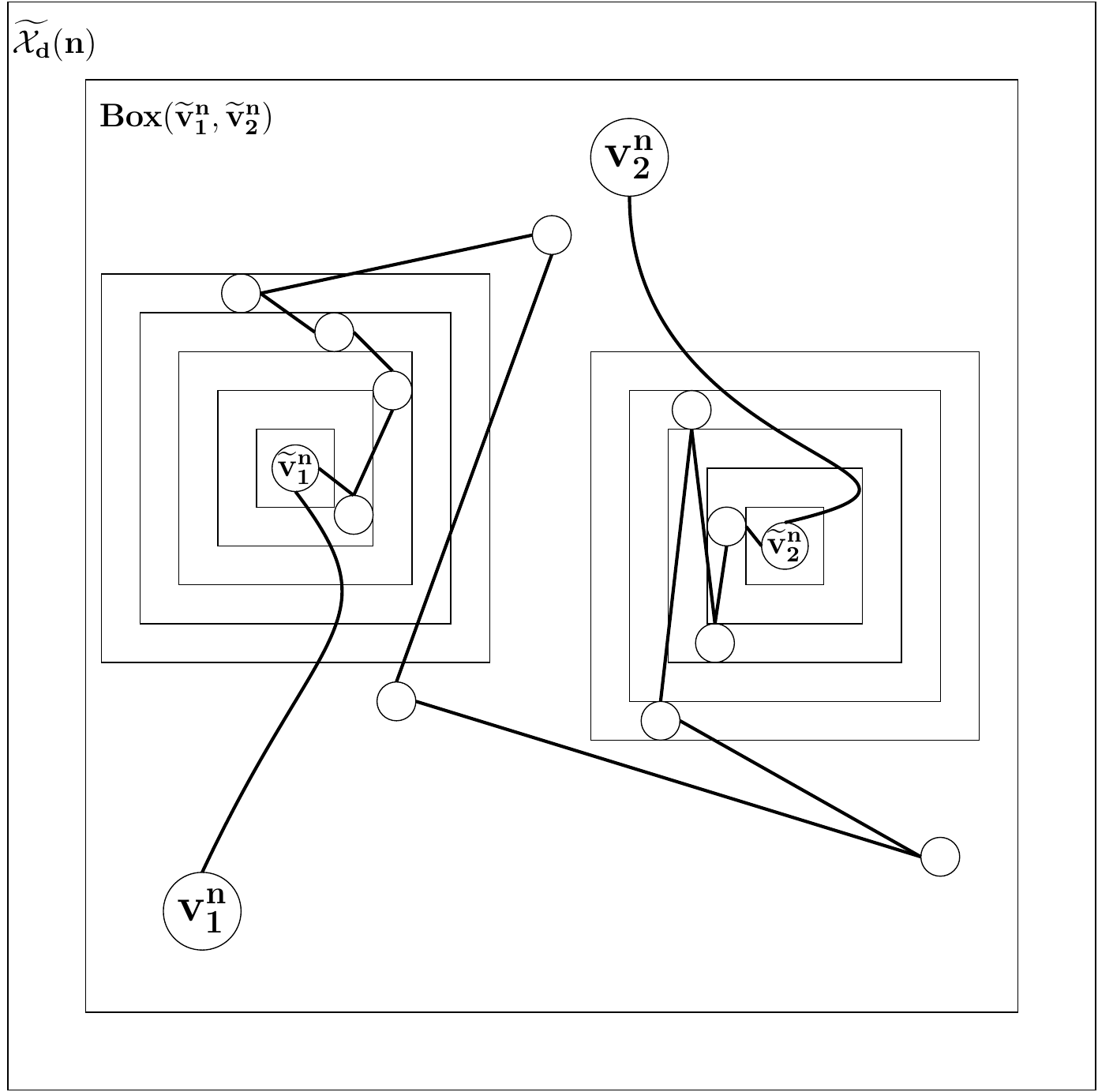}
		\caption{The idea of the proof of the upper bound. We connect $v_n^1$ with $v_n^2$ by following the two shortest explosive rays to two vertices with degree at least $K$, denoted by $\wt{v}_n^1(K),\wt{v}_n^2(K)$. Then we  connect these two vertices using a boxing structure, and following two paths via centres of subboxes. We control the total weight on the edges of the connecting segment between $\wt{v}_n^1(K),\wt{v}_n^2(K)$ by doing the boxing structure in a percolated graph, where only edges with small edge-weight are kept.}\label{Figure:overviewupperbound}
	\end{figure}
	
	 Since the percolation provides a deterministic  upper bound on the edge-lengths of \cb retained edges\cz, once we have bounds on the weights of the vertices, we have a deterministic bound $\ve_K$ on the length of the connecting path, where $\ve_K$ tends to zero with $K$. \cb The first part of the path, constructed by following the explosive path in $\Eone$ is part of $\BGIRG$  with probability that tends to $1$ as $n\to \infty$, while the connecting path using the percolation and the boxing method exists with probability that tends to $1$ as $K\to \infty$, as in Figure \ref{Figure:overviewupperbound}\cz. \cb Thus we arrive at the bound\cz 	\be\label{overviewupperbound1} d_L(v_n^1,v_n^2) \le  T_K(v_n^1)+ T_K(v_n^2) + \ve_K \le Y^\mathrm{E}_1({v_n^1})+Y^\mathrm{E}_1({v_n^2})+ \ve_K,\ee 
	 where we have  used the monotonicity of the limit in \eqref{eq:explosion-conv}.	From here we finish the proof by controlling the error probabilities by first choosing $K$ large enough, then $n$ large enough and showing that 
	\be\label{explosiontoiid} (Y^\mathrm{E}_1({v_n^1}),Y^\mathrm{E}_1({v_n^2}))\toindis (Y^{(1)},Y^{(2)}),\ee two i.i.d.\ copies of $Y^\mathrm{E}_{1}(0)$, a part that is non-trivial itself, similar to the method in the lower bound.
	
	\subsection{Structure of the paper}
	In Section \ref{sec:Blowuppoissonization}, we describe a coupling between $\BGIRG$ and the extended model $\Elambda$ and state two propositions capturing the upper and lower bound of Theorem \ref{Th:GIRGexplosive}. We provide the proof of Theorem \ref{Th:GIRGexplosive} in this section subject to these propositions. In Section \ref{sec:BRW}, \cb we introduce a branching random walk in a random environment\cz, and describe its behaviour when $\Ev[W^2]<\infty$ or $\Ev[W^2]=\infty$. We also prove Theorem \ref{thm:exp-charact} in this section. The BRW allows us to bound the growth of neighbourhoods in $\BGIRG$, \cb which we utilise in Section \ref{sec:maintheoremproof} \cz by proving the lower bound in Theorem \ref{Th:GIRGexplosive} (proof of Proposition \ref{Prop:lowerbound} below). We discuss weight-dependent percolation and a boxing method  in Section \ref{sec:boxinginfinitecomponent}, both preliminaries for the upper bound in Theorem \ref{Th:GIRGexplosive}. Section \ref{sec:bestexplosive} is devoted to the optimal explosive path in $\Eone$ and its presence in $\BGIRG$, and contains the proof of the  upper bound in Theorem \ref{Th:GIRGexplosive} (proof of Proposition \ref{Prop:upperbound} below)). We then use Section \ref{sec:extensionSFP} to describe the necessary adjustments  for our results to hold for scale-free percolation (proof of Theorem  \ref{Th:SFPexplosive}). 
	Finally, in Section \ref{s:hyperbolic}, we describe the connection between hyperbolic random graphs to GIRGs and show that Theorem \ref{Th:GIRGexplosive} is valid for hyperbolic random graphs as well. 	
	
	\section{Coupling of GIRGs and their extension to $\mathbb{R}^d$}\label{sec:Blowuppoissonization}
	In this section we state two main propositions for Theorem \ref{Th:GIRGexplosive} and describe a \emph{multiple-process coupling}.	
	To be able to show the convergence of distances in \eqref{convergenceGIRG} to the explosion time of $\Eone$ (Definition \ref{def:explosiontime}),
	we need to relate $\mathrm{GIRG}_{W,L}(n)$ to $\Eone$. Therefore, we first blew the model up to $\Xdn$, obtaining a unit vertex-density model $\BGIRG$, and then extended this model to $\R^d$, obtaining $\Elambda$, using PPPs with intensity $\la$ as vertex sets, allowing for vertex-densities other than $1$. 
	While in $\mathrm{BGIRG}_{W,L}(n)$ there are $n$ vertices, the number of vertices in $\mathrm{EGIRG}_{W,L}(\lambda)\cap \wt{\mathcal X}_d(n)$ is $\mathrm{Poi}(\lambda n)$. Thus the two models cannot be compared directly. To circumvent this issue, we slightly increase/decrease the vertex-densities and even the edge-densities to obtain models with resp.\ more/less vertices in $\wt{\mathcal X}_d(n)=[-n^{1/d}/2, n^{1/d}/2]$ than $n$. Recall Definitions \ref{def:BGIRG}, \ref{def:EGIRG}.
	\begin{definition}[Extended GIRG with increased edge-density]\label{def:Eupper}
		Let us define $\overline{\mathrm{EGIRG}}_{W,L}(\lambda)$ as follows: Use the same vertex positions and vertex-weights $(x, W_{x})_{x\in \CV_\la}$ as in the $\Elambda$ model. Conditioned on $(x, W_{x})_{x\in \CV_\la}$,  an edge between two vertices $y,z\in \CV_\lambda$ is present independently with probability
		\be
		\mathbb{P}\left(y\leftrightarrow z \text{ in } \overline{\mathrm{EGIRG}}_{W,L}(\lambda) \mid \left(x,W_{x}\right)_{x\in \CV_\la}\right)=\overline{C}_1\overline g(y-z,W_{y},W_{z}).
		\ee
		Each edge $e$ carries a random length $L_e$, i.i.d.\ from some \cb distribution $F_L$\cz. We denote the edge set by $\overline \CE_\la$. \end{definition}	
	An elementary claim is the following:
	\begin{claim}\label{claim:Poissonization} Let $V_\lambda^n \ {\buildrel d \over =}\ \mathrm{Poi}(\la n)$ be the number of vertices in $\CV_\lambda \cap\wt{\mathcal{X}}_d(n)$.
		Set  $\xi_n:=\sqrt{4\log n/n}$. Then there exists an $n_0\in\mathbb{N}$, such that for all $n\geq n_0$, $V^n_{1-\xi_n}\leq n\leq V^n_{1+\xi_n}$ \cb holds almost surely\cz.
\end{claim}
	
	\begin{proof}
		The distribution of $V^n_{\lambda}$ is $\Poi(n\lambda)$. Let  $I_{1+\xi_n}(1):=(1+\xi_n)-1-\log(1+\xi_n)\leq \xi_n^2/2$ be the rate function of a Poisson random variable with parameter $1+\xi_n$ at $1$. Then, using Chernoff bounds \cite[Theorem 2.19]{Hofbook},
		\be
		\P{V^n_{1+\xi_n}<n}=\exp\left\{-n I_{1+\xi_n}\left(1\right)\right\}\leq  \exp\{-n \xi_n^2/2\}=n^{-2},
		\ee
		which is summable in $n$. Following similar steps for the $1-\xi_n$ case, the claim follows using the Borel-Cantelli lemma on the sequence of events $V^n_{1-\xi_n}\leq n\leq V^n_{1+\xi_n}$.	\end{proof}
	
	\begin{claim}\label{cl:vertex-containment}
		There exists a coupling between $(\mathrm{EGIRG}_{W,L}(1\pm \xi_n))_{n\ge 3}$, $\Eone$, and the model $(\BGIRG)_{n\ge n_0}$, where $n_0$ is from Claim \ref{claim:Poissonization}, such that  for all $n\ge 3$,
		\be \label{eq:vertex-containment2} 
		\mathcal V_{1-\xi_n} \subseteq \mathcal V_1 \subseteq \mathcal V_{1+\xi_n},\quad \quad
		\mathcal E_{1-\xi_n}\subseteq \mathcal E_1 \subseteq \mathcal E_{1+\xi_n},
		\ee
		and, for all $n\ge n_0$,
		\be \ba \label{eq:vertex-containment3}(\mathcal V_{1-\xi_n}\cap \Xdn) &\subseteq \mathcal V_B(n) \subseteq (\mathcal V_{1+\xi_n}\cap \Xdn).\\
		\ea\ee
	\end{claim}
	\begin{proof}	
		We introduce a coupling between the PPPs  $(\CV_{1\pm\xi_n})_{n\ge 3}$ for all $n$ at once. Note that $\xi_n$ is decreasing in $n$ whenever $n\ge 3$.
		We use a PPP with intensity $1+\xi_3$ to create the vertex set $\CV_{1+\xi_3}$.  Then, for each vertex $v\in \CV_{1+\xi_3}$ we draw an i.i.d.\ uniform $[0,1]$ random variable that we denote by $U_v$. 
		Finally, for all $n\ge 3$ we set $v\in \CV_{1+\xi_n}$ if and only if $U_v\le (1+\xi_n)/(1+\xi_3)$, and $v\in \CV_{1-\xi_n}$ if and only if $U_v\le (1-\xi_n)/(1+\xi_3)$. Setting $\xi_n=0$ yields also a coupling to $\Eone$.
		The independent thinning of Poisson processes ensures that $\CV_{1\pm\xi_n}$ is a PPP with intensity $1\pm\xi_n$. To determine the edge set of $\Elambda$, note that the edge probabilities are determined by the same function $h$ for all $\la$. So, first we determine the presence of each possible edge between any two vertices in $\CV_{1+\xi_3}$. This edge is present in $\mathrm{EGIRG}_{W,L}(1\pm\xi_n)$ precisely when both end-vertices of the edge are present in $\CV_{1\pm\xi_n}$. This coupling guarantees that the containment in \eqref{eq:vertex-containment2} holds also for the edge-sets.
		
		To construct a coupling ensuring \eqref{eq:vertex-containment3}, we determine $\CV_B(n)$, given that $V_{1+\xi_n}^n=k_1 \ge n$ and $V_{1-\xi_n}^n=k_2\le n$. 
		Note that $\CV_B(n)$ is a set of $n$ i.i.d.\ points sampled from $\Xdn$, while, $\CV_{1\pm\xi_n}\cap \Xdn$ conditioned on having $k_1$ (resp. $k_2$) points, has the distribution of a set of $k_1$ (resp. $k_2$) many i.i.d.\ points sampled from $\Xdn$. Thus, the way to generate the location of the $n$ vertices in $\BGIRG$ so that \eqref{eq:vertex-containment2} holds is by taking the vertex set of  $\CV_{1-\xi_n}\cap \wit{\mathcal X}_d(n)$, and then adding a uniform subset of points of size $n-k_2$ from $(\CV_{1+\xi_n}\cap \wit{\mathcal X}_d(n))\setminus (\CV_{1-\xi_n}\cap \wit{\mathcal X}_d(n))$.
	\end{proof} 
	Now we couple the edge-sets of $\BGIRG$ to the extended models.  	
	The content of the following claim is the precise version of the inequality in \eqref{eq:path-bound-1}. 

	\begin{claim}\label{claim:edge-containment}
		There is a coupling of $\BGIRG$ and $\Elambda, \overline{\mathrm{EGIRG}}_{W,L}(\la)$ such that
		\be\label{manyedges} \CE_B(n)\subseteq \overline\CE_{1+\xi_n}, \quad \CE_\la \subseteq \overline\CE_{\la}, \ee
		where the latter holds for all $\la>0$.
		Let  $E=(e_1, \dots, e_k)\subseteq \overline \CE_{1+\xi_n}$ be a set of edges, 
		where $e_i$ connects vertices $v_{i_1}, v_{i_2}$ $\in \mathcal V_B(n)$ with locations and weights $(x_{i_1}, W_{i_1}), (x_{i_2}, W_{i_2})$ in $\Eplus$. Assume that the event $E_{\Delta,W}:=\{\forall i\le k: \|x_{i_1}-x_{i_2}\|\in I_\Delta(n), W_{i_1}, W_{i_2}\in I_w(n)\}$ holds and let  $n\ge n_0$ where $n_0$ is from Claim \ref{claim:Poissonization}. Then,
		\be\label{eq:path-bound-2} \Pv\Big( \cb\big(E \cap \mathcal E_B(n)\big) \ominus \big(E \cap \mathcal E_{1+\xi_n}\big) \neq \emptyset \cz \mid E \subseteq \overline \CE_{1+\xi_n}, E_{\Delta,W}\Big)
		\le \cb k(\epsilon(n)+2\epsilon_{\mathrm{TV}}(n)) \ee
		\cb where $\ominus$ denotes symmetric difference\cz. 
		The same bound is true for $\mathrm{EGIRG}_{W, L}(1-\xi_n)$ when we assume that $v_{i_1}, v_{i_2} \in \mathcal V_{1-\xi_n}$. 	\end{claim}
	\begin{proof}
		Recall the notations $g_n^{\mathrm{B}}, h, \overline g$ from \eqref{eq:gn-intro2}, from in Assumption \ref{assu:extendable} and from Claim \ref{claim:hg}, the bound from \eqref{assu:BGIRG-prob}\cb, and recall the weights $W^{(n)},W$ from Assumption \ref{assu:weight} and their total variation distance from \eqref{eps-tv}\cz.
		By the coupling described in Claim \ref{cl:vertex-containment}, the vertex set of $\BGIRG$ is a subset of $\CV_{1+\xi_n}$. For any possible edge connecting vertices in $\CV_B(n)$, we construct a coupling between the presence of this edge in $\BGIRG$ vs that of in $\CE_{1+\xi_n}, \overline\CE_{1+\xi_n}$ and $\CE_{1-\xi_n}$ (in the latter case, given that the two vertices are also part of $\mathcal V_{1-\xi_n}$). First, use the optimal coupling realising the total-variation distance between $W^{(n)}$ and $W$ for each endpoint of the edges under consideration. Then, for an edge $e$ connecting vertices $i_1, i_2$, 
		\be\label{eq:weightcouplingprob} \Pv(W^{(n)}_{i_1}\neq W_{i_1} \mbox{ or }  W^{(n)}_{i_2}\neq W_{i_2})=2\epsilon_{\mathrm{TV}}(n).\ee	
		Given that the two pairs of weights are equal, for a possible edge $e$ connecting vertices with  locations and weights $(x_{i_1}, W_{i_1}), (x_{i_2}, W_{i_2})$, draw $U_{e}$, an i.i.d. rv uniform in $[0,1]$. Include this edge respectively in $\CE_B(n), \CE_{1\pm\xi_n},\overline \CE_{1+\xi_n}$, if and only if 
		\be\label{singleedge} \ba e\in \CE_B(n)\quad  \Leftrightarrow \quad U_{e}&\le  g_n^{\mathrm B}(x_{i_1}, x_{i_2}, (W_{j})_{j\in[n]\setminus\{i_1,i_2\}})]=: g_n^{\mathrm B}(e), \\
		e\in \CE_{1\pm\xi_n}(n)\quad  \Leftrightarrow \quad U_{e}&\le h(x_{i_1}- x_{i_2}, W_{i_1}, W_{i_2})=:h(e)\\
		e\in \overline \CE_{1+\xi_n}(n)\quad  \Leftrightarrow \quad U_{e}&\le \overline C_1\overline g(x_{i_1}- x_{i_2}, W_{i_1}, W_{i_2})=:\overline C_1\overline g(e)
		\ea \ee
		These, per definitions of the models in Def.\ \ref{def:BGIRG}, \ref{def:EGIRG}, \ref{def:Eupper} give the right edge-probabilities. 
		By  \eqref{assu:BGIRG-prob} and Claim \ref{claim:hg}, both $h(e),  g^{\mathrm B}(e) \le \overline C_1\overline g(e)$. Since $\CV_B(n)\subset \CV_{1+\xi_n}$ and the vertex set is $\CV_\la$ for both $\Elambda, \Eupper$, this, together with \eqref{singleedge} ensures \eqref{manyedges}.
		By Assumption \ref{assu:extendable}, for any  $e$, 		\be\ba\label{eq:conditional}  \Pv(\cb(e \in \mathcal E_B(n)) \ominus (e \in \mathcal E_{1+\xi_n})\neq \emptyset \mid \cz e\in \overline \CE_{1+\xi_n}, E_{\Delta, W})& \le 2\epsilon_{\mathrm{TV}}(n)+\frac{|g_n^{\mathrm B}(e)-h(e)|}{\overline C_1\overline g(e)}\\
		&\le2\epsilon_{\mathrm{TV}}(n)+\epsilon(n),\ea\ee
by Assumption \ref{assu:extendable} and Claim \ref{claim:hg}.  A union bound finishes the proof of \eqref{eq:path-bound-2}.
	\end{proof}
	\emph{Key propositions for Theorem \ref{Th:GIRGexplosive}}.
	Next, we formulate three propositions. The first shows the existence of an infinite component in $\mathrm{EGIRG}_{W,L}(\lambda)$, and the other two formulate the upper and lower bound on $d_L(v_n^1, v_n^2)$ described in the overview in \eqref{overviewlowerbound2} and \eqref{overviewupperbound1}.  Recall the model-dependent notation for metric balls from Definition \ref{def:distances}.
	
	\begin{proposition}[Existence of $\mathcal{C}^\lambda_{\infty}$]\label{Prop:existinfcomp}
		Consider the $\Elambda$ model, as in Definition \ref{def:EGIRG}, with power-law parameter $\tau\in (2,3)$. Then, for all $\la>0$, there is a unique infinite component $\mathcal{C}^\lambda_\infty$.
	\end{proposition}
	We mention that there are some existing results also for $\tau\ge 3$. In this case, for a specific (illustrative) choice of the edge probabilities, the infinite model $\Elambda$ has an infinite component only when the edge-intensity (or equivalently, the vertex intensity) is sufficiently high, see \cite{DeijHofHoog2013, DepWut13}. For one dimension, hyperbolic random graphs form a special case when there is never a giant component when $\tau\ge 3$, see \cite{BodFouMul15}.
	 
Recall from before Theorem \ref{Th:GIRGexplosive} that $\CC_{\max}$ denotes the unique, linear-sized giant component of $\BGIRG$, that exists whp  by Bringmann \emph{et al}.\ \cite[Theorem 2.2]{BriKeuLen17}.
	\begin{proposition}[Lower bound on distances in the BGIRG model]\label{Prop:lowerbound}
		Let  $v_n^q,\;q=1,2,$ be two uniformly chosen vertices in $\CC_{\max}$ of $\BGIRG$, with parameters $\tau\in\left(2,3\right)$, $\alpha>1$ and $d\geq1$. For $q\in \{1,2\}$, let 
			\cb\be\label{v+} V^q_+(n,k):=d^{1+\xi_n}_L(v_n^q,\partial B_{1+\xi_n}^{G}(v_n^q, k)),\ee\cz
		be defined in $\Eplus$. 
		\cb Let $A_{k}$ be the event in \eqref{eq:ak}. Then, there exists a sequence $k_n$ \cz with $k_n\to \infty$, such that $\Pv(A_{k_n})\to 1$ as $n\to \infty$, and 		
		\be\label{lowerboundGIRG}
		d_L\left(v_n^1,v_n^2\right)\geq \ind_{A_{k_n}} \left(V^1_+(n,k_n)+V^2_+(n,k_n)\right),
		\ee
		and finally, \cb conditionally on $A_{k_n}$, $V^1_+(n,k_n),V^2_+(n,k_n)$ are independent\cz. 
			\end{proposition}

	In the same manner, we formulate a proposition for the upper bound. Here, we shall take a double limit approach, so the error bounds are somewhat more involved.
	\cb For a vertex $v\in \CV_1$, let $\wt v(K)$ be the first vertex with weight larger than $K$ on the  optimal explosive path $\pi_{\mathrm{opt}}^1(v)$ \cb from  \eqref{eq:piopt-1} in $\Eone$, let $\pi_{\mathrm{opt}}^1[u,\wt u(K)]$ denote the segment of $\pi_{\mathrm{opt}}^1(u)$ from $u$ to $\wt u(K)$, and let $T_K(u):=d_L^1(u, \wt u(K))=|\pi_{\mathrm{opt}}^1[u,\wt u(K)]|_L$\cz.
		\begin{lemma}[Best explosive path can be followed]\label{lem:ank}
		Let $v_n^1, v_n^2$ be uniformly chosen vertices in $\CC_{\max}$ of $\BGIRG$. Set
		\be\ba\label{eq:ank} \cb\wit A_{n,K}:= \{ &\cb\pi^1_{\mathrm{opt}}[v_n^q, \wit v_n^q(K)]\cz \text{ is present in } \BGIRG,  \text{ for }q\in\{1,2\}\}.\ea\ee
	Then  $1-\Pv(\wit A_{n,K}) \le f(n,K)$
		for some function $f(n,K)$ with $\lim_{n\to \infty}f(n,K)=0$ for each fixed $K$. 
	\end{lemma}
	Note that $\wit A_{n,K}$ implies the event that $ v_n^1, v_n^2$ are in $\CC_\infty^1$ of $\Eone$, otherwise $\cb\pi^1_{\mathrm{opt}}(v_n^q)$ is not defined.
	With the event $\wit A_{n,K}$ at hand, we are ready to state the key proposition for the upper bound.
	\begin{proposition}[Upper bound on distances in the GIRG model in the explosive case]\label{Prop:upperbound}
		Let $v_n^1, v_n^2$ be uniformly chosen vertices in $\CC_{\max}$ of $\BGIRG$.
		On $\wit A_{n,K}$ from \eqref{eq:ank}, for some $\eps_K$ that tends to $0$ as $K\to \infty$, the bound
		\be\label{upperboundGIRG}
		d_L\left(v_n^1,v_n^2\right)\leq T_K(v_n^1)+T_K(v_n^2)+\eps_K
		\ee
		holds with probability $1-\eta(K)$, for some function  $\eta(K)\to 0$ as $K\to \infty$.
	\end{proposition}
	We note that the short connection between the optimal explosive paths is the key part that is missing from  \cite[Conjecture 1.11]{HofKom2017}.
	We prove Proposition \ref{Prop:existinfcomp} in Section \ref{sec:boxinginfinitecomponent},
	Proposition \ref{Prop:lowerbound} in Section  \ref{sec:maintheoremproof} and Lemma \ref{lem:ank} and Proposition \ref{Prop:upperbound} in Section \ref{sec:bestexplosive}.
	\cb From Lemma \ref{lem:ank}\cz, Propositions \ref{Prop:lowerbound} and \ref{Prop:upperbound}, the proof of Theorem \ref{Th:GIRGexplosive} follows
	
	\begin{proof}[Proof of Theorem \ref{Th:GIRGexplosive} subject to Propositions \ref{Prop:lowerbound} and \ref{Prop:upperbound}.]
		Let $x$ be a continuity point of the distribution function of $Y^{(1)}+Y^{(2)}$. We start by bounding $\Pv(d_L(v_n^1, v_n^2) \le x)$ from above by using  Proposition \ref{Prop:lowerbound}.
		Let $V_1^q(n,k):=d^1_L(v_n^q,\cb\partial B_1^{G}(v_n^q, k)\cz)$, similar to \eqref{v+}.
		Due to the coupling developed in Claim \ref{cl:vertex-containment}, given that $v_n^q\in \CV_1$,
		$\CE_1 \subseteq \CE_{1+\xi_n}$ holds. Thus $V^q_+(n,k_n)\le  V^q_1(n,k_n)$.  Let  $\cb\pi_\star^q(k)$ be the shortest path connecting $v_n^q$ to \cb$\partial B_{1+\xi_n}^{G}(v_n^q, k)$\cz. Since on $A_{k_n}$, $\cb\pi_\star^q(k_n)\subseteq B_1^{G}(v_n^q, k_n)$, this means that \cb conditioned on $A_{k_n}$\cz,
		\be\label{eq:v1=v+} (V_1^1(n,k_n),V^2_1(n,k_n)) = (V^1_+(n,k_n),V^2_+(n,k_n))\ee
		Further, $\cb\pi_\star^q(k)\subseteq \BGIRG$ and $B_n^{G}(v_n^1, k)\cap B_n^{G}(v_n^2, k)=\emptyset$ implies that the shortest path between $v_n^1, v_n^2$ in $\BGIRG$ is at least as long as the total length of $\cb\pi_\star^1(k)$ and $\cb\pi_\star^2(k)$. Thus, \cb conditioned on $A_{k_n}$\cz,
		\be\label{eq:lower-distr} \ba \Pv(d_L(v_n^1, v_n^2) \le x) &\le \Pv\left(\ind_{A_{k_n}} \left(V^1_+(n,k_n)+V^2_+(n,k_n) \right)\le x\right)\\
		&=\Pv\left(\ind_{A_{k_n}} \left(V_1^1(n,k_n)+V^2_1(n,k_n)\right) \le x\right)\\
		&\le 1-\Pv(A_{k_n}^c) + \Pv(V_1^1(n,k_n)+V^2_1(n,k_n) \le x), 
		\ea \ee 
		where $1-\Pv(A_{k_n}^c)\to 0$ by Proposition \ref{Prop:lowerbound}. Next we show that the joint distribution of these two variables tend to $(Y^{(1)}, Y^{(2)})$ as $n\to \infty$.
		The model $\Eone$ is translation invariant, thus, marginally,
		\be\label{eq:xkn} V_q^1(n,k_n)\ {\buildrel d \over =}\ d_L^1(0, \cb\partial B_1^{G}(0, k_n)\cz)=:X_{k_n},\ee
		for $q\in \{1,2\}$, where we condition on $0\in \CV_1$.
		By \eqref{eq:v1=v+} and the independence of $V_1^+(n, k_n)$ and $V_2^+(n, k_n)$ that both hold on $A_{k_n}$, $V_1^1(n, k_n)$ and $V_2^1(n, \cb k_n)$ are also (conditionally) independent of each other.  Thus, the distributional identity
		\be\label{eq:v1xkn} \big(V_1^1(n,k_n), V^2_1(n,k_n)\big)\ {\buildrel d \over =} \ \big(X_{k_n}^{(1)}, X_{k_n}^{(2)}\big)\ee
		holds on $A_{k_n}$, where  $X_{k_n}^{(q)}$ are i.i.d.\ copies of $X_{k_n}$.
		Finally, we show that 
		\be\label{eq:xkn-toexp} X_{k_n}=d_L^1(0, \cb\partial B_1^{G}(0, k_n)\cz)\toas Y_1^E(0),\ee the (possibly infinite) explosion time\footnote{The definition $\lim_{k\to \infty}X_{k}$ is not the standard definition of the explosion time given in Definition \ref{def:explosiontime} so we have to show that the two limits are the same.} of $0$ in $\Eone$.  This is actually a general statement that holds in any locally finite graph: the weighted distance to the boundary of the graph-distance ball of radius $k$ tends to the explosion time, as the radius tends to infinity.

		Recall that $Y_1^E(0)=\lim_{k\to \infty}\tau_1^E(0,k)$ from Definition \ref{def:explosiontime}, the distance to the $k$-th closest vertex. Let  $\tau(k):=\tau_1^E(0,k)$. To show that $Y_1^E(0)=\lim_{k\to \infty} d_L^1(0, \cb\partial B_1^{G}(0, k))$, we show that for any $t>0$, the events $\{Y_1^E(0)\le t\}=\{\lim_{k\to \infty} \tau(k)\le t \}$ and $\{\lim_{k\to \infty} d_L^1(0, \cb\partial B_1^{G}(0, k)) \le t \}=\{\lim_{k\to \infty} X_k\le t\}$ coincide. 
		\cb  To show this, we first observe that
		 \[ \{\lim_{k\to \infty} \tau(k) \le t\}=\{|B^L_1(0,t)|=\infty\}.\] Indeed, when $|B^L_1(0,t)|=\infty$, then $\tau(k)$, the $L$-distance of  $k$th closest vertex to $0$ is also closer than $t$, hence $\tau(k)\le t$ holds. Reversely, assuming $|B^L_1(0,t)|=:K<\infty$ implies that $\tau(K+1)> t$.   
		 		
		Next we show that $\{|B^L_1(0,t)|=\infty\}\subseteq \{\lim_{k\to \infty} d_L^1(0, \partial B_1^{G}(0, k)) \le t \}$.
		Since all the degrees are finite, if there are infinitely many vertices within $L$-distance $t$, then there must be also a vertex $v$ at graph distance $k$ that is within $L$-distance $t$, (simply since one cannot squeeze in infinitely many vertices within graph distance $k-1$). Hence, the distance between $0$ and $\partial B^{L,G}_1(0,k)$ is at most $t$. This holds for all $k\ge 0$,  hence,\cz
				\be\label{eq:intersect} \{|B^L_1(0,t)|=\infty\}\subseteq\bigcap_{k\in \N} \{ \exists\, v\in \cb\partial B^{G}_1(0,k)\cz, d_L(0,v)\le t\}=\bigcap_{k\in \N}\{ X_k\le t\}=\{\lim_{k\to \infty} X_k\le t\},\ee
		where the \cb last equality \cz holds since $X_k$ is non-decreasing in $k$. \cb For the reverse direction, 
		observe that $|B^L_1(0,t)|=\sum_{k=1}^\infty |\cb\partial B^{G}_1(0,k) \cap B^L_1(0,t)|$, and if $|B^L_1(0,t)|=:Z<\infty$, then only finitely many terms can be non-zero. If a term with index $k$ is non-zero then so is every smaller index $i\le k$, since, the term being non-zero implies that there is a vertex $v$ in  $\partial B^{G}_1(0,k)$ with $d_L(0,v)\le t$. But then any vertex $w$ on the shortest path from $v$ to $0$ has $d_L(0,w)\le t$. 
 For any term $k$ that is $0$, $|\cb\partial B^{G}_1(0,k)\cz \cap B^L_1(0,t)| =0$ implies that every vertex at graph distance $k$ is at weighted distance larger than $t$, hence  $d_L^{1}(0,\partial B^{G}_1(0,k) )=X_k>t$. Hence the limit of $X_k$ is also $>t$ in this case. \cz
This finishes the argument that $\lim_{k\to \infty}X_{k_n}=Y_1^E(v)$.

Returning to \eqref{eq:lower-distr}, combined with \eqref{eq:v1xkn} and this convergence result finishes the proof that $\Pv(d_L(v_n^1, v_n^2) \le x)\le  \Pv(Y_1^{(1)}+Y_1^{(2)}\le x)+\ve$ for all $\ve>0$ and sufficiently large $n$, where $Y_1^{(1)}, Y_1^{(2)}$ are i.i.d. copies of $Y_1^E(0)$.

		To estimate $\Pv(d_L(v_n^1, v_n^2)\le x)$ from below, we use Proposition \ref{Prop:upperbound}. Recall the event $\wit A_{n,K}$ from \eqref{eq:ank}.
		Let us write 
		\[ \cb A^\star_{n,K}:= \wit A_{n,K}\cap\{d_L\left(v_n^1,v_n^2\right)\leq T_K(v_n^1)+T_K(v_n^2)+\eps_K\}.\] 		By the statement of Proposition \ref{Prop:upperbound} and Lemma \ref{lem:ank}, $\Pv(A^\star_{n,K})\ge 1-f(n,K)-\eta(K)$. Further, since $T_K(v_n^q)$ denotes the $L$-length of a section on the optimal path to infinity, (see the definition before Lemma \ref{lem:ank}) $T_K(v_n^q)\le Y^E_1(v_n^q)$, the explosion time of $v_n^q$ in $\Eone$. Thus 		\be\label{eq:dl-1} \ba \Pv(d_L(v_n^1, v_n^2)\le x) &\ge (1-f(n,K)-\eta(K))\Pv( Y^E_1(v_n^1)+Y^E_1(v_n^2)+\eps_K \le x ).\ea
		\ee
		We will show below that for two i.i.d.\ copies $Y^{(1)}, Y^{(2)}$ of the explosion time of $0$ in $\Eone$,
		\be\label{eq:dl-2} (Y^E_1(v_n^1),Y^E_1(v_n^2)) \toindis (Y^{(1)}, Y^{(2)}).\ee
		Given this, fix $\delta, x$ first and \cb then choose \cz $K_1$ such that for all $K\ge K_1$, $\eta(K)\le \delta/4$. Let  us assume that $x$ is a continuity point of the distribution of $Y^{(1)}+ Y^{(2)}$,  and $K_2\ge K_1$ satisfies that for all $K\ge K_2$,
		\be\label{eq:dl-3}  | \Pv\left(Y^{(1)}+ Y^{(2)} \le x-\eps_{K} \right) - \Pv\big( Y^{(1)}+ Y^{(2)} \le x\big)| \le \delta/4.\ee
		If a function is continuous at $x$, then it is also continuous on some interval $(x-\zeta,x]$, \cb for some small $\zeta>0$\cz, so let $K_3\ge K_2$ satisfy further that $\eps_{K_3}\le \zeta$.  Thus, $x-\ve_{K_3}$ is also a continuity point of the distribution of $Y^{(1)}+ Y^{(2)}$. Let $n_1$ now satisfy that for all $n\ge n_1$,
		\be\label{eq:dl-4}  | \Pv\left(Y^E_1(v_n^1)+Y^E_1(v_n^2) \le x-\eps_{K_3} \right) - \Pv\big( Y^{(1)}+ Y^{(2)} \le x-\ve_{K_3}\big)| \le \delta/4.\ee
		Note that $f(n,K)$ for fixed $K$ \cb tends to zero with $n$ \cz (see Lemma  \ref{lem:ank}).
		 So, set $n_2:=n_2(K_3)\ge n_1$ such that for all $n\ge n_2$, $f(n,K_3)\le \delta/4$. 
		Combining \eqref{eq:dl-1}-\eqref{eq:dl-4} yields that for all $n\ge n_2$,
		\be\label{eq:dl-5} \ba \Pv(d_L(v_n^1, v_n^2)\le x) &\ge (1-f(n_2,K_3)-\eta(K_3))\left( \Pv( Y^{(1)}+ Y^{(2)} \le x-\ve_{K_3})\pm \delta/2\right)\\
		&\ge \Pv( Y^{(1)}+ Y^{(2)} \le x) -\delta.\ea
		\ee
		This finishes the proof of the lower bound, given \eqref{eq:dl-2}. Next we show \eqref{eq:dl-2}. Recall from the proof of the upper bound that on the event $A_{k_n}$, \cb as in \eqref{eq:ak}\cz, the variables $V_q^1(n, k_n)=d^1_L(v_n^q,\cb\partial B_1^{G,L}(v_n^q, k))$ are  independent variables for $q\in\{1,2\}$, and they satisfy the distributional identity \eqref{eq:v1xkn} on $A_{k_n}$ and they approximate, as $k_n\to \infty$, two i.i.d. copies of $Y_1^{\mathrm{E}}(0)$. \cb Let both $z_1$ and $z_2$ be continuity points of the distribution of $Y_1^E(0)$\cz.
		Thus, estimating the joint distribution functions of \eqref{eq:dl-2} by the triangle inequality, we write
		\be\label{eq:indep-111}\ba  |\Pv(Y^E_1(v_n^1)\le z_1,& Y^E_1(v_n^2) \le z_2) - \Pv(Y^{(1)}\le z_1, Y^{(2)} \le z_2)|   \\
		\le{}&|\Pv(Y^E_1(v_n^1)\le z_1, Y^E_1(v_n^2) \le z_2) - \Pv(V_1^1(n, k_n)\le z_1, V_2^1(n,k_n) \le z_2)|\\
		&+|\Pv(V_1^1(n, k_n)\le z_1, V_2^1(n,k_n) \le z_2) - \Pv(X_{k_n}^{(1)}\le z_1, X_{k_n}^{(2)} \le z_2)|\\
		&+|\Pv(X_{k_n}^{(1)}\le z_1, X_{k_n}^{(2)} \le z_2 ) - \Pv(Y^{(1)}\le z_1, Y^{(2)} \le z_2 )|\\
		\cb=:{}&\cb T_1+T_2+T_3,\ea \ee
		\cb where we denote the terms on the rhs in the second, third and fourth row by $T_1, T_2$ and $T_3$, respectively\cz. Then, due to the distributional identity \eqref{eq:v1xkn} on $A_{k_n}$, $T_2$ is bounded from above by  $1-\Pv(A_{k_n})$, and
		$T_3$ tends to $0$ by the argument after \eqref{eq:intersect} as $k_n\to \infty$ (that is, $n\to \infty$). To estimate $T_1$, note that the four variables involved are all defined on the same probability space (\cb namely, in terms of the realisation $\Eone$\cz). Further, fixing $n$, the variable $V_q^1(n,k)=d_L^1(v_n^q, \cb\partial B_1^{G,L}(v_n^q,k))$ is increasing in $k$ and tends to $Y_1^E(v_n^1)$, the explosion time of $v_n^q$ in a realisation of $\Eone$ as $k\to \infty$. In particular, for every $k\ge 0$, almost surely, $Y_1^E(v_n^1)-V_q^1(n,k)\ge 0$ holds.	Hence
			\be \ba \cb T_1&=|\Pv(Y^E_1(v_n^1)\le z_1, Y^E_1(v_n^2) \le z_2) - \Pv(V_1^1(n, k_n)\le z_1, V_2^1(n,k_n) \le z_2)|\cz\\
		&=\Pv(V_1^1(n, k_n)\le z_1, V_2^1(n,k_n) \le z_2)-\Pv(Y^E_1(v_n^1)\le z_1, Y^E_1(v_n^2) \le z_2) \ea\ee
		 Next we observe that $Y_1^E(v_n^1)-V_q^1(n,k)\ge 0$ tends to zero as we increase $k$, so the difference is unlikely to be less then $\delta>0$ for $k$ large enough. Thus, we can add for any $\delta>0$ the extra term in the middle:
\be\ba
		T_1&\le \Pv\big(V_1^1(n, k_n)\le z_1, V_2^1(n,k_n) \le z_2) - \Pv(V_1^1(n, k_n)\le z_1-\delta, V_2^1(n,k_n) \le z_2-\delta\big)\\
		&\quad+ \Pv(V_1^1(n, k_n)\le z_1-\delta, V_2^1(n,k_n) \le z_2-\delta\big) - \Pv(Y^E_1(v_n^1)\le z_1, Y^E_1(v_n^2) \le z_2).\ea\ee
		Note that the first row can be bounded from above by \cb the probability of the event that at least one of the variables $V^1_q(n, k_n)$ falls in a small interval of length  $\delta$\, while 
		the second row equals the probability below\cz:
		\be \label{eq:t1-123} \ba T_1&\le  \Pv(V_1^1(n, k_n )\in (z_1-\delta, z_1])+\Pv(V_1^1(n, k_n )\in (z_2-\delta, z_2])\\
		&\quad +\Pv\big( \{  V_1^1(n, k_n)\le z_1-\delta,  V_2^1(n, k_n)\le z_2-\delta \} \setminus \{Y^E_1(v_n^1)\le z_1,  Y^E_1(v_n^2) \le z_2\} \big)\\
		\ea \ee
		The event in the second row is satisfied only when $Y^E_q(v_n^1)-V^1_q(n, k_n)\ge \delta$ for one of the $q=1,2$'s.
Due to the translation invariance of the model, we can use that \cz \emph{marginally} $(V_q^1(n,k_n), Y_1^E(v_n^q))$ has the same distribution as $(X_{k_n}, Y_1^E(0))$ from \eqref{eq:xkn}, and hence, by a union bound, the second row in \eqref{eq:t1-123} is at most 
	\be\label{eq:end-indep} 2 \Pv(Y_1^E(0)-X_{k_n} \ge \delta ),\ee
		which tends to $0$ for every fixed $\delta>0$ as $k_n\to \infty$ by the argument after \eqref{eq:v1xkn}.
		To bound the first row on the rhs of \eqref{eq:t1-123}, due to the translation invariance again, marginally, $V^1_q(n, k_n)=X_{k_n}^{(q)}\ {\buildrel d \over =}\ X_{k_n}$ and $X_{k_n}\toindis Y_1^E(0)$, and $z_1, z_2$ are continuity points of the distribution of $Y_1^E(0)$.	Hence, for any $\ve>0$, one can choose 
		first $\delta>0$ small enough and then $k_n$ large enough so that this probability is less than $\ve$.
	\end{proof}
	\section{Branching Random Walks in random environment and GIRGs}\label{sec:BRW}
	In this section, we provide the framework for the lower bound in Theorem \ref{Th:GIRGexplosive}. As a preparation for Proposition \ref{Prop:upperbound}, we show that every path with finite total $L$-length to infinity must leave the subgraph of $\Elambda$ restricted to vertices with bounded weight. 
	Both proofs rely on a coupling to a process, a \emph{branching random walk} (BRW), \cb that has a faster growing neighbourhood around a vertex $v$ than the infinite component of $\Elambda$\cz. Recall $\Eupper$ from Def.~\ref{def:Eupper}.
	\cb This BRW has a random environment\cz. The environment is formed by $(x, W_{x})_{x\in \CV_\la}$, where $\CV_\la$ is a PPP of intensity $\lambda$ on $\R^d$ and $(W_{x})_{x\in \CV_\la} $ are i.i.d.\ copies of $W$. 
	We name the individuals in the \emph{skeleton} of the BRW (a random branching process tree without any spatial embedding), via the Harris-Ulam manner. That is, children of a vertex are sorted arbitrarily and we call the root $\empty$, its children $1,2,\ldots$, their children $11,12,\ldots,21,22,\ldots$ and so on. Generally, individual $\underline i:=i_1 i_2\ldots i_k$ is the $i_k^{\text{th}}$ child of the $i_{k-1}^{\text{th}}$ child $\ldots$ of the $i_1^{\text{st}}$ child of the root. This coding for all individuals is referred to as the \emph{name} of an individual. \cb We write $p\left(\underline i\right)$ for parent of individual $\underline i$\cz, set $p\left(\empty\right)=\empty$. \cb Abusing notation somewhat\cz, we write $x_{\underline i}, W_{\underline i}$ for the location and weight of the individual $\underline i$.
	
	We locate the root $\empty$ at the initial vertex $v\in \CV_\la$, thus we set $x_{\empty}:=v$. 	
	Conditioned on the environment, let the number of children at location $x\in \CV_\lambda\backslash\{x_{\underline i}\}$ of every individual  $\underline i$ be \emph{independent}, and denoted and distributed as
	\be\label{berbrw-edge-prob} N^B_x\left(\underline i \right)\ {\buildrel d \over =}\  \text{Ber}\left(\overline C_1\min\left\{1,\overline a_1\|x_{\underline i}-x\|^{-\alpha d}\left(W_{\underline i} W_x\right)^\alpha\right\}\right)=  \text{Ber}\left(\overline C_1 \overline g(x_{\underline i}-x, W_{\underline i}, W_x)\right).\ee 
	Then, the location  and weight of the offspring of $\underline i$ in the Bernoulli BRW is described as 
	\be\label{BerBRW}
	\mathcal{N}^{B}\left(\underline i\right):=\bigcup_{x\in \CV_\lambda\backslash\{x_{\underline i}\}:N^B_x(\underline i)=1}\left\{(x, W_x)\right\}.
	\ee	
	Note that individuals at the same location reproduce conditionally independently. We denote the resulting BRW by $\mathrm{BerBRW}_\la(v)$ and its generation sets by $\cG^{\text{Ber}}_\lambda(v,k),k\in\mathbb{N}$ and $v$ the root.
	Note that the Bernoulli BRW is created so that the probability that a particle at a location $y\in \CV_\la$ has a child at location $z\in \CV_\la\backslash\{y\}$ equals the probability of the presence of the edge $(y,z)$ in $\Eupper$, namely
	\begin{equation}\label{upperboundingBRW}
	\begin{aligned}
	\mathbb{P}&\left(\text{an individual at }y \text{ has a child at }z  \text{ in } \mathrm{BerBRW}_\la(v)\mid\left(x ,W_{x}\right)_{x \in \CV_\lambda}\right)\\
	&\qquad \qquad\qquad\qquad\qquad \qquad\qquad\qquad\qquad= \P{(y,z) \in \overline\CE_\la\mid\left(x, W_{x}\right)_{x\in \CV_\lambda}}.
	\end{aligned}
	\end{equation}
	After $\mathrm{BerBRW}_\la(v)$ is generated, we assign i.i.d.\ edge-lengths from \cb distribution $F_L$ \cz to all existing parent-child relationships $(p(\underline i), \underline i).$ We denote the resulting edge-weighted BRW by $\mathrm{BerBRW}_{\la,L}(v)$. Let us denote by $ \overline B_\la^L(v,t), \overline B_\la^G(v,t)$ the set of vertices within $L$-distance $t$ and graph distance $t$ of $v$ in $\Eupper$, and $ B_{\mathrm{Ber}}^L(v,t), B_{\mathrm{Ber}}^G(v,t)$ the corresponding quantities in $\mathrm{BerBRW}_{\la, L}(v)$, that is, the set of individuals available from $\emptyset=v$ on paths of $L$-length at most $t$ and  the set of individuals with name-length ($=$ generation number) at most $t$, respectively.	
	The exploration process on a similar BRW (defined in terms of the SFP model), coupled to the exploration on the SFP model is introduced in \cite[Section 4.2]{HofKom2017}. The exploration process on $\mathrm{BerBRW}_{\la,L}(v)$ (as defined above) coupled to the exploration on $\Eupper$ is analogous. The exploration algorithm runs on $\mathrm{BerBRW}_{\la,L}(v)$ and $\Eupper$ at the same time, exploring  $B_{\mathrm{Ber}}^L(v,t)$ as $t$ increases, and thinning parent-child relationships in $\mathrm{BerBRW}_{\la,L}(v)$ that are between already explored locations, yielding $\overline  B_{\la}^L,(v,t)$.
	A consequence of this thinning procedure is the following lemma, which is an adaptation and simplified statement of \cite[Proposition 4.1]{HofKom2017}:
	\begin{lemma}\label{lemma:ballsbrw}
		There is a coupling of the exploration of $(B^{L}_{\mathrm{Ber}}\left(v,t\right))_{t\ge 0}$ on $\mathrm{BerBRW}_{\la, L}(v)$  to the exploration of $(\overline B_\la(v,t))_{t\ge 0}$ on $\Eupper$ so that for all $t\geq 0$, under the coupling,
		\be\label{ballsubset}
		\overline B^{L}_\la\left(v,t\right)\subseteq B^{L}_{\mathrm{Ber}}\left(v,t\right).
		\ee
\cb Further, there is a coupling of $\mathrm{BerBRW}_{\la}(v)$ and $\Eupper$ such that under the  coupling,
\be \partial \overline B^{G}_\la\left(v,t\right)\subseteq \partial B^{G}_{\mathrm{Ber}}\left(v,t\right),  \ee
i.e. the vertices at graph distance $k$ from $v$ in $\Eupper$ form a subset of the individuals in generation $k$ of $\mathrm{BerBRW}_{\la}(v)$.\cz
	\end{lemma}
	\begin{proof}
		The proof of Lemma \ref{lemma:ballsbrw} is a direct analogue of the proof of \cite[Proposition 4.1]{HofKom2017} and we refer the interested reader to their proof. The only difference is that in their proof the vertex set is $\Z^d$ while here it is a PPP $\CV_\la$. Heuristically, the coupling follows by \eqref{upperboundingBRW}, but it is non-trivial since one needs to assign the edge-lengths to parent-child relationships in $\mathrm{BerBRW}_\la(v)$ as well as the edge-lengths $L_e$ in $\Eupper$ also in a coupled way.
	\end{proof} 
	A combination of the couplings in Claim \ref{claim:edge-containment} and Lemma \ref{lemma:ballsbrw} yields that for all $t$, under the couplings, \cb for all $\la\ge 1+\xi_n$,\cz
	\be\label{eq:ball-containment}\ba 
	B^{L}_\la\left(v,t\right)\subseteq \overline B^{L}_\la\left(v,t\right)\subseteq B^{L}_{\mathrm{Ber}}\left(v,t\right) \quad \text{and} \quad  B^{G}_\la\left(v,t\right)\subseteq \overline B^{G}_\la\left(v,t\right)\subseteq B^{G}_{\mathrm{Ber}}\left(v,t\right),\\
	B^{L}_n\left(v,t\right)\subseteq \overline B^{L}_\la\left(v,t\right)\subseteq B^{L}_{\mathrm{Ber}}\left(v,t\right) \quad \text{and} \quad  B^{G}_n\left(v,t\right)\subseteq \overline B^{G}_\la\left(v,t\right)\subseteq B^{G}_{\mathrm{Ber}}\left(v,t\right),
	\ea\ee 
	where the last row holds when we further assume that $W_u^{(n)}=W_u$ for all $u \in B^{L}_n\left(v,t\right)$ and $u\in B^{G}_n\left(v,t\right)$, respectively, and $B^L_n, B^{G}_n$ are the quantities in $\BGIRG$.
	
	\subsection{Non-explosiveness with finite second moment weights.}\label{sec:brw-nonexplosion}
	The aim of this section is to prove Theorem \ref{thm:exp-charact} and its following corollary.
	\begin{corollary}[Corollary to Theorem \ref{thm:exp-charact}]\label{lem:expgen}
		Let $\CV_{\lambda, \le K}:=\{x\in \CV_\lambda, W_x\le K\}$ and 
		consider an infinite path $\cb\pi:=(\pi_0=v, \pi_1, \dots )\subset \mathcal E_\la$. If the total $L$-length of the path $\cb|\pi|_L=\sum_{i\ge 0} L_{(\pi_i, \pi_{i+1})}<\infty$, then 
		$\cb\pi \not \subseteq\CV_{\lambda, \le K}$. In words, any infinite path with finite total $L$-length has vertices with arbitrarily large weights. \end{corollary}
	\begin{proof}[Proof of Corollary \ref{lem:expgen} subject to Theorem \ref{thm:exp-charact}]
		We prove the statement for $\Eupper$ as well. 
		Let us denote shortly by $G_{\le K}$, (resp., $\overline G_{\le K}$) the subgraph of $\Elambda$ (resp.\ $\Eupper$) spanned by vertices in $\CV_{\lambda, \le K}$. Note that the statement is equivalent to showing that explosion is impossible within $G_{\le K}$ (resp.\ within $\overline G_{\le K}$). 
		Since the vertex-weights $(W_x)_{x\in\CV_\la }$ are i.i.d., every $x\in \CV_\la$ belongs to the vertex set  $\CV_{\la,\le K}$  independently with probability $\Pv(W\le K)$. An independent thinning of a Poisson point process (PPP) is again a PPP, thus $\CV_{\la,\le K}$ is a PPP with intensity $\la_K:=\la \Pv(W \le K)$.  Further, every location $x\in \CV_{\la,\le K}$ receives an i.i.d.\ weight distributed as $W_K\  {\buildrel d \over {:=}}\ (W \mid W\le K)$. Finally, edges are present conditionally independently, and the edge between $y,z\in \CV_{\la\le K}$ with given weights $w_y,w_z$  is present with probability $h(y-z, w_y, w_z)$ in $\CE_\lambda$ (respectively, $\overline C_1 \overline g(y-z, w_y, w_z)$ in $\overline\CE_\la$). 
		Thus, $G_{\le K}$ and $\overline G_{\le K}$ can be considered as an instance of $\mathrm{EGIRG}_{W_K, L}(\la_K)$ and $\overline{\mathrm{EGIRG}}_{W_K, L}(\la_K)$, respectively,  see Definitions~\ref{def:EGIRG} and \ref{def:Eupper}.
		Theorem \ref{thm:exp-charact} now applies since $\Ev[W_K^2]<\infty$.
	\end{proof}
	To prove the non-explosive part of Theorem \ref{thm:exp-charact}, we need a general lemma. We introduce some notation related to trees first. Consider a rooted (possibly random) infinite tree $\CT$. Let $\CG_k$, \emph{generation} $k$, be  the vertices at graph distance $k$ from the root. 
	\begin{lemma}\label{non-explosion-exp}  Consider a rooted (possibly random) tree $\CT$.  Add i.i.d.\ lengths from \cb distribution $F_L$ \cz to each edge with $\Pv(L=0)=0$. Let $d_L(x, \emptyset):=\sum_{e\in P_{\emptyset,x}} L_e$ be the $L$-distance of  $\emptyset, x\in \CT$, where $P_{\emptyset,x}$ is  the unique path  from vertex $x$ to the root $\emptyset$.  
		Suppose that the tree has at most exponentially growing generation sizes, that is,  
		\be\label{eq:gen-sizes-nu} \exists \nu \in (1, \infty): \Pv\Big( \bigcup_{k_0\in \N} \bigcap_{k\ge k_0} \{|\CG_k| \le \nu^k\} \Big) =1. \ee
		Then explosion is impossible on $\CT$:	
			\[ \Pv( \lim_{k\to \infty} d_{L}( \emptyset, \CG_k) < \infty) = 0. \]
	\end{lemma}
	\cb\begin{proof}\cz
	By \eqref{eq:xkn-toexp}, the limit  $\lim_{k\to \infty} d_{L}( \emptyset, \CG_k)$ equals the explosion time of the root $\emptyset$. Our goal is to show that there exists a constant such that $ d_{L}( \emptyset, \CG_k)\ge c_L k$ for all sufficiently large $k$. Hence the limit is a.s. infinite.
		Let us fix a small enough $\ve>0$ and set $t_0$ so small that $F_L(t_0)<1/(4\nu^2)-\ve$. Then, let us define the event
		\[ E_k:=\{ \exists y \in \CG_k: d_L(y,\emptyset) < k t_0/2\}.\]
		Then, since every vertex in $\CG_k$ is graph distance $k$ from the root,  by a union bound, 
		\be\label{eq:en1} \Pv(E_k \mid |\CG_k|) \le |\CG_k| F_L^{\star,k}(k t_0/2), \ee
		where we write $F_L^{\star,k}$ for the $k$-fold convolution of $L$ with itself\footnote{The distribution of $L_1+L_2+\dots +
			L_k$, with $L_i$ i.i.d.\ from $L$.}. Now, we bound  
		$F_L^{\star,k}(k t_0/2)$. Note that if more than $k/2$ variables in the sum $L_1+\dots L_k$ have length at least $t_0$, then the sum exceeds $k t_0 /2$. As a result, we must have at least $k/2$ variables with value at most $t_0$.
		Thus, 
		\be\label{eq:FL-convolution} F_L^{\star, k}(k t_0 /2) \le {k \choose k/2}(F_L(t_0))^{k/2} \le 2^k \left( (4\nu^2)^{-1}-\ve\right)^{k/2}.  \ee
		Using this bound in \eqref{eq:en1}  and that $|\CG_k\le \nu^k$ for all $k\ge k_0$ for some $k_0$, we arrive to the upper bound
		\[\Pv(E_k) \le \nu^k 2^k \left( (4\nu^2)^{-1}-\ve\right)^{k/2}=\left(1- 4 \nu^2 \ve\right)^{k/2}. \]
		Since the rhs is summable, the Borel-Cantelli Lemma implies that a.s.\ only finitely many $E_k$ occur. That is, for all large enough $k$, all vertices in generation $k$ have $L$-distance at least $k t_0/2$. Thus, $d_L(x,  \CG_k)$ tends to infinity and the result follows. 
	\end{proof}
	\cz
	\begin{proof}[Proof of Theorem \ref{thm:exp-charact}, finite second moment case]
		First we show non-explosion when $\Ev[W^2]<\infty$. We prove the statement for $\Elambda$ and $\Eupper$ as well. By translation invariance, it is enough to show that the explosion time $Y_E^\la(0)$ (see Def.\ \ref{def:explosiontime}) of the origin is a.s.\ infinite, given $0\in\CV_\la$. \cb Recall the convergence $X_{k}=d_L^\la(0, \cb\partial B_\la^{G}(0, k)\cz)\toas Y_1^E(0)$ as $k\to \infty$ from \eqref{eq:xkn-toexp} where it is stated for $\la=1$. Similar to the proof of Lemma \ref{non-explosion-exp}, we show that for some constant $c_L$ that depends on $L$,  $X_k\ge c_L k$ for all $k$ sufficiently large, hence the limit cannot be finite.\cz
		
		Consider the corresponding upper bounding $\mathrm{BerBRW}_{\la}(0)$, with environment $(x, W_x)_{x\in \CV_{\la}}$. By the coupling in Lemma \ref{lemma:ballsbrw}, and \eqref{eq:ball-containment},  $\partial B_\la^{G}(0, k)$ is contained in generation $k$ (denoted by $\CG_k^{\la}$) of $\mathrm{BerBRW}_{\la}(0)$, hence it is enough to show that 
		\be\label{eq:xkber} X_k^{\mathrm{Ber}}=d_L^\la(0, \cb\CG_k^{\la}\cz) \ge c_L k\ee
		for all sufficiently large $k$, since $ X_k\ge X_k^{\mathrm{Ber}}$. The inequality \eqref{eq:xkber} follows from Lemma \ref{lem:expgen}, when we  show that $|\CG_k|$ grows at most exponentially in $k$ when $\Ev[W^2]<\infty$. Thus, we aim to show that when $\Ev[W^2]<\infty$, for some $m<\infty$ and constant $C$,
		\be\label{eq:gensizes-111} \Ev[|\CG_k^{\la}|] \le C m^k.\ee
		Indeed, setting $A_k:=\{ |\CG^\la_k| \ge (2m)^k \}$, by Markov's inequality, $\Pv(A_k)\le C2^{-k}$, summable in $k$. By the Borel-Cantelli Lemma, there is a $k_0$ such that $A_k^c$ holds for all $k\ge k_0$, thus \eqref{eq:gen-sizes-nu} holds with $\nu:=2m$.
		At last we show \eqref{eq:gensizes-111}. Recall the definition of the displacement probabilities in $\mathrm{BerBRW}$ from \eqref{berbrw-edge-prob}. For a set $\CA\subset \R_+$, $N_{w}^{\mathrm{B}}(\CA)$, called the \emph{reproduction kernel}, denotes the number of children with vertex-weight in the set $\CA$ of an individual with vertex-weight $w$ located at the origin. \cb Let us fix $s\ge 1$ an arbitrary number, and think of  $\mathrm ds$ as infinitesimal, and slightly abusing notation let us write $N_{w}^{\mathrm{B}}([s,s+\mathrm ds]):=N_{w}^{\mathrm{B}}(\mathrm ds)$\cz. In case of the $\mathrm{BerBRW}_{\la}(v)$, the distributional identity holds: 
		
		\begin{equation}\label{eq:kernel-1}
		N_{w}^{\mathrm{B}}(\mathrm d s) \ {\buildrel d \over =}   \sum_{y\in \CV_{\la}} \ind_{\{W_y\in (s, s+\mathrm ds)\}}  \mathrm{Ber}\big( \overline C_1\big(1\wedge\overline a_1\|y\|^{-\alpha d}(w s)^\alpha\big)\big).
		\end{equation}

		The \emph{expected reproduction kernel} (cf.\ \cite[Section 5]{Jagers89}) is defined as $\mu(w, \mathrm ds):=  \Ev[N_w^{\mathrm{B}}(\mathrm ds)]$, that can be bounded from above as
		\be\label{eq:mu-ber} \mu(w,\mathrm ds )
		\le F_{W}(\mathrm ds) \overline C_1 \Ev\Bigg[\cb\bigg(\sum_{
			y \in \CV_{\la}: \|y\|\le (a_1^{1/\al}ws)^{1/d}}\hspace{-25pt} 1 \hspace{25pt}\bigg)\cz +  \overline a_1(ws)^\al \sum_{y\in \CV_{\la}: \|y\|\ge a_1^{1/(\al d)}(ws)^{1/d}} \|y\|^{-\al d}\Bigg],\ee
			\cb where\footnote{The $\mathrm ds$ refers to the Lebesgue-Stieltjes integration wrt to the measure generated by $F_L$, e.g. when $L$ has a density $f_L$ then $F_L(ds)=f_L(s) ds$.} $F_L(\mathrm ds):=F_L(s+\mathrm ds)-F_L(s)$\cz.	The expectation of the first sum equals $\la \mathrm{Vol}_d  \cdot ws$, with $\mathrm{Vol}_d$ denoting the volume of the unit ball in $\R^d$. The expectation of the second sum equals, for some constants $c_d, \wit c_d$,
		\be\label{eq:sumterm-1} (ws)^{\al}\la\overline a_1 \int_{(ws)^{1/d}}^{\infty} c_d t^{d-1-\al d} \mathrm dt =(ws)^{\al} \la \wt c_d (ws)^{1-\al} = \la \wt c_d ws.  \ee
		With $c_\la:=(\mathrm{Vol}_d + \wt c_d)$, defining 
		\be\label{eq:muber-2}\mu_+(w, \mathrm ds):= ws F_{W}(\mathrm ds) c_{\la},\ee we obtain an upper bound 
		$  \mu(w,\mathrm ds) \le \mu_+(w, \mathrm ds)$ \emph{uniform in} $s,w$. 
		
		By the translation invariance of the displacement probabilities in BerBRW,  the \emph{expected} number of individuals with given weight in generation $k$ can be bounded from above by applying  the composition operator $\ast$ acting on the `vertex-weight' space (type-space), that is, define 
		\[ \mu^{\ast n}(w, \CA):=\int_{\R_+} \mu(v, \CA) \mu^{\ast (n-1)}(w, \mathrm \mathrm ds), \qquad  \mu^{\ast1}:=\mu. \]
		For instance, $\mu^{\ast2}$ counts the expected number of individuals with weight in $\CA$ in the second generation, averaged over the environment.  
		The rank-1 nature of the kernel $\mu_+$ implies that its composition powers factorise. An elementary calculation using \eqref{eq:muber-2} shows that 
		\be\label{eq:mu-convolv}\mu_+^{\ast k}(w, \mathrm ds) =wsF_{W}(\mathrm ds) \cdot (c_{\la})^k \Ev[W^{2}]^{k-1}. \ee
		Since the weight of the root is distributed as $W$,
		\be \Ev[|\CG^\la_k|]=\Ev[\mu^{\ast k}(W, \R_+) ]\le \Ev[\mu^{\ast k}_+(W, \R_+)]
		=\big(c_{\la}\Ev[W^{2}]\big)^k \cdot \Ev[W] / \Ev[W^2].  \ee  
		This finishes the proof with $m:=c_{\la}\Ev[W^{2}]<\infty$.
			\end{proof}
	\subsection{Growth of the Bernoulli BRW when $\tau\in(2,3)$}\label{sec:BRWanalysis}
	To prove the lower bound on the $L$-distance  in Proposition \ref{Prop:lowerbound}, we analyse the properties of $\mathrm{BerBRW}_{\la}(v)$. We investigate the doubly-exponential growth rate of the generation sizes and the maximal displacement. Note that this analysis does not include the edge-lengths. The main result is formulated in the following proposition, that was proved for the SFP model in \cite{HofKom2017} that we extend to $\Elambda$:
	
	\begin{proposition}[Doubly-exponential growth and maximal displacement of BerBRW]\label{Prop:BerBRW} 
		Consider the $\mathrm{BerBRW}_{\la}(v)$ in $\R^d, d\ge 1$, with weight distribution $W$ satisfying \eqref{powerlaw} with $\tau\in\left(2,3\right)$, \cb and reproduction distribution \cz given by \eqref{berbrw-edge-prob} with $\alpha>1$. Let $\mathcal{G}^{\mathrm{Ber}}_\lambda(v,k)$, $Z^{\mathrm{Ber}}_\lambda(v,k)$ denote the set and size of generation $k$ of  $\mathrm{BerBRW}_{\la}(v)$. We set
		\be\label{doubleexp}
		\ba
		c_k\left(\eps,i\right):= & 2\exp\Big\{i\left(1+\eps\right)\left(\frac{1+\eps}{\tau-2}\right)^k\Big\}, \qquad
		S_k\left(\eps,i\right):= & \exp\Big\{i\zeta\left(\frac{1+\eps}{\tau-2}\right)^k\Big\},
		\ea
		\ee
		where $\zeta\geq \max\{2(\eps/(1+\eps)+(\tau-1)/(\tau-2))/d,(2\alpha+\eps(\tau-2)/(1+\eps))/((\alpha-1)d)\}$ is a constant. \cb Recall the definition of $B^2(v,r)$ in \eqref{metricballs}\cz. Then, for every $\eps>0$, there exists an a.s. finite random variable $Y\left(\eps\right)<\infty$, such that the event
		\be\label{BerBRWpropstatement}
		\forall k\geq 0:\ Z^\mathrm{Ber}_\la(v,k)\leq c_k(\eps,Y(\ve)) \text{ and } \mathcal{G}^\mathrm{Ber}_\la(v,k)\cap \cb\left( B^2(v,S_k\left(\eps,Y(\ve)\right))\right)^c\cz=\empty
		\ee
		holds. Furthermore, the $Y\left(\eps\right)$ have exponentially decaying tails with
		\be\label{Ytail}
		\ba
		\P{Y\left(\eps\right)\geq K}& \leq C_\eps\exp\left\{-\eps K/4\right\},
		\ea
		\ee
		where the constant $C_\eps$ depends on $\eps,\tau,\alpha,\la,d$ and the slowly-varying function $\ell$ only, not on $K$.
	\end{proposition}
	So, both the generation size and the maximal displacement of individuals in generation $k$ grow at most doubly exponentially with rate $\left(1+\eps\right)/\left(\tau-2\right)$ and random prefactor $Y(\ve)$. We believe based on the branching process analogue \cite{Dav78} that the true growth rate is $1/(\tau-2)$ with a random prefactor, but for our purposes, this weaker bound also suffices. 
	It will be useful to know the definition of $Y(\ve)$ later.
	\begin{remark}\label{rem:Ydef}\normalfont
		Let us define
		\be\ba \label{def:yeps}
		\wt{E}_k(\eps,i)&:=\left\{Z^{\mathrm{Ber}}_\la(v,k)\leq c_k(\eps,i)\right\}\cap \left\{\mathcal{G}^{\mathrm{Ber}}_\la(v,k)\cap \cb (B^2(v,S_k(\eps,i)))^c\cz=\empty \right\},\\
		H_i&:=\bigcap_{k=0}^\infty \wt{E}_k(\eps,i),
		\ea\ee
		that is, $H_i$ describes the event that the generation sizes and maximal displacement of $\mathrm{BerBRW}_\la(v)$ grow (for all generations) doubly exponentially with prefactor at most $i$ in the exponent, as in \eqref{doubleexp}.
		It can be shown that $\sum_{i=1}^\infty \P{H_i^c}<\infty$. It follows from the Borel-Cantelli lemma that only finitely many $H_i^c$ occur. Since $c_k(\ve, i), S_k(\ve,i)$ are increasing in $i$, $H_1\subset H_{2} \dots$ is a nested sequence. Thus $Y(\ve)$ is the first $i$ such that $H_i$ holds. The rhs of \eqref{Ytail} equals $\sum_{i\ge K} \Pv(H_i^c)$.
	\end{remark}
	The key ingredient to proving Proposition \ref{Prop:BerBRW} is Claim \ref{claim:Nw} below. Let us write $N_w\left(\mathrm ds,\geq S\right)$ for the number of offspring with weight in $\left(s,s+\mathrm ds\right)$ and displacement at least $S$ of an individual with weight $w$, and then let $N_w\left(\geq t,\geq S\right)$ for the number of offspring with weight at least $t$ and displacement at least $S$  of an individual with weight $w$. 
	\begin{claim}[Bounds on the number of offspring of an individual of weight $w$]\label{claim:Nw}
		Consider $\mathrm{BerBRW}_{\lambda}(v)$ from Section \ref{sec:BRWanalysis}, with $\tau>2$, $d\geq 1$ and $\alpha>1$. Then,
		\be\label{boundsonNw}
		\ba
		\E{N_w\left(\geq t,\geq 0\right)} \leq {}&  M_{d,\tau,\alpha}\lambda \,w\,t^{-\left(\tau-2\right)}\ell\left(u\right),\\
		\E{N_w\left(\geq 1,\geq S\right)} \leq{}&  M_{d,\tau,\alpha}\lambda\,  w^{\tau-1}S^{-\left(\tau-2\right)d}\,\wit \ell\left(S^d/w\right)
		+ \mathbbm{1}_{\left\{\alpha< \tau-1 \right\}} M_1 \lambda w^\alpha S^{-\left(\alpha-1\right)d},
		\ea
		\ee
		where constants $\M$ and $M_1$ depend on $d,\tau,\alpha$ and the slowly-varying function $\ell$ only, and $\wit \ell(x):=\ind_{\{\al>\tau-1\}}\ell(x) + \ind_{\{\al=\tau-1\}}\int_1^x \ell(t)/t\  \mathrm dt$ is slowly varying at infinity.
	\end{claim}
	This claim, combined with translation invariance of the model, allows one to estimate the expected number of individuals with weights in generation-dependent intervals, which in turn also allows for estimates on the expected number of individuals with displacement at least $S_k(\ve,i)$ in generation $k$. These together with Markov's inequality give a summable bound on $\Pv(H_i^c)$ (from \eqref{def:yeps}). As explained in Remark \ref{rem:Ydef}, a Borel-Cantelli type argument directly shows Proposition \ref{Prop:BerBRW}.  This method is carried out in \cite[Lemma 6.3]{HofKom2017} for the scale-free percolation model, \cb based on \cite[Claim 6.3]{HofKom2017}\cz, that is the same as Claim \ref{claim:Nw} above, modulo that $\gamma-1$ there is replaced by $\tau-2$ here. 
	The proof of Proposition \ref{Prop:BerBRW} subject to Claim \ref{claim:Nw} follows in the exact same way, by replacing $\gamma-1$ in \cite{HofKom2017} by $\tau-2$ here. The proof of Claim \ref{claim:Nw} is similar to \cb the proof in \cite[Claim 6.3]{HofKom2017}\cz, with some differences. First, the vertex set here forms a PPP while there it is $\Z^d$, and second,  the exponents need to be adjusted, due to the relation between the parameters used in the SFP and GIRG models, as explained in Section \ref{sec:discussion}. 
	\begin{proof}[Proof of Claim \ref{claim:Nw}]
		We start bounding $\mathbb{E}\left[N_w\left(\geq 1,\geq S\right)\right] $. Due to translation invariance, we assume that the individual is located at the origin. Recall the environment is $(x, W_x)_{x\in \CV_\la}$, and the displacement distribution from \eqref{berbrw-edge-prob}.  Using the notation $F_W\left(\mathrm ds\right)=\P{W\in\left(s,s+\mathrm ds\right)}$, for the Stieltjes integral similarly to \eqref{eq:mu-ber},
		\be\label{NwSupperbound}
		\ba
		\mathbb{E}\left[N_w\left(\geq 1,\geq S\right)\right]  &\leq \overline C_1 \int_{1\leq s\leq S^d/w}\mathbb{E}\Bigg[\sum_{x\in \mathcal{V}_\lambda:  \|x\|\geq S}\overline a_1\|x\|^{-\alpha d}\left(ws\right)^\alpha\Bigg]F_W\left(\mathrm ds\right)\\
		&\  +  \overline C_1 \int_{s\geq S^d/w}\mathbb{E}\Bigg[\sum_{x\in \mathcal{V}_\lambda:  \|x\|\leq\left(\overline a_1^{1/\al}ws\right)^{1/d}}1\Bigg]F_W\left(\mathrm ds\right)\\
		&\  +  \overline C_1 \int_{s\geq S^d/w}\mathbb{E}\Bigg[\sum_{x\in \mathcal{V}_\lambda:  \|x\|\geq\left(\overline a_1^{1/\al}ws\right)^{1/d}}\overline a_1\|x\|^{-\alpha d}\left(ws\right)^\alpha\Bigg]F_W\left(\mathrm ds\right)\\
		&\cb=:T_1+T_2+T_3.
		\ea
		\ee
		where we mean the Lebesgue-Stieltjes integral wrt $F_L$.
		Bounding the expectation inside $T_1$ and $T_3$,  for some constant $c_d$, similar to \eqref{eq:sumterm-1}, for any $R>0$
		\be\label{eq:sumterm-2} \mathbb{E}\Bigg[\sum_{x\in \mathcal{V}_\lambda:  \|x\|\geq R}\|x\|^{-\alpha d}\left(ws\right)^\alpha\Bigg]=(ws)^{\al} \la \int_{R}^{\infty} c_d t^{d-1-\al d} \mathrm dt =(ws)^{\al} \la \wt c_d R^{-d(\al-1)}.  \ee
		Using this bound for $R=S$ and $R=(\overline a_1^{1/\al}ws)^{1/d}$, for some constant $\wit c_d$,
		\be\label{eq:t1t2} \ba T_1 &= \overline C_1\la\wit  c_d w^\al S^{-d(\al-1)}\,  \mathbb{E}\left[W^\alpha \mathbbm{1}_{\left\{W\leq S^d/w\right\}}\right],\\
		T_3 &= \overline C_1 \la \wit c_d w \, \Ev[W \mathbbm{1}_{\left\{W\ge S^d/w\right\}}] . \ea\ee	
		The following Karamata-type theorem \cb\cite[Propositions 1.5.8, 1.5.10]{BinGolTeu89} \cz
yields that for any $\beta>0$ and $a>0$, and function $\ell$ that varies slowly at infinity,
		\be\label{Karamata}
		\lim_{u\rightarrow\infty}\frac{1}{u^{-\beta}\ell\left(u\right)}\int_u^\infty\ell\left(x\right)x^{-\beta-1}\mathrm{d}x=\lim_{u\rightarrow\infty}\frac{1}{u^\beta\ell\left(u\right)}\int_a^u\ell\left(x\right)x^{\beta-1}\mathrm{d}x=\frac{1}{\beta}.
		\ee
		while for $\beta=0$, the integral equals $\ell^\star(x):=\int_1^x \ell(t)/t \,\mathrm dt$, which is slowly varying at infinity itself and $\ell^\star(t)/\ell(t)\to \infty$ as $t\to \infty$.  $W$ satisfies \eqref{powerlaw}, so we can bound the expectations in \eqref{eq:t1t2} using \eqref{Karamata}:
		\be\label{Karamatabound}
		\mathbb{E}\left[W^\alpha \mathbbm{1}_{\left\{W\leq S^d/w\right\}}\right] \leq  \int_1^{S^d/w}\alpha t^{\alpha-1}\ell\left(t\right)t^{-\left(\tau-1\right)}\mathrm{d}t\leq C_{\alpha,\tau}\left(S^d/w\right)^{\alpha-\tau+1}\ell\left(S^d/w\right),
		\ee	
		when $\alpha>\tau-1$, and $\E{W^\alpha \mathbbm{1}_{\left\{W\leq S^d/w\right\}}}\leq \E{W^\alpha}=:M_\alpha<\infty$ when $\alpha< \tau-1$. When $\al=\tau-1$, the integral in \eqref{Karamatabound} itself is slowly varying, thus, with $\ell^\star(x):=\int_1^x \ell(t)/t \,\mathrm dt$,  the bound in \eqref{Karamatabound} remains true with $\ell$ replaced by $\ell^\star$.	Further,
		\be\label{Nwmomentbound}
		\mathbb{E}\left[W\mathbbm{1}_{\left\{W\geq S^d/w\right\}}\right] \leq \int_{S^d/w}^\infty \ell\left(t\right)t^{-\left(\tau-1\right)}\mathrm{d}t \leq C_\tau\left(S^d/w\right)^{-\left(\tau-2\right)}\ell\left(S^d/w\right).
		\ee		
		Returning to \eqref{eq:t1t2}, with $C_{\al,\tau},C_\tau$ from \eqref{Karamatabound} and \eqref{Nwmomentbound}, with $\wit \ell(\cdot):=\ind_{\{\al<\tau-1\}}\ell(\cdot) + \ind_{\{\al=\tau-1\}}\ell^\star(\cdot)$,	\be\label{eq:t1t2-2}  T_1 +T_3\le \overline C_1\la\wit  c_d (C_{\al,\tau}+C_\tau) \cdot w^{\tau-1} S^{-d(\tau-2)} \wit \ell(S^d/w) +\ind_{\{\alpha<\tau-1\}} \overline C_1\la\wit  c_d M_\alpha w^\alpha S^{-\left(\alpha-1\right)d}.	\ee	
		Finally, we bound $T_2$ in \eqref{NwSupperbound}. With $\mathrm{Vol}_d$ the volume of the unit ball in $\R^d$, as in \eqref{eq:mu-ber}, the expectation within $T_2$ equals $\mathrm{Vol}_d \la \overline a_1^{1/\al} ws$, thus, by \eqref{Nwmomentbound} once more, 
		\be\label{Nwbound2}
		T_2\le \overline C_1 \la \overline a_1^{1/\al} w\Ev\big[W \mathbbm{1}_{\left\{W\ge S^d/w\right\}}\big]\leq \overline C_1 \la  \mathrm{Vol}_d  C_\tau \cdot w^{\tau-1 } S^{-d(\tau-2)} \ell(S^d/w).
		\ee
		Combining \eqref{eq:t1t2-2} and \eqref{Nwbound2} yields the desired bound in \eqref{boundsonNw} with $M_{d, \al, \tau}:=\overline C_1 (\wit c_d C_{\al,\tau}+ \wit c_d C_\tau+  \mathrm{Vol}_d C_\tau)$ and $M_1:=\overline C_1 \wit c_d M_\al$. The bound on $\cb\mathbb{E}\left[N_w\left(\geq t,\geq 0\right)\right]$ follows from the calculations already done in \eqref{eq:kernel-1}-\eqref{eq:sumterm-1},  that yield, when integrating over $s\ge t$, 		\be
		\mathbb{E}\left[N_w\left(\geq t,\geq 0\right)\right]\leq \overline C_1 \la (\mathrm{Vol_d}+\wit c_d) w \int_t^\infty F_W\left(\mathrm ds\right) s \mathrm ds
		\leq M_{d,\tau,\alpha}\lambda w t^{-\left(\tau-2\right)}\ell\left(t\right).
		\ee
		where we have applied the bound in \eqref{Nwmomentbound} by replacing $S^d/w$ by $t$.
	\end{proof}
	\section{Lower bound on the weighted distance: Proof of Proposition \ref{Prop:lowerbound}}\label{sec:maintheoremproof}
	In this section we prove  Proposition \ref{Prop:lowerbound}.
	We start with a preliminary statement about the minimal distance between vertices in a given sized box:
	\begin{claim}\label{cl:min-dis} Consider a box $B(H_n):=[x\pm H _n^{1/d}]^{d}$ in $\Xdn$. Then, in $\Elambda$, and in $\BGIRG$, for some $c_2>0$,
	\be\label{eq:minmax} \Pv\Big(\big\{\min_{v_1, v_2\in B(H_n)} \|v_1-v_2 \| \le 1/H_n^{2/d}\big\} \bigcup \big\{\max_{v\in B(H_n)}W_v \ge H_n^{2/(\tau-1)}\big\}\Big)\le c_2 H_n^{-3/4}. \ee 
	\end{claim} 
\begin{proof}
 To bound the minimal distance between two vertices in the box, one can use standard extreme value theory or simply a first moment method: given the number of vertices, the vertices are distributed uniformly within $H_n$, and hence the probability that there is a vertex within norm $z>0$ from a fixed vertex $v\in \CV_\la$, conditioned on $|\CV_\lambda \cap B(H_n)|$, is at most $|\CV_\lambda \cap B(H_n)| z^d \mathrm{Vol_d} /(2^dH_n)$, and so by Markov's inequality,
			\be \ba \P{\exists v_1,v_2: \|v_1-v_2\|\leq z} &\leq \Ev\big[\P{\exists v_1,v_2: \|v_1-v_2\|\leq z\mid |\CV_\lambda \cap B(H_n)| }\big]\\
			&\le \Ev\big[|\CV_\lambda \cap B(H_n)|^2 \cdot \mathrm{Vol}_d z^d/ (2^dH_n)\big] \\
			&\le  (\lambda^2 H_n^2+\la H_n) \mathrm{Vol}_d z^d/(2^dH_n),
		\ea\ee 
that is at most $\Theta( H_n^{-3/4})$ for $z=1/H_n^{2/d}$. For $\BGIRG$, the same argument works with $\la=1$.
\cb Bounding the maximal weight, we  use that $|B(H_n)\cap \mathcal{V}_\lambda|$ has distribution  $\mathrm{Poi}(\la 2^d H_n)$, and that each vertex has an i.i.d.\ weight with power-law distribution from \eqref{powerlaw}. So,  by the law of total probability, for any $x\le 0$:
		\be\label{upperboundweights}
		\mathbb{P}\Big(\max_{v\in B(H_n) \cap \CV_\la} W_v > x\Big)\leq \Ev\Big[\sum_{v\in B(H_n) \cap \CV_\la} \ind_{\{W_v >
			x\}}\Big]=\ell\left(x\right)x^{-\left(\tau-1\right)}2^d H_n\lambda.
		\ee
		Inserting  $x=H_n^{2/(\tau-1)}$ yields
		\be
		\mathbb{P}\Big(\max_{v\in B(M_n)\cap \CV_\la} W_v > \cb H_n^{2/(\tau-1)}\cz\Big) \leq \la \ell\left(H_n^{2/(\tau-1}\cz\right)H_n^{-2} 2^d H_n\le c H_n^{-3/4}
		\ee
		where we use that $\ell(H_n^{2/(\tau-1)}\cz)H_n^{-1/4}<1$ by  \eqref{Potter} for large $H_n$. Combining the two bounds we arrive at \eqref{eq:minmax}.
\end{proof}
	\begin{proof}[Proof of Proposition \ref{Prop:lowerbound}]
		Let $q\in\{1,2\}$ throughout the proof. We slightly abuse notation and write $v_n^q$ for the vertices as well as their location in $\R^d$. Recall the event $A_{k}$ \cb from \eqref{eq:ak}\cz. First we show that the inequality \eqref{lowerboundGIRG} holds on $A_{k_n}$: 
		Observe that on the event $\{B^G_{1+\xi_n}(v_n^1, k_n)\cap B^G_{1+\xi_n}(v_n^2, k_n)=\emptyset\}\subseteq A_{k_n}$,  any path connecting $v_n^1, v_n^2$ in $\Eplus$ must intersect the boundaries of these sets, thus the bound 
		\be\label{eq:1plus} d_L^{1+\xi_n}(v_n^1, v_n^2) \ge d_L^{1+\xi_n}(v_n^1, \cb\partial B^G_{1+\xi_n}(v_n^1, k_n)\cz) + d_L^{1+\xi_n}(v_n^1, \cb\partial B^G_{1+\xi_n}(v_n^2, k_n)\cz) \ee
		holds\footnote{The distance is per definition infinite when there is no connecting path. In that case the bound is trivial.}. Let us denote the shortest path connecting $v_n^q$ to $\cb\partial B^G_{1+\xi_n}(v_n^q, k_n)\cz$ in $\Eplus$ by $\cb\pi^q_\star(k_n)$. In general it can happen that the rhs of \eqref{eq:1plus} is not a lower bound\footnote{For instance, in $\BGIRG$ there might be an `extra-edge' from some vertex in $B^G_{1+\xi_n}(v_n^1, k_n)$ leading to $B^G_{1+\xi_n}(v_n^1, k_n)^c$ and the shortest path connecting $v_n^1, v_n^2$ in $\BGIRG$ uses this particular edge, avoiding $\cb\pi_\star^q(k_n)$. Or, there is an extra edge in $\BGIRG$ not present in $\CE_{1+\xi_n}$ that `shortcuts' $\cb\pi^q_\star(k_n)$ and makes an even shorter path present in $\BGIRG$.  In these cases, $d_L(v_n^1, v_n^2)$ might be shorter than the rhs in \eqref{eq:1plus}.} for $d_L(v_n^1, v_n^2)$ in $\BGIRG$.  However, on the event $A_{k_n}$, this cannot happen, since on $A_{k_n}$, the three graphs $B^G_{1+\xi_n}(v_n^q, k_n), B^G_{n}(v_n^1, k_n)$  $B^G_{1}(v_n^q, k_n)$ coincide. Thus, $\cb\pi_\star^q(k_n)$ is \emph{present} and also \emph{shortest} in $\BGIRG$ and $\Eone$ as well. 
		As a result, we arrive to \eqref{lowerboundGIRG}. Next we bound $\Pv(A_{k_n}^c)$ from above.		
		We recall that, by the coupling in Claim \ref{cl:vertex-containment}, $v_n^q\in \mathcal{V}_{1+\xi_n}$ for all $n$ sufficiently large. Recall $ \epsilon_{\mathrm{TV}}(n)$ from Assumption \ref{assu:weight} and $\epsilon(n),I_\Delta(n), I_w(n)$ from Assumption \ref{assu:extendable}.  \cb The last two expressions are intervals for the distance of vertices as well as for vertex-weights, such that the convergence in Assumption \ref{assu:extendable} holds when vertex distances and vertex-weights are in these intervals. Let us thus introduce the notation for the endpoints of these intervals as follows: 
	\be I_\Delta(n)=:[\Delta_{\min}(n), \Delta_{\max}(n)], \qquad I_w(n)=:[1,w_{\max}(n)]\ee and set: 
			\be\label{eq:mn} H_n:=\min\big\{ \Delta_{\min}(n)^{-d/2},  \Delta_{\max}(n)/2\sqrt{d}, w_{\max}(n)^{(\tau-1)/2}, \epsilon(n)^{-1/4}, \epsilon_{\mathrm{TV}}(n)^{-1/4}, n^{1/(4d)}   \big\}.\ee
					Note that, by the requirements  in Assumption \ref{assu:extendable}, $I_\Delta(n)\to \infty, I_w(n)\to [1, \infty]$ implies that $\Delta_{\min}(n)\to 0, \Delta_{\max}(n)\to \infty, w_{\max}(n)\to \infty$, and also note that $\epsilon(n), \epsilon_{\mathrm{TV}}(n) \to 0$. Hence, $H_n\to \infty$ and it will take the role of $H_n$ in Claim \ref{cl:min-dis} above. Let us now define the two boxes in $\Xdn$:
		\be\label{boxesexploration}\ba 
		\text{Box}\left(v_n^q\right)&:=\big[v_n^q-H_n^{1/d}, v_n^q +H_n^{1/d}\big]^d.\\
		\ea
		\ee

		\cb We divide the proof of bounding the probability of $A_{k_n}^c$ into the following steps: we first define two events, $E_n^{(1)},E_n^{(2)}$, which denote that $\mathrm{Box}(v_n^1)$ and $\mathrm{Box}(v^2_n)$ are disjoint and within $\widetilde{\mathcal{X}}_d(n)$, and that the vertices and edges within these boxes have vertex-weights and spatial distances between vertices within $I_w(n)$ and $I_\Delta(n)$, respectively. We show that these events hold whp. Then, we estimate the largest graph distance $k_n$ from $v_n^q$, such that $B^G_{1+\xi_n}(v^q_n,k_n)$ is still within $\text{Box}\left(v_n^q\right)$. By Lemma \ref{lemma:ballsbrw} and \eqref{eq:ball-containment}, generation $k$ in $\mathrm{BerBRW}_{1+\xi_n}(v^q_n)$ is a superset of vertices at graph distance $k$ from $v_n^q$. Hence, 
		 we determine the largest generation number $k_n$ such that generation $k_n$ in $\mathrm{BerBRW}_{1+\xi_n}(v^q_n),q\in\{1,2\}$ is still within $\mathrm{Box}(v^q_n)$.  
		 Proposition \ref{Prop:BerBRW}, bounds the maximal displacement of the generations of $\mathrm{BerBRW}_{1+\xi_n}(v^q_n)$ for \emph{all} $k\geq0$, which yields the definition of $k_n$.
		 
		  There is a technical problem. We intend to prove our results for \emph{two} branching random walks, within the boxes $\mathrm{Box}(v_n^q),q\in\{1,2\}$, whereas the growth estimate of one BerBRW in Proposition \ref{Prop:BerBRW} depends on the whole of $\R^d$, yielding a random prefactor $Y(\eps)$.
		   Hence, we redefine the random prefactor $Y(\eps)$, as used in Proposition \ref{Prop:BerBRW}. We cannot use this limiting $Y(\ve)$ for \emph{two} branching random walks, one started at $v_n^1$ and one at $v_n^2$, since they would not be independent.  To maintain independence, we slightly modify the growth variable $Y(\ve)$ in Proposition \ref{Prop:BerBRW} to be determined by the growth of the process within the box only. This way the two copies stay independent. 
		   		   Using these adjusted growth variables for the two BerBRWs, we find the sequence $k_n$, such that the first $k_n$ generations  of both BerBRWs ( started at $v_n^1,v_n^2$) are contained within $\mathrm{Box}(v_n^1),\mathrm{Box}(v_n^2)$, respectively. 
				   
				  Finally, having $k_n$, we show that whp the balls $B^G_n(v_n^q,k_n),B^G_{1+\xi_n}(v_n^q,k_n), B^G_1(v_n^q,k_n)$ have identical vertex sets, edge sets and vertex-weights. Combining all of the above, we show that $A_{k_n}$ holds whp. We now work out these steps in detail.\cz
			
		Let us define two events: one that these boxes are disjoint and are not intersecting the boundary of $\Xdn$, and two, that   all vertex-vertex distances and vertex-weights  within $\BGIRG$ and $\Eplus$ are within $I_\Delta(n), I_w(n)$, respectively. Recall that $\CV_B(n)\subseteq \CV_{1+\xi_n}$ from Claim \ref{cl:vertex-containment}. Then,
	\be\ba \label{en1}	E_{n}^{(1)}&:=\big\{ \text{Box}(v_n^1)\cap \text{Box}(v_n^2)=\emptyset, \text{Box}\left(v_n^q\right)\subseteq \Xdn, q\in\{1,2\}  \big\},\\
	E_{n}^{(2)}&:=\big\{\forall v\in\CV_{1+\xi_n} \cap (\text{Box}(v_n^1)\cup \text{Box}(v_n^2)): W_v^{(n)}\in I_w(n), W_v\in I_w(n),\big.\\
	&\qquad\quad\big. \forall u,v\in \CV_{1+\xi_n} \cap (\text{Box}(v_n^1)\cup \text{Box}(v_n^2)): \|u-v\|\in I_\Delta(n)  \big\}.\ea\ee
		Recall that the position of the vertices in $\BGIRG$ is uniform in $\Xdn$ and that $\CC_{\max}$ has linear size. As a result, when $v_n^1, v_n^2$ are uniformly chosen in $\CC_{\max}$, their location is uniform over a linear subset of $n$ i.i.d.\ position-vectors within $\Xdn$. Hence, $\Pv(E_{n}^{(1)}) =1-e_1(n)$ for some function $e_1(n)\downarrow 0$ as $n\to \infty$.  
		\cb Further, $E_n^{(2)}$ captures the event that no two vertices are too close or too far away from each other, and that their weight is not too large for the bound $\epsilon(n)$ in \eqref{eq:h-intro} in Assumption \ref{assu:extendable} to hold.  By Claim \ref{cl:min-dis}, compared to the the definition of $H_n$ in \eqref{eq:mn}, we see that with probability at least $1-c_2 H_n^{-3/4}$:
		\begin{enumerate}
		\item[1)] the minimal distance between vertices within $\text{Box}(v_n^q)$ is at most $H_n^{-2/d}\ge \Delta_{\min}(n)$, 
		\item[2)] the maximal distance is at most $\sqrt{d}2H_n\le \Delta_{\max}$ (on the diagonal),  
		\item[3)] the maximal weight of vertices is at most $H_n^{2/(\tau-2)}\le w_{\max}(n)$.
		\end{enumerate}
		These are precisely the requirements for $E_n^{(2)}$ to hold\cz, 
		hence $\Pv(E_n^{(2)})\ge 1-e_2(n)$ for $e_2(n):=c_2 H_n^{-3/4}\downarrow 0$ by Claim \ref{cl:min-dis}. 	
					
		\cb Next we turn to taking care of the independence issue and determining $k_n$\cz. On $E_n^{(1)}\cap E_n^{(2)}$, we investigate $\text{BerBRW}_{1+\xi_n}(v_n^q), q\in\{1,2\}$. The two BRWs are independent as long as they are contained in the disjoint boxes $\text{Box}\left(v_n^q\right)$, since the location of points in two disjoint sets in a PPP are independent, and the vertex-weights of different vertices are i.i.d.\ copies of $W$, and edges are present conditionally independently. 
		Recall that $\CG^{\text{Ber}}_{1+\xi_n}(v_n^q,k)$ stands for the set of vertices in generation $k$ of $\mathrm{BerBRW}_{1+\xi_n}(v_n^q)$. Our goal is to find $k_n$ so that the following event holds:
		\be\label{containinbox}
		\cb E_{n, k_n}:=\cz\Big\{\cb \bigcap_{k\leq k_n}\{\CG^{\text{Ber}}_{1+\xi_n}(v_n^q,k) \subseteq \text{Box}\left(v_n^q\right)\}\cz, \quad q\in \{1,2\}\Big\}.
		\ee
		On $E_{n,k_n}$, Lemma \ref{lemma:ballsbrw} can be extended to hold for both BRWs until \cb generation $k_n$\cz, with independent environments. A combination of Claim \ref{claim:edge-containment} and Lemma \ref{lemma:ballsbrw},  as in \eqref{eq:ball-containment},  yields that \cb on $E_{n,k_n}$\cz, 
		\be\label{event:inbox}  B^G_{1+\xi_n}(v_n^q,k_n) \subseteq \text{Box}\left(v_n^q\right), \text{ and } \cb B^G_{1+\xi_n}(v_n^q,k_n)\subseteq \Xdn,\quad q\in\{1,2\}, \ee
		see Figure \ref{Figure:lowerbound}, thus these sets with index $1,2$ are disjoint on $E_n^{(1)}$. 
		Next we determine $\cb k_n$ such that $E_{n,k_n}$ holds whp. We would like to apply Proposition \ref{Prop:BerBRW} to say that the maximal displacement in the two BRWs grow at most doubly exponentially, as in \eqref{BerBRWpropstatement}. The random variable prefactor $Y(\ve)$, (the double exponential growth rate of the generations) in Proposition \ref{Prop:BerBRW} bounds the growth of a BRW for \emph{all} generations $k\ge 0$. The problem with this variable is that it depends on the entire environment $\CV_\la$.  Now that we have two BRWs that only stay independent as long as they are contained in their respective boxes $\text{Box}(v_n^q)$, we would like to have a similar definition of doubly-exponential growth based on only observing the generations within $\text{Box}(v_n^q)$ so that we can maintain independence.  For this, we modify the definition of $Y(\eps)$, as in Remark \ref{rem:Ydef}, (the double exponential growth rate of the generations), to reflect this. Recall $c_k(\ve,i), S_k(\ve,i)$ from \eqref{doubleexp}. We define the event 
		\be\label{eq:growth-ei}
		\wt{E}_k^{\left(q\right)}(\ve,i):=\Big\{Z^{\mathrm{Ber}}_{1+\xi_n}\left(v_n^q,k\right)\leq c_k\left(\eps,i\right)\Big\}\cap \left\{\mathcal{G}^{\text{Ber}}_{1+\xi_n}\left(v_n^q,k\right)\cap \big(B^2_{S_k\left(\eps,i\right)}\left(v_n^q\right)\big)^c=\empty \right\}.
		\ee
		Let $K(\ve, i)$ be the largest integer $k$ such that $S_k(\eps,i)\leq H_n^{1/d}$, then, by \eqref{doubleexp},
		\be\label{boundk}
		K\left(\eps,i\right)=\Bigg\lfloor \frac{\log\log H_n -\log i-\log \left(d \zeta\right)}{\log\left(\left(1+\eps\right)/\left(\tau-2\right)\right)}\Bigg\rfloor.
		\ee
		Thus,  the event that the generations sizes and displacement of $\mathrm{BerBRW}_\la (v_n^q)$ are not growing faster than doubly exponentially with prefactor $i$ for the first $K(\ve,i)$ generations, that is, as long as the process stays within $\text{Box}(v_n^q)$, see \eqref{boxesexploration}, can be expressed as follows:
		\be
		\wt{H}^{\left(q\right)}_i:=\bigcap_{k=0}^{K\left(\eps,i\right)}\wt{E}^{\left(q\right)}_k\left(\eps,i\right).
		\ee
 Clearly, $H_i \subseteq \wit H_i^{(q)}$ from \eqref{def:yeps}, so, by Remark \ref{rem:Ydef}, 
		\be
		\sum_{i=1}^\infty \Pv\big((\wt{H}^{\left(q\right)}_i)^c\big)\leq \sum_{i=1}^\infty \P{H_i^c}<\infty.
		\ee
		By the Borel-Cantelli lemma, we define $\wt{Y}_q(\eps)$ as the first index $i$ from which point on no $(\wt{H}^{(q)}_i)^c$ occur any more. Then, $\wt{Y}_1(\eps),\wt{Y}_2(\eps)$ are independent, since they are determined on disjoint sets of vertices.
		By Remark \ref{rem:Ydef}, the rhs of the tail estimate \eqref{Ytail} equals $\sum_{i\ge K} \Pv(H_i^c)$, so, again by  $H_i \subseteq \wit H_i^{(q)}$,  \eqref{Ytail} is valid for $\Pv(\wt{Y}_q(\eps) \ge K)$ as well.
		 Let $\omega(n)\to \infty$ be defined as
		\be\label{eq:omegan}	\omega(n):=(\log H_n )^{1/2}.
		\ee
		Then, the event
		\be
		E_{n}^{(3)}:=\big\{ \wt{Y}_q(\eps)\leq \omega(n)(\tau-2)/((1+\ve)(d\zeta)), q\in\{1,2\} \big\}\ee
		holds with probability at least $1-e_3(n)$ with $e_3(n)\downarrow 0$.
		An elementary calculation using \eqref{boundk} and the upper bound on $\wit Y_q(\ve)$ in $E_n^{(3)}$ yields that on $E_{n}^{(3)}$, 
		\be\label{wtKn}
		K(\eps,\wt{Y}_q\left(\eps\right))\geq \frac{\log\log H_n}{2\log\left(\left(1+\eps\right)/\left(\tau-2\right)\right)}=:k_n.
		\ee
		Combining this with the definition of  $\wt{Y}_q\left(\eps\right)$ and $\wt H_i^{(q)}$, \eqref{eq:growth-ei} holds for the first $\cb k_n$ generations. Recall \eqref{boxesexploration} and \eqref{containinbox}. Thus, on $E_{n}^{(1)}\cap E_{n}^{(3)}$, both $\cb E_{n, k_n}$  and \eqref{event:inbox} hold with $\cb k_n$ as in \eqref{wtKn}. 		
		To move towards the event $A_{k_n}$ \cb as in \eqref{eq:ak}\cz, 
		we need to compare the vertex and edge-sets of $B_n^G(v_n^q, k_n), B_{1+\xi_n}^G(v_n^q, k_n), B_1^G(v_n^q, k_n)$ and show that they are whp identical.
		Let
		\be E_n^{(4)}:=\{ |\CV_{1+\xi_n}\cap \Xdn| \le (1+2\xi_n) n,  |\CV_{1+\xi_n}\cap \text{Box}(v_n^q)| \le (1+2\xi_n) 2^d H_n, q\in\{1,2\}\},\ee
		\cb be the event that there are not too many vertices in the boxes $\text{Box}(v_n^q)$ and in the whole $\Xdn$ that the Poisson density $1+\xi_n$ allows. Then, by standard calculation on Poisson variables\cz, $\Pv(E_n^{(4)}) = 1-e_4(n)$ with $e_4(n)\downarrow 0$. 
		  By the \emph{proof} of Claim \ref{cl:vertex-containment}, each vertex in $\CV_{1+\xi_n}$ is not in $\CV_B(n)$ with probability $1-n/|\CV_{1+\xi_n}|$, which is at most $2\xi_n$ on $E_n^{(4)}$, and each vertex in $\CV_{1+\xi_n}$ is not in $\CV_1$ with probability $\xi_n/(1+\xi_n)<\xi_n$. Further, on the second event in $E_n^{(4)}$, the number of vertices and possible edges in $\Eupper$ within $\text{Box}(v_n^q)$ is at most $(1+2\xi_n)2^d H_n$ and $(1+2\xi_n)^24^d H_n^2$, respectively. Since on $E_n^{(1)}\cap E_n^{(3)}$, \eqref{event:inbox} holds, on $E_n^{(1)}\cap E_n^{(3)}\cap E_n^{(4)}$, it is an upper bound to say that in $\cb \overline B_{1+\xi_n}^G(v_n^q, k_n)$, the number of vertices and edges until graph distance $k_n$ is at most
		\be\label{eq:vertex-est} |\CV(\overline B_{1+\xi_n}^G(v_n^q, \cb k_n\cz))| \le  (1+2\xi_n)2^d H_n, \qquad
		|\CE(\overline B_{1+\xi_n}^G(v_n^q,\cb k_n\cz))| \le  (1+2\xi_n)^24^d H_n^2.\ee 
		Let $\mathcal W(A)$ denote the weights of vertices within a set $A$, and let
		\be\label{eq:e4} \ba E^{(5)}_{n,k}:=\bigcup_{q\in\{1,2\}}\big\{& \CV(B_n^G(v_n^q, k))=\CV(B_{1+\xi_n}^G(v_n^q, k))=\CV( B_1^G(v_n^q, k)), \big.\\
		&\CE(B_n^G(v_n^q, k))=\CE(B_{1+\xi_n}^G(v_n^q, k))=\CE( B_1^G(v_n^q, k)),
		\\
		& \mathcal W(B_n^G(v_n^q, k))=\mathcal W(B_{1+\xi_n}^G(v_n^q, k))=\mathcal W( B_1^G(v_n^q, k)) \big\}.\ea \ee
By the coupling in Claim \ref{claim:edge-containment}, an edge in $\CE_{1+\xi_n}$ is in $\CE_1$ whenever both end vertices are in $\CV_1$. Further, on $E_n^{(2)}$ all weights and vertex-vertex distances are within $I_\Delta(n), I_w(n)$, and hence, conditioned that $e$ is present in $\overline\CE_{1+\xi_n}$, the probability that the presence of an edge in $\CE_{1+\xi_n}, \CE_1$ differs from that in $\CE_B(n)$ is at most $\epsilon(n)+2\epsilon_{\mathrm{TV}}(n)$. This coupling also contains that the vertex-weights are identical.
Finally,  by \eqref{eq:ball-containment},  the vertex sets within graph distance $k$ are contained in the set of individuals within generation $k$ of the corresponding BRW. Using the bound in \eqref{eq:vertex-est} on the number of vertices and potential edges, and \eqref{eq:mn},	and $\xi_n=\sqrt{\log n/n}$, 	\be \ba \Pv\big((E^{(5)}_{n,k_n})^c \mid E_n^{(4)}\cap E_n^{(3)}\cap E_n^{(2)} \cap E_n^{(1)}\big) &\le 6 \xi_n \cdot (1+2\xi_n)2^d H_n\\& \quad+  (\epsilon(n)+2\epsilon_{\mathrm{TV}}(n)) \cdot (1+2\xi_n)^24^d H_n^2 =:e_5(n)\downarrow 0, \ea\ee
\cb where the convergence of $e_5(n)$ to $0$ with $n$ is implied by the definition of $H_n$ in \eqref{eq:mn}\cz. Finally, with $k_n$ as in \eqref{wtKn}, and $A_{k_n}$ \cb as in \eqref{eq:ak}\cz,
$E_n^{(1)}\cap E_n^{(2)}\cap E_n^{(3)}\cap E_n^{(4)}\cap E_{n,k_n}^{(5)}\subseteq A_{k_n}$ and hence
		\be\label{eq:an-bound1} \Pv(A_{k_n}) \ge 1-(e_1(n)+e_2(n)+e_3(n)+e_4(n)+e_5(n))\to 1.\ee This finishes the proof.	\end{proof}
	
	\section{Percolation and boxing method on GIRGs}\label{sec:boxinginfinitecomponent}
	In this section, we develop the necessary tools to prove  Propositions \ref{Prop:existinfcomp} and \ref{Prop:upperbound}.
	First, we define  \emph{weight-dependent percolation} on GIRGs and relate percolated GIRG graphs to GIRGs with modified parameters, a useful tool in  bounding the length of the connecting path  in Proposition \ref{Prop:upperbound}. Then, we define a \emph{boxing method} that enables to build this connecting path as well the path to infinity in Proposition \ref{Prop:existinfcomp}.  
	\cb We start by defining percolation on GIRGs\cz.  
	\begin{definition}[Weight-dependent percolation] \label{def:percolationprob} Consider $\BGIRG$ or $\Elambda$, and a function \cb $p:\left[1,\infty\right)^2\rightarrow\left(0,1\right]$\cz. Conditioned on the weights of vertices, keep every existing edge $e=\left(u,v\right)$ connecting vertices with weights $w_u, w_v$ independently with probability $\cb p\left(w_u,w_v\right)$. Let us denote the percolated graph (the graph formed by the kept edges) by $G^p$\cz. 
	\end{definition}
	Recall that on $\BGIRG$, each edge carries an i.i.d.\ edge-lengths $L_e$. Given any threshold function $\mathrm{thr}:\left[1,\infty\right)^2\rightarrow\left(0,1\right]$, by setting
	\be\label{percolationprob}
	\ind\{\text{edge $e=\left(u,v\right)$ is kept}\}=\ind\{\cb L_e\le \mathrm{thr}\left(w_u,w_v\right)\cz\},
	\ee
	yields a weight dependent percolation. Indeed, since the $L_e$ are i.i.d., conditioned on the weights of vertices, edges are kept independently with probability
	\[ F_L(\mathrm{thr}\left(w_u,w_v\right)) =: p(w_u,w_v), \]
	\cb which defines the function $p$ in Definition \ref{def:percolationprob}\cz. The next claim connects the percolated graph $\cb G^{\mathrm{p}}$ to an instance of a GIRG with modified parameters. 
	
	\begin{claim}\label{claim:percolation}\cb
		Consider the $\BGIRG$, as in Definition \ref{def:GIRG}. Assume that $\BGIRG$ satisfies Assumptions \ref{assu:GIRGgen}, Assumption \ref{assu:weight} and \ref{assu:extendable} with $\al>1,d\geq1,\tau>1$, and, for some $c\in(0,\al)$ and $\wit\gamma\in (0,1)$, the percolation function $p$ satisfies for all $w_u,w_v$ that
		\be\label{eq:pw} p(w_u, w_v)\ge\exp\{-c (\log w_u)^{\wt \gamma}-c (\log w_v)^{\wt \gamma}\}.\ee
		Then the percolated graph $\cb G^{\mathrm{p}}$ (without the edge-lengths) contains a subgraph $\cb \underline G^{\mathrm{p}}$ that is an instance of  a $\cb \mathrm{BGIRG}$, and the parameters of $\underline G^{\mathrm{p}}$ as a BGIRG again satisfy Assumptions \ref{assu:GIRGgen}, \ref{assu:weight} and \ref{assu:extendable} with unchanged parameters $\cb(\alpha^{\mathrm{p}}, d^{\mathrm{p}}, \tau^{\mathrm{p}})=(\alpha,d,\tau)$. The results remain valid when we percolate $\Elambda$ instead.\cz
	\end{claim} 
	
	\begin{proof} \cb Recall $g_n,  g_n^{\mathrm B}$ from Definition \ref{def:BGIRG} and  from Assumption \ref{assu:extendable}. Since edge-lengths are  independent, conditioned on $(x_i, W_i^{(n)})_{i\in [n]}$, each edge $e=(u,v)$ is present in $\cb G^{\mathrm{p}}$ \emph{independently} with probability
		\be\ba\label{eq:perc-prob} \Pv\big((u,v)&\in \CE(\cb G^{\mathrm{p}}\cz) \mid (x_i, W_i^{(n)})_{i\in [n]} \big) \\
		&=  g_n^{\mathrm B}\big(\cb x_u,x_v\cz, (W_i^{(n)})_{i\in [n]}\big)\cb p(W_u^{(n)},W_v^{(n)})\cz\\
		&\ge  c_1\Big(\e^{-a_2(\log W_u^{(n)})^\gamma-a_2(\log W_v^{(n)})^\gamma}\wedge \underline a_1 \|x_u-x_v\|^{-\alpha d}\left(W_u^{(n)}W_v^{(n)}\right)^\alpha\Big) \e^{-c(\log W_u^{(n)})^{\wit \gamma}-c(\log W_v^{(n)})^{\wit\gamma}}\\
		&=:j(x_u, x_v, W_u^{(n)},W_v^{(n)}) \cz 
		\ea \ee
		where we used that the connection probability of $\BGIRG$ satisfies the lower bound coming from \eqref{assu:BGIRG-prob} and \eqref{eq:gunder}, as well as the lower bound on $p$ in \eqref{eq:pw}. 
		Observe that formally the third row equals
		\be\label{eq:jjj}\ba  j(x_u, x_v, W_u^{(n)},W_v^{(n)})&=c_1\Big(\e^{-a_2(\log W_u^{(n)})^\gamma-c(\log W_u^{(n)})^{\wit \gamma}-a_2(\log W_v^{(n)})^\gamma -c(\log W_u^{(n)})^{\wit\gamma}}\Big. \\
		&\ \ \Big.\wedge \underline a_1 \|x_u-x_v\|^{-\alpha d}\left(W_u^{(n)} \e^{-(c/\al)(\log W_u^{(n)})^{\wit \gamma}}  W_v^{(n)} \e^{-(c/\al)(\log W_v^{(n)})^{\wit \gamma}}\right)^\alpha\Big).\ea\ee
As a result, with $m(w):=w \exp\{ -(c/\alpha) (\log w)^{\wt \gamma}\}$, it is natural to define the vertex-weights for $\underline G^{\mathrm p}$ as  
		 \be\label{eq:new-weights} W^{(n)}_{\mathrm {p},i}:= m(W_i^{(n)})=W_i^{(n)} \exp\{ -(c/\alpha) (\log W_i^{(n)})^{\wt \gamma}\}, \ee
		 and observe that the collection $(W_{\mathrm {p},i}^{(n)})_{i\in[n]}$ is also i.i.d. In what follows we bound the function $j(x_u, x_v, W_u^{(n)}, W_v^{(n)})$ from below by a function $\underline g^{\mathrm{p}}(x_u,x_v, m(W_u^{(n)}), m(W_v^{(n)}) )$, and then we will define $\underline G^{\mathrm{p}}$ as a BGIRG with weights from \eqref{eq:new-weights}, and connection probabilities given by $\underline g^{\mathrm{p}}$. To this end, we need to express the first line in $\eqref{eq:jjj}$ as a function of $m(W_u^{(n)}), m(W_v^{(n)})$ instead of $W_u^{(n)}, W_v^{(n)}$. 
		 Since $c\in (0,\al)$, for all $w\ge 1$,
		\be\log (m(w)) = \log (w) (1 - (c/\al) (\log w)^{\wit \gamma-1}) \ge (1-c/\al)\log w, \ee
		and so we can bound, $\gamma_{\max}:=\max\{\gamma,\wt{\gamma}\}$,
		\[ \ba \exp\big\{-a_2(\log w)^\gamma-c(\log w)^{\wt{\gamma}}\big\}&\geq \exp\{-(a_2+c)(\log w)^{\gamma_{\max}}\}
		\\
		&\ge \exp\{-(a_2+c)(1-c/\al)^{-\widehat \gamma}(\log m(w))^{\gamma_{\max}}\}. \ea\]
		Hence, with $\widehat c:=(a_2+c)(1-c/\al)^{-\widehat \gamma}$,
		\be \ba  j(x_u, x_v, W_u^{(n)},W_v^{(n)}) &\ge  \e^{-\widehat c(\log m(W_u^{(n)}))^{\gamma_{\max}} -\widehat c(\log m(W_v^{(n)}))^{\gamma_{\max}}}
		\wedge \underline a_1 \|x_u-x_v\|^{-\alpha d}\big(m(W_u^{(n)}) m(W_v^{(n)}) \big)^{\alpha} \\
		&=:\underline g^{\mathrm p}(x_u, x_v, m(W_u^{(n)}), m(W_v^{(n)}) ).\ea\ee		
		Let us hence define $\underline G^{\mathrm p}$ as the graph where the vertex-weights are i.i.d. given by \eqref{eq:new-weights} and the connection probabilities are given by $\underline g^{\mathrm p}(x_u, x_v, m(W_u^{(n)}), m(W_v^{(n)}) )$, in a coupled way to $G^{\mathrm p}$, i.e. additionally thinning\footnote{This can be achieved by allocating an i.i.d. uniform $[0,1]$ variable to each edge and making the decision of retention based on this variable.} some edges to obtain the right retention probabilities.
\cb Due to the conditional independence of the presence of edges, this model satisfies Definition \ref{def:GIRG}, and the new connection probability function $\underline g^{\mathrm p}$ satisfies the bounds in Assumption \ref{assu:GIRGgen} with the original $\al$. It is elementary to check that $\underline  g_n^{\mathrm{p}}$ satisfies Assumption \ref{assu:extendable} since it is independent of $n$ and hence its limit is itself.
		It is left to show that the distribution of $\cb m(W^{(n)})$ satisfies Assumption \ref{assu:weight} with $\cb\tau^{\mathrm{p}}\cz=\tau$. \cb By Assumption \ref{assu:weight}, $W^{(n)}$ converges in distribution to a limiting variable $W$. Since the map $m(\cdot)$ is continuous, $m(W^{(n)})$ converges in distribution to $m(W)$\cz. First we show that \eqref{powerlaw} is still satisfied for this \emph{limiting distribution} $\cb W_{\mathrm{p}}\cz:=m(W)$ of $\cb W^{(n)}_{\mathrm{p}}$.
		Let  $m^{(-1)}$ denote the inverse function\footnote{The function $m$ is monotone on $[x_0, \infty)$, for some $x_0>0$, thus its inverse can be defined on this interval, which is sufficient since the assumption \ref{assu:weight} is concerned with the tail behaviour of the weight distribution.} of $m(w)$. Note that $m(w)$ varies regularly \footnote{A function $z(x)$ varies regularly at infinity with index $\vr$ when $\lim_{x\to \infty} z(cx)/z(x)=c^\vr$ for all $c>0$.} at infinity with index $1$. As a result of \cite[Theorem 1.5.12]{BinGolTeu89}, $m^{(-1)}$ varies regularly at infinity again with index $1$ and hence can be written as $m^{(-1)}(w):=\ell_2(w) w$ for some function $\ell_2(w)$ that varies slowly at infinity.  	
		Using that \cb $W_{\mathrm{p}}\cz=m(W)$ and \eqref{powerlaw},
		\be 
		\cb\P{W_{\mathrm{p}}> x}\cz =  \Pv (W> m^{(-1)}(x))= \ell(m^{(-1)}(x))\cdot  (m^{(-1)}(x))^{-(\tau-1)} = \ell( \ell_2(x) x) \ell_2(x)^{-(\tau-1)} x^{-(\tau-1)}.
		\ee 
		\cz Let  $\ell^\star(x):= \ell( \ell_2(x) x) \ell_2(x)^{-(\tau-1)} $. By \cite[Proposition 1.5.7]{BinGolTeu89}, it follows that $\ell^\star$ is slowly varying at infinity itself. So, the limiting distribution $W_\mathrm{p}$ of the weights $\cb(W_{\mathrm{p},i}^{\mathrm{(n)}})_{i\in[n]}$ is power-law distributed with parameter $\cb\tau^{\mathrm{p}}\cz=\tau$. Likewise, for the weights $\cb W_{\mathrm{p}}^{(n)}$ themselves, using \eqref{powerlaw-n},
		\be
		\cb\mathbb{P}(W_{\mathrm{p}}^{(n)}>x)\cz=\mathbb{P}(W^{(n)}>m^{(-1)}(x))=\ell^{(n)}(\ell_2(x)x)\ell_2(x)^{-(\tau-1)}x^{-(\tau-1)}=:\cb\ell_{\mathrm{p}}^{(n)}\cz(x)x^{-(\tau-1)}.
		\ee  
		Again using \eqref{powerlaw-n}, we can formulate a lower and an upper bound for $\cb \ell_{\mathrm{p}}^{(n)}$:
		\be 
		\underline \ell^{\star}(x):=\underline \ell (\ell_2(x)x)\ell_2(x)^{-(\tau-1)}\leq \cb\ell_{\mathrm{p}}^{(n)}\cz(x)\leq \overline \ell (\ell_2(x)x)\ell_2(x)=:\overline \ell^{\star}(x).
		\ee 
		By the same arguments as above, $\underline \ell^\star,\overline \ell^\star$ vary slowly at infinity. Since $W^{(n)}\geq \cb W^{(n)}_{\mathrm{p}}$ a.s., 
		\be 
		\cb \mathbb{P}(W_{\mathrm{p}}^{(n)}>M_n)\cz\leq \mathbb{P}(W^{(n)}>M_n)=o(n^{-1}).
		\ee 
		\cb Finally, $m(1)=1$ hence $W_{\mathrm{p}}^{(n)}\ge 1$ also holds\cz. 
		Hence, the weights $\cb(W_{\mathrm{p},i}^{(n)})_{i\in[n]}$ satisfy Assumption \ref{assu:weight}.
		The modification to $\Elambda$ is analogous by working with $h$ from \eqref{eq:EGIRG-prob} instead of $\wit g_n$.
	\end{proof}
	In the rest of the section we introduce a boxing structure that will be useful in establishing the connecting path in Proposition \ref{Prop:upperbound} as well as the existence of $\CC^\infty_\la$ in Proposition \ref{Prop:existinfcomp}. For constants $C,D>1$, to be defined shortly, and a parameter $\mu$, let us define the following:
	\be\label{radii}
	D_k:=\mu^{DC^k/d},\qquad R_k:=\mu^{C^k/d},
	\ee
	where $\mu$ can tend to infinity. We define a \emph{boxing system} centred at $u\in \R^d$, by defining for $k\geq0$,
	\be\label{Boxing}
	\ba
	\text{Box}_k\left(u\right):=&\left\{x\in\mathbb{R}^d\;:\;\|x-u\|_\infty\leq D_k/2\right\},\\
	\Gamma_k\left(u\right):=&\text{Box}_k\left(u\right)\backslash \text{Box}_{k-1}\left(u\right),\ k\geq1,\qquad \Gamma_0\left(u\right):=\text{Box}_0\left(u\right),\\
	B_k\left(z\right):=&\left\{x\in\mathbb{R}^d \;:\; \|x-z\|_\infty\leq R_k/2\right\}.
	\ea
	\ee
	Let us further set $k_{\max}=k_{\max}(\mu,n):=\max\{k\in\mathbb{N}:\ D^d_k\leq n\}$. We `tile' the annuli $\Gamma_k\left(u\right)$ with disjoint boxes of radius $R_k$, that we call subboxes, in a natural way. We order the subboxes  arbitrarily and refer to the $i^{\text{th}}$ subbox in $\Gamma_k(u)$ as $\bki$ and denote the number of subboxes in $\Gamma_k\left(u\right)$ by $b_k$. For sufficiently large $\mu$, by the ratio of volumes it can be shown that
	\be\label{boundbk}
	\mu^{\left(D-1\right)C^k}/2\leq b_k\leq \mu^{\left(D-1\right)C^k}.
	\ee
	Within each subbox $\bki$, let $c_k^{\sss\left(i\right)}$ be the vertex with the highest weight, i.e. $\cki:=\argmax_{v\in\bki}\left\{W_v\right\}$, that we call the  `centre' of $\bki$. The idea is to find a path of centres, one in each annulus, with a total $L$-length that is small, as displayed in Figures \ref{Fig:boxing} and \ref{Figure:overviewupperbound}. 
	Let us fix a small $\eps>0$ and define
	
	\be\label{CD}
	\ba
	\delta(\ve)&:=\left(\tau-2\right)\eps/\left(2\left(\tau-1\right)\right), \qquad
	C\left(\eps\right):=  (1-\eps)/(\tau-2),\\ 
	D\left(\eps\right)&:=  (1-\delta(\ve))\frac{1-\ve/(\tau-1)}{1-\ve}-\delta(\ve)/2.
	\ea
	\ee
	It is elementary to show that $D(\ve)\ge 1$. In the following lemma, we give  bounds on the weight of centres, and show that centres are connected to many centres in the next annulus. 
	\begin{lemma}[Weights and subgraph of centres]\label{lemma:weightconnectioncenters}
		Consider $\Elambda$ (resp. $\BGIRG$) with parameters $d\geq1,\ \tau\in (2,3),\ \alpha>1$. Let $u\in\mathbb{R}^d$, and consider the boxing with parameters $\mu>1$, fixed small $\eps>0$, and let $C:=C\left(\eps\right),D:=D\left(\eps\right), \delta:=\delta(\ve)$.
		Define  $N_{j}(\cki)$ as the number of centres in  $\Gamma_{j}(u)$ that are connected to $\cki$  and define the events 
		\be\ba\label{boundsweight}
		F_k^{(1)}&:=\left\{\forall i \le b_k: W_{\cki}\in\left[\cb\mu^{C^k(1-\delta)/(\tau-1) },\mu^{ C^k(1+\delta)/(\tau-1)}\cz\right] \right\},\\
		F_k^{(2)}&:=\left\{ \forall i \le b_k: N_{k+1}(\cki), N_k(\cki) \ge \exp\{ C^k (D-1)(\log \mu) /2  \} /2 \right\}.
		\ea\ee
		Then, for some $c_3(\ve)>0$, whenever $\mu>\mu_0(\ve)$,
		\be\label{boundsweight-error}
		1- \Pv\big(\cap_{k\ge 0} (F_k^{(1)} \cap F_k^{(2)})\big) \le  c_3( \mu^{-\delta/4}+ \exp\{-\lambda \mu^{\delta/2}/2\} + \exp\{ -\mu^{(D-1)C/2}/32\}).
		\ee
		Finally, for a vertex $v$ in $\mathrm{Box}_0(u)$, conditioned on $\cap_{k\ge 0}F_k^{(1)}  \cap \{W_{v}=\mu\}$,  $v$ is connected to at least one centre in annulus $\Gamma_1(v)$ with probability at least $1- \exp\{ -\mu^{(D-1)C/2}/16\}$.
	\end{lemma}
\cb The statements in Lemma \ref{lemma:weightconnectioncenters} are quenched, in the sense that they hold for almost all realisations of $\mathrm{EGIRG}_{W,L}(\lambda)$ or $\mathrm{BGIRG}_{W,L}(n)$, depending in which model you apply the lemma. The message of the lemma is the following: a centre of a subbox is connected to many centres of subboxes in the next annulus. By following paths that jump from centres to centres in annuli with increasing size, it is possible to  create paths going to infinity. This result is an important factor in various proofs: it shows that there is an infinite component in $\Elambda$. Moreover, since the degrees of centres grow doubly exponentially, it is even possible to construct paths going to infinity in a way that the total edge-length of the constructed path will be small (by always choosing the shortest option), which leads to connecting the shortest paths between two vertices. The idea of the proof of Lemma \ref{lemma:weightconnectioncenters} is that the boxes and subboxes are defined such that the weights of centres are sufficiently large and they are sufficiently close to the centres in consecutive annuli, and therefore the connection probability between these centres is fairly large. More specifically, it can be bounded by the first argument of $\underline g$, as in \eqref{eq:gunder}. \cz
\begin{proof}[Proof of Proposition \ref{Prop:existinfcomp} subject to Lemma \ref{lemma:weightconnectioncenters}]
		Using the boxing structure \cb as defined in \eqref{Boxing} \cz and the connectivity argument in Lemma \ref{lemma:weightconnectioncenters}, we prove the existence of an infinite component. Fix a sequence $(\mu_1<  \mu_2<\dots)$ so that the error bound in \eqref{boundsweight-error} is summable for the sequence $\mu_i$.
		Construct a boxing system centred at the origin with parameter $\mu_i$, for each $i\ge 1$. Since the error probabilities are summable, the Borel-Cantelli Lemma implies that there is an index $i_0$ for which $\cap_{k\ge 0} (F_k^{(1)} \cap F_k^{(2)})$ will hold with parameter $\mu_{i_0}$. 
		Thus, in this boxing system, an infinite component exists since we can simply find an infinite path along a sequence of centres $c_1^{(i_1)}, c_2^{(i_2)}, \dots,$ in annuli with increasing index. This component is almost surely unique by \cite[Theorem 6.3]{MeeRoy96}, which proves Proposition \ref{Prop:existinfcomp}.
	\end{proof}
	
	\begin{proof}[Proof of Lemma \ref{lemma:weightconnectioncenters}]
		For this proof we make use of Potter's theorem \cite[Theorem 1.5.6]{BinGolTeu89} that states that for any $\eps>0$,
		\be\label{Potter}
		\lim_{x\rightarrow \infty}\ell\left(x\right)x^{-\eps}=0,\ \ \lim_{x\rightarrow\infty}\ell\left(x\right)x^\eps=\infty.
		\ee
		We prove the statements in Lemma \ref{lemma:weightconnectioncenters} for the $\Elambda$ model and discuss how to adjust the proof for the $\BGIRG$ model. We start by estimating  the weights of the centres. \cb Using that the vertices are generated by a PPP with intensity $\lambda$, $\mathrm{Vol}(\bki)=R_k^d$, the power-law probability in \eqref{powerlaw}, and conditioning on $|B_k^{(i)}\cap \mathcal{V}_\lambda|$ yields\cz
		\begin{equation}\ba \label{eq:lowerboundweights2}
		\mathbb{P}\Big(\max_{v\in \bki \cap \CV_\la} W_v \leq x \Big) &= \sum_{j=0}^\infty  \mathbb{P}(W\leq x )^j \mathrm{e}^{-R_k^d\lambda}\left(R_k^d\lambda\right)^j\!/j! \\
		&=  \exp\left\{-R_k^d\lambda+ R_k^d \lambda (1-\ell(x)x^{-(\tau-1)}\right\}\\
		&=\exp\left\{-R_k^d \lambda \ell(x)x^{-(\tau-1)}\right\}.
		\ea\end{equation}
		Inserting the lower bound in $F_k^{(1)}$ in \eqref{boundsweight} for $x$ and using the definition of $R_k$ from \eqref{radii}, \cb we find that $R_k^d x^{-(\tau-1)}=\mu^{\delta C^k}$ and hence\cz
		\be\label{eq:weight-below}
		\ba
		\mathbb{P}\Big(\max_{v\in \bki \cap \CV_\la} W_v  \leq \cb\mu^{C^k(1-\delta)/(\tau-1) }\cz \Big) & \leq \exp\left\{-\lambda\ell\big(\cb\mu^{C^k(1-\delta)/(\tau-1) }\cz\big)\mu^{\delta C^k} \right\}
		\leq \exp\left\{-\lambda \mu^{(\delta/2) C^k} \right\},
		\ea
		\ee
		where we use that $\ell(\mu^{(1-\delta)/(\tau-1) C^k}) \mu^{ (\delta/2) C^k}>1$ for large $\mu$ by \eqref{Potter}. 
		For the upper bound, we use  the first moment method to find
		\be\label{upperboundweights}
		\mathbb{P}\Big(\max_{v\in \bki \cap \CV_\la} W_v > x\Big)\leq \Ev\Big[\sum_{v\in \bki \cap \CV_\la} \ind_{\{W_v >
			x\}}\Big]=\ell\left(x\right)x^{-\left(\tau-1\right)}R_k^d\lambda.
		\ee
		Inserting the upper bound from \eqref{boundsweight} for $x$ yields
		\be\label{eq:upperboundweights2}
		\mathbb{P}\Big(\max_{v\in \bki} W_v > \cb\mu^{C^k(1+\delta)/(\tau-1)}\cz\Big) \leq \ell\left(\cb\mu^{C^k(1+\delta)/(\tau-1) }\cz\right)\lambda \mu^{-\delta C^k}\leq \lambda \mu^{-(3/4)\delta C^k},
		\ee
		where we use that $\ell(\cb\mu^{C^k(1+\delta)/(\tau-1)}\cz)\mu^{- (\delta/4) C^k}<1$ by  \eqref{Potter} for large $\mu$. Combining the bounds in \eqref{eq:weight-below} and \eqref{eq:upperboundweights2} with a union bound over $k$ and the at most $b_k$ subboxes, for some constant $c_3>0$,
		\be\label{eq:unionboundweights}\ba 
		\mathbb{P}\Big(\big(\cap_{k\ge 0} F_k^{(1)}\big))^c\Big) &\leq   \sum_{k=0}^\infty b_k \big(\lambda \mu^{-(3/4)\delta C^k} + \exp\{-\lambda \mu^{(\delta/2) C^k} \}\big) \\
		&\leq c_3 \mu^{\left(D-1-(3/4)\delta\right)}+ c_3\exp\{-\lambda \mu^{\delta/2}/2\},
		\ea\ee
		where we can further use that $D-1-(3/4)\delta<-\delta/4$, as follows from the definition of $D(\eps), \delta(\ve)$ in \eqref{CD}, and the bound in  \eqref{boundbk}. \cb This yields the first two terms on the rhs of \eqref{boundsweight-error}. 
		
	\cb	We now turn to the events $(F_k^{(2)})_{k\in\N}$. The heuristics of the proof are the following: consider a centre $c_k^{(i)}$ of a subbox in annulus $k$. We would like to estimate the number of centres in annulus $k+1$ that it is connected to, $N_{k+1}(c_k^{(i)})$. For this, on $F_k^{(1)}\cap F_{k+1}^{(1)}$, we have  good bounds on the weight of $c_k^{(i)}$ and the centres in annulus $k+1$. Their Euclidean distance is at most the diameter of the outer annulus. Since the connection probabilities are bounded from below by the function $\underline g$ from \eqref{eq:gunder}, we show that the connection probability is large, in the sense that the minimum is taken by its first term when applied to $c_k^{(i)}$ and a centre in annulus $k$. Since the presence of edges is conditionally independent given the weights involved, $N_{k+1}(c_k^{(i)})$ is bounded from below by a binomial random variable with the number of centres in annulus $k+1$ and the lower bound we derived on the connection probability as its parameters. Concentration of this binomial variable then yields the desired lower bound.
		 Now we work out the details. Recall Claim \ref{claim:hg} and \eqref{eq:gunder}, in particular that the connection probability in EGIRG with vertices of weight $w_1, w_2$ at distance $\| x\|$ away is at least 
 \be\label{eq:gunder2} \underline c_1 \underline{g}(x,w_1,w_2)=\underline c_1\Big(\e^{-a_2((\log w_1)^\gamma+(\log w_2)^\gamma)}\wedge \underline a_1 \|x\|^{-\alpha d}\left(w_1 w_2\right)^\alpha\Big)\ee
 \cz
We condition that the event $F_k^{(1)}\cap F_{k+1}^{(1)}$ holds. 
		For two centres \cb in annuli $k,k+1$, respectively\cz, their weights satisfy the bounds in $F_k^{(1)}\cap F_{k+1}^{(1)}$, and their Euclidean distance is at most $\sqrt{d}D_{k+1}$ by \eqref{radii} and \eqref{Boxing}, hence, the second term in the minimum in \eqref{eq:gunder2} is at least
		\be\label{eq:dist-center1} a_1 \|\cki-c_{k+1}^{\sss\left(j\right)}\|^{-\al d} \big(W_{\cki} W_{c_{k+1}^{\sss\left(j\right)}}\big)^\al \ge a_1d^{-\al d/2}\cb\mu^{\al(C^k+C^{k+1})(1-\delta)/(\tau-1)}\cz/ (D_{k+1})^{\al d} =: a_{1,d}\mu^{\al Z C^k}\ee
		with $Z:=(1-\delta)(1+C)/(\tau-1)-CD$. It can be shown using \eqref{CD} that $Z=\ve(1-\ve)/4(\tau-1)>0$. 
		Let us abbreviate $l(w):=\exp\{ -a_2 (\log w)^\gamma\}$. Then, the first term in \eqref{eq:gunder2} is, conditioned on the event $F_k^{(1)}\cap F_{k+1}^{(1)}$, using the lower bound in \eqref{boundsweight}),
		\be \label{eq:dist-222}l(W_{\cki}) l(W_{c_{k+1}^{\sss\left(j\right)}}) \ge  \exp\left\{ -a_2 \Big(\frac{1+\delta}{\tau-1}\Big)^\gamma\Big(\log \mu^{C^k}\Big)^\gamma (1+C^\gamma)\right\}=:a_{\mu,k}/\underline c_1,\ee
		\cb with $\underline c_1$ the same constant as in Claim \ref{claim:hg}\cz. Hence,
			\be
		\ba\label{eq:connectprob}
		\mathbb{P} \Big(\cki \leftrightarrow c_{k+1}^{\sss\left(j\right)}\; \mid \cap_{k\ge 1} F_k^{(1)}\Big)&\geq  \underline c_1\min\Big\{a_{\mu,k}/\underline c_1,a_{1, d} \mu^{\alpha Z C^k } \Big\}\cb=\cz a_{\mu,k}\\
		\ea
		\ee
		since the second term in the minimum is larger than its first term when $\mu>\mu_0$ is sufficiently large. Recall that the number of subboxes in annulus $k+1$ is $b_{k+1}$ (see before \eqref{boundbk}). Since conditionally on the weights of vertices, edges are present independently,		
		\be\label{NBin}
		\Big(N_{k+1}(c_k^{(i)})\mid  W_{c_{k}^{(i)}}, (W_{c_{k+1}^{(j)}})_{j\le b_{k+1}}\Big)\overset{d}{\geq} \text{Bin}\left(b_{k+1},a_{\mu,k}\right),
		\ee
		 with $\text{Bin}(n,p)$ a binomial rv with parameters $n$ and $p$, \cb where the lhs denotes the rv $N_{k+1}(c_k^{(i)})$ conditionally on the weights of vertices involved, that should satisfy the event  $\cap_{k\ge 1} F_k^{(1)}$\cz.
By the concentration of binomial rvs \cite[Theorem 2.21]{Hofbook},
		\be\label{binomialbound}
		\mathbb{P}\big(\text{Bin}(n,p)\le  np/2\big)\leq \mathrm{e}^{-np/8}.
		\ee
		Combining this with \eqref{NBin}, \eqref{eq:connectprob} and the bound on $b_{k+1}$ in \eqref{boundbk} yields
				\be\label{Nboundbinomial}
		\mathbb{P}\Big(N_{k+1}\left(\cki\right)\leq  b_{k+1}a_{\mu,k}/2 \mid \cap_{k\ge 1} F_k^{(1)}\Big) \leq    \exp\left\{-a_{\mu,k}b_{k+1}/8\right\}\leq \exp\left\{- \mu^{\left(D-1\right)C^{k+1}/2}/16\right\},
		\ee
		for $\mu\ge \mu_1\ge \mu_0$ sufficiently large, since  $\gamma\in(0,1)$ in $a_{\mu,k}$ and  the bound on $b_{k+1}$ in \eqref{boundbk}. Elementary calculation yields that  $a_{\mu,k} b_{k+1}/2 \ge  \exp\{ C^{k+1} (D-1)(\log \mu) /2\}/2$, as in \eqref{boundsweight}. A union bound results in
		\be\label{connectioneq}
		\mathbb{P}\Big( \cap_{k\ge1} F_k^{(2)} \mid \cap_{k\ge1} F_k^{(1)}) \leq \sum_{k=0}^\infty b_k\exp\left\{-\mu^{\left(D-1\right)C^{k+1}/2}/16\right\}
	 \leq c_3 \exp\left\{- \mu^{\left(D-1\right)C/2}/32\right\},
		\ee
		where the final statement follows since the sum is dominated by a geometric series, and increasing  $c_3>0$ when necessary.
		The bound on  $N_k(\cki)$ in \eqref{boundsweight} follows in a similar manner, where now set $a_{\mu,k}:=\underline c_1l(\cb\mu^{C^k(1+\delta)/(\tau-1)}\cz)^2$.
		The proof of the last statement in Lemma \ref{lemma:weightconnectioncenters} follows by noting that $v$ can take the role of $c_0^{(i)}$, i.e. the bound on the second event $F_0^{(2)} $ for $k=0$ directly implies the estimate on $N_1(c_0^{(i)})>1$, when $\mu>\mu_0$ sufficiently large so that $\exp\{(\log \mu)(D-1)C/2\}>1$. 
	
		For the results to hold in the $\BGIRG$ model, the following adjustments are necessary: 
		First note that  the events $F_k^{(1)}$ and $F_k^{(2)}$ are only relevant for $k\leq k_{\max}(\mu,n):=\max\{k\in\mathbb{N}:\ D^d_k\leq n\}$, so that the annulus $\Gamma_{k_{\max}}$ has a non-trivial intersection with $\Xdn$. 
		Further, in $\BGIRG$, the weight distribution depends on $n$, as in Assumption \ref{assu:weight}, hence, all slowly-varying functions $\ell$ in the proof above should be changed to to either $\underline \ell$ or $\overline \ell$. \cb In particular, in the proof of the lower bound on the weights $\underline \ell$ should be used, turning the second line of \eqref{eq:lowerboundweights2} into an inequality. In the upper bound for the weights, $\underline \ell$ can be used, and then \eqref{eq:upperboundweights2} remains valid.
		The replacement of $\ell$ with $\underline \ell$ or $\overline \ell$ could be done without leaving the interval $[0, M_n]$ where these bounds hold, since $M_n\gg n^{1/(\tau-1)}$ by the condition in Assumption \ref{assu:weight}, while the maximal weight vertex in the whole graph $\BGIRG$ - following the calculation in \eqref{eq:lowerboundweights2} with deterministically $n$ vertices - is of order $n^{1/{\tau-1}}$.
		In \eqref{eq:lowerboundweights2}, the Poisson probabilities are replaced by 
		\be \mathbb{P}(|B_k^{(i)}\cap \CV(\BGIRG)|=j)={n \choose j}\big(\mathrm{Vol}(B_k^{(i)})/n\big)^j \big(1-\mathrm{Vol}(B_k^{(i)})/n\big)^{n-j},\ee 
		but the results in the upper bounds in \eqref{eq:lowerboundweights2} and \eqref{eq:upperboundweights2} stay valid with $\lambda=1$.
	\end{proof}

	\section{Best explosive path can be followed}\label{sec:bestexplosive}
	In this section we prove Lemma \ref{lem:ank}. We decompose its proof into three claims that we first state all at once and then prove one-by-one.
	\begin{claim}\label{claim:bg-to-eone} Let $E_{\ref{claim:bg-to-eone}}(n)$ be the event that  both uniformly chosen vertices $v_n^1, v_n^2\in \CC_{\max}$ of $\BGIRG$ are also in $\CC_\infty^1$, the infinite component of $\Eone$. Then $E_{\ref{claim:bg-to-eone}}(n)$ holds with probability $1-e_{\ref{claim:bg-to-eone}}(n)$, for some $e_{\ref{claim:bg-to-eone}}(n)\downarrow 0$.
	\end{claim}
	
	\begin{claim}\label{claim:explosive-ray}The infinite component of $\Elambda$ is a.s.\ explosive. That is, conditioned on $0\in \CC_{\infty}^\lambda$ there exists an infinite path $\cb\pi=\{0=\pi_0, \pi_1, \dots\}$ such that $(\pi_i,\pi_{i+1}) \in \mathcal E (\Elambda)$ and $\sum_{i\ge 0} L_{(\pi_i, \pi_{i+1})} < \infty$. 
		
	\end{claim}
	We denote the  infinite path from $v$ with shortest total $L$-length in $\Elambda$ by $\cb\pi^\lambda_{\mathrm{opt}}(v)$, defined in \eqref{eq:piopt}. Claim \ref{claim:explosive-ray} implies that  $\cb|\pi^\lambda_{\mathrm{opt}}(v)|_L<\infty$ a.s. By Corollary \ref{lem:expgen}, a.s., every infinite path with finite total $L$-length contains vertices with arbitrary large weights. This is important since we would like to find  large-weight vertices on $\pi_{\mathrm{opt}}^1(v_n^1),  \pi_{\mathrm{opt}}^1(v_n^2)$ in order to construct a connecting path and hence upper bound the distance between these vertices. 
	
	\begin{claim}\label{claim:ray-in-box}
		On the event $E_{\ref{claim:bg-to-eone}}(n)$ from Claim \ref{claim:bg-to-eone}, let $\wit v_n^q(K)$ denote the first vertex on $\cb\pi^1_{\mathrm{opt}}(v_n^q)$ with weight at least $K$. Let $E_{\ref{claim:ray-in-box}}(n,K)\subset E_{\ref{claim:bg-to-eone}}(n)$ denote the event that 
		the segment $\cb\pi_{\mathrm{opt}}^1[v_n^q, \wit v_n^q(K)]$ is within $\Xdn$ and also contained in $\BGIRG$. Then, $\Pv(E_{\ref{claim:ray-in-box}}(n,K))\ge 1-e_{\ref{claim:ray-in-box}}(n,K)$ for some function $e_{\ref{claim:ray-in-box}}(K,n)$ with $\lim_{n\to \infty} e_{\ref{claim:ray-in-box}}(n,K)=0$ for all fixed $K$.
	\end{claim}  
	Before we proceed to the proofs, we quickly prove Lemma \ref{lem:ank} subject to these three claims.
	\begin{proof}[Proof of Lemma \ref{lem:ank} subject to Claims \ref{claim:bg-to-eone} - \ref{claim:ray-in-box}.]
		By noting that $\wit A_{n,K}:=E_{\ref{claim:ray-in-box}}(n,K)$, the proof follows by Claims \ref{claim:bg-to-eone} - \ref{claim:ray-in-box} with $f(n,K):=e_{\ref{claim:bg-to-eone}}(n)+e_{\ref{claim:ray-in-box}}(n,K)$.
	\end{proof}
\begin{proof}[Proof of Theorem \ref{thm:exp-charact} when $\tau\in(2,3)$ subject to Claim \ref{claim:explosive-ray}]
Suppose first that ${\mathbf I}(L)=\infty$. In this case, due to the coupling in Lemma \ref{lemma:ballsbrw} and \eqref{eq:ball-containment}, non-explosion of $\mathrm{BerBRW}_\la(v)$ implies non-explosion of $\Elambda$.  Non-explosion of $\mathrm{BerBRW}_\la(v)$ is the content of the proof of \cite[Theorem 1.1]{HofKom2017} that can be found in \cite[Section 7]{HofKom2017}. The idea is that when ${\mathbf I}(L)=\infty$ and the generation sizes of a BRW grow at most doubly exponentially, the sum over $k$ of the length of the \emph{shortest edge} between generation $k$ and $k+1$ is already infinite a.s., hence, any path to infinity must also have infinite total length a.s. We refer the reader to \cite[Section 7]{HofKom2017} for further details. 
		The proof of the explosive direction follows from  Claim \ref{claim:explosive-ray}.  
		\end{proof}

	\begin{proof}[Proof of Claim \ref{claim:bg-to-eone}]
		We distinguish two cases, whether $W_{v_n^q}^{(n)} \le \max\{\xi_n, \epsilon_{\mathrm{TV}}(n), \epsilon(n)\}^{-1/(2(\tau+\eta_1))}=:w(n)$ or not, with $\eta_1>0$ small. We start with the first case. We make use of a result by Bringmann et al., \cite[Lemma 5.5]{BriKeuLen17}. Fix a weight $w\ge 1$ that might depend on $n$ later and a small constant $\eta_1>0$. Let the event $E_n^1(v,w)$ denote that a vertex $v\in [n]$ with weight at most $w$ has component size larger than $w^{\tau+\eta_1}$, and $E_n^2(v,w)$ denote that vertex $v$ is within graph distance $w^{\tau+\eta_1}$ of some vertex with weight at least $w$. Then, by \cite[Lemma 5.5]{BriKeuLen17}, if $v$ is chosen uniformly from the vertices with weight $\le w$, 
		\be\label{eq:lenglerineq}
		\P{E_n^1(v,w)\cap (E_n^2(v,w))^c}\leq \exp\{-w^{\eta_2}\},
		\ee 
		for some $\eta_2=\eta_2(\eta_1)>0$. Note that if $v_n^q$ is uniformly chosen in $\CC_{\max}$ and we assume that $W_{v_n^q}^{(n)} \le w$, then, since the weights are originally i.i.d., $v_n^q$ is also uniformly chosen in $\CC_{\max} \cap \{v\in[n]: W_v^{(n)}\le w\}$. 
		Since $|\CC_{\max}|$ is linear in $n$ by \cite[Theorem 5.3]{BriKeuLen17}, it follows that for any $w$,  $\Pv(E_n^2(v,w)\mid v\in \CC_{\max}) \ge 1-c_4\exp\{-w^{\eta_2}\}$ for some $c_4>0$.
		On $E_n^2(v,w)$, let us choose a vertex $v_{\geq}$ that has weight at least $w$ and graph distance at most $w^{(\tau+\eta_1)}$ from $v$, denote its weight by $\mu_{\ge}$ and denote the path from $v$ to $v_{\geq}$ by $\cb\pi[v,v_{\geq}]$. Using the coupling in Claim \ref{cl:vertex-containment}, each vertex in $\CV_B(n)$ is not in $\CV_{1-\xi_n}\subseteq \CV_1$ with probability $\Ev[1-|\CV_{1-\xi_n}\cap \Xdn|/n]=\xi_n$. Given that a vertex $u$ is in both models, we can couple its weight in $\BGIRG$ and in $\Emin$ with error $\epsilon_{\mathrm TV}(n)$.   Finally, by \eqref{singleedge} in Claim \ref{claim:edge-containment}, given that vertex-weights are equal, each edge that is part of $\CE_B(n)$ is not part of $\CE_{1-\xi_n}\subseteq \CE_1$ with probability at most $\epsilon(n)$. So, setting $w=w(n)$,
		\be \ba
		\Pv&\big(\cb\pi[v,v_{\geq}]\cz\not\subseteq \mathrm{EGIRG}_{W,L}(1)\mid E_n^{2}(v,w)\big)\leq{} \P{\exists u\in\cb\pi[v,v_{\geq}]\cz:\ u\not\in \mathcal{V}_{1-\xi_n}}\\
		&\  \ +\P{\exists u\in\cb\pi[v,v_{\geq}]\cz:\ W_u^{(n)}\neq W_u\mid \forall u \in\cb\pi[v,v_\geq]\cz:\ u\in\mathcal{V}_{1-\xi_n}}\\
		&\ \ +\Pv\Big(\exists (u_1,u_2)\in\EE(\cb\pi[v,v_{\geq}]\cz):\ (u_1,u_2)\not\in \EE_{1-\xi_n}\mid \forall u \in\cb\pi[v,v_\geq]\cz:\ u\in\mathcal{V}_{1-\xi_n}, W_u^{(n)}=W_u\Big)\\
		& \leq{} w(n)^{\tau+\eta_1} (\xi_n +\epsilon_{\mathrm{TV}}(n)+\epsilon(n)),
		\ea\ee
		that tends to zero by the choice of $w(n)$. \cb To summarise, assuming that a vertex $v$ is in the giant component of $\BGIRG$, with high probability we find a vertex $v_\ge$ with weight at least $w(n)$ that is connected to $v$, and the whole segment of the path between $v, v_\ge$ is part of $\Eone$ whp as well. Next, $v_\ge$ will serve as the initial vertex to construct a component with infinitely many vertices $\Eone$.
Namely, we construct the boxing system $\mathrm{EGIRG}_{W,L}(1)$, described in \eqref{radii}--\eqref{Boxing}, and centred at $v_{\geq}$ with starting weight $\mu:=\mu_\ge \ge w(n)$. Lemma \ref{lemma:weightconnectioncenters}, then ensures that the centres of subboxes in each annulus of this boxing system have sufficiently many connections to centres in the next annulus,  in particular, at least one connection. Starting from $v_\ge$, we find that it is connected to a centre in annulus one, and iteratively following an arbitrary centre in the next annulus that the centre is connected to, we find an infinite path $\{v_{\geq}=c_0^{\sss(i_0)},c_1^{\sss(i_1)},\ldots\}$ containing a centre from each annulus $\Gamma_k,k\geq1$. 

This construction fails with probability that is the rhs of \eqref{boundsweight-error}, that is at most  $c_5\mu_\ge(n)^{-\delta(\ve)/4}\le c_5w(n)^{-\delta(\ve)/4}$, for some $c_5>c_3$, since the dominating term is the first one in \eqref{boundsweight-error} when $\mu=\mu(n)\to \infty$. 	As a result, $v$ is \emph{not} part of an infinite component of $\Eone$ with probability at most	\be\ba \label{eq:e5}e_{\ref{claim:bg-to-eone}}(n)&:=c_4\exp\{- \max\{\xi_n, \epsilon_{\mathrm{TV}}(n), \epsilon(n)\}^{-\eta_2/2(\tau+\eta_1)} \}\\&\ \ +2\max\{\xi_n, \epsilon_{\mathrm{TV}}(n), \epsilon(n)\}^{1/2}+c_5\max\{\xi_n, \epsilon_{\mathrm{TV}}(n), \epsilon(n)\}^{\delta(\ve)/8}, \ea\ee
which tends to zero as $n\to \infty$. Since the infinite component of $\Eone$ is almost surely unique, this is sufficient. 
		When the weight of $v_n^q$,  $W_{v_n^q}^{(n)}\ge w(n)$, then we can directly construct a boxing system in $\Eone$ around $v_n^q$ and use Lemma \ref{lemma:weightconnectioncenters}, i.e. $v_n^q=v_\ge$ can be set in the proof above. 
		So, \eqref{eq:e5} is an upper bound in this case as well, finishing the proof.\cz\end{proof}
	
	\begin{proof}[Proof of Claim \ref{claim:explosive-ray}]
		It is enough to construct an infinite path starting at $0\in\CC_{\infty}^\lambda$ with a.s.\ finite total length. Let $(\mu_i)_{i\in\mathbb{N}}$ be an increasing sequence such that 
		the rhs of \eqref{boundsweight-error} with $\mu=\mu_i$
		are summable over $i\ge1$. Then, construct a boxing system around $0$ for each $i\ge 1$ and let $i_0$ be the first index such that $\cap_{k\ge 0} (F_k^{(1)}\cap F_k^{(2)})$ holds for the boxing system with parameter $\mu_{i_0}$ (see \eqref{Boxing}). The rv $i_0$ is a.s.\ finite by the Borel-Cantelli lemma. Consider this boxing system.
	Let $c_{0}^{(1)}$ be the vertex with maximal weight in $\Gamma_{0}$. \cb Then, since \cz $\cap_{k\ge 0} (F_k^{(1)}\cap F_k^{(2)})$ holds, there is an infinite path through centres in consecutive annuli starting from $c_{0}^{(1)}$, implying that $c_0^{(1)}\in \CC_\infty^\la$ holds. By the uniqueness of $\CC_\infty^\la$ and the fact that $\|c_0^{(1)}\|_\infty\le \mu_{i_0}$, the graph distance $d_G(0, c_0^{(1)})$ is a proper random variable, and so is the $L$-length of the path $\cb\pi[0, c_0^{(1)}]$ realising the graph distance. 
		We iteratively continue this path into a path $\cb\pi_{\text{greedy}}$ in a greedy manner so that it has finite total $L$-length, as follows: let us call $c_0^{(1)}=:c_0^{\mathrm{gr}}\in \cb\pi_{\text{greedy}}\cz$, and for $k\geq 0$, from $c_k^{\mathrm{gr}}\in\cb\pi_{\text{greedy}}$, \cb take the minimal-length edge among those leading to \cz the $N_{k+1}(c_k^{\mathrm{gr}})$ many
		centres in $\Gamma_{k+1}$ that are connected to $c_k^{\mathrm{gr}}$ and let $c_{k+1}^{\mathrm{gr}}$ be its other end-vertex.
		By the lower bound on $N_{k+1}(c_k^{(i)})$ on $F_k^{(2)}$ in \eqref{boundsweight}, the total length of $\cb\pi_{\text{greedy}}$ is, in distribution, bounded from above by
		\be\label{explosivesum}
		\cb|\pi_{\text{greedy}}|_L\cz\overset{d}{\leq} \cz|\pi[0,c_0^{(1)}]|_L\cz+\sum_{k=1}^\infty \min\left\{L_{k,1},L_{k,2},\ldots,L_{k,d_k}\right\},
		\ee
		with $d_k:=\exp\{ C^k (D-1)(\log \mu_{i_0}) /2  \} /2$, which is doubly-exponential in $k$. It follows from \cite[Lemma 2.5]{HofKom2017} that the integral criterion in \eqref{Lassumption} is equivalent to the a.s.\  convergence of the sum in \eqref{explosivesum}.\end{proof}
	
	\begin{proof}[Proof of Claim \ref{claim:ray-in-box}] In this proof, we make use of the event $A_{k_n}$, \cb as in \eqref{eq:ak}. Recall $k_n$ from \eqref{wtKn}, \cb where $H_n$ is defined in \eqref{eq:mn}\cz, and that on $A_{k_n}$, the graph distance balls $B^G_n(v_n^q, k_n), B_1^G(v_n^q, k_n)$ \emph{as weighted graphs} coincide, and are part of $\Xdn$ for $q\in\{1,2\}$, as well as that this event holds with probability at least $1-\sum_{i\le 5} e_i(n)\to 1$ by \eqref{eq:an-bound1}. \cb Recall also $E_{\ref{claim:bg-to-eone}}$, the event that $v_n^q\in \CC_{\max}$ is in the infinite component of $\Eone$. On $E_{\ref{claim:bg-to-eone}}$, the shortest explosive path $\pi^1_{\mathrm{opt}}(v_n^q)$ exists\cz.
		Let $G_K(v^q_n)$ denote the number of edges in the segment $\cb\pi_{\mathrm{opt}}^1[v^q_n,\wt v^q_n(K)]$. We define the event
		\be  
		E^{(6)}_{n,K}:=\{G_K(v_n^q)< k_n\}
		\ee
		On $E^{(6)}_{n,K}\cap E_{\ref{claim:bg-to-eone}}\cap A_{k_n}$, $\cb\pi^1_{\mathrm{opt}}[v_n^q,\wt v_n^q(K)]\cz\subseteq \BGIRG$ holds.
		 \cb Generally
			\[ \ba\Pv((A\cap B\cap C)^c) 
			&\leq \Pv(A^c | B\cap C) \Pv(B \cap C)  + \Pv(A^c | (B\cap C)^c) \Pv((B \cap C)^c) + \Pv(B^c) + \Pv(C^c)\\
			&\le \Pv(A^c | B\cap C) + 2 \Pv(B^c) + 2\Pv(C^c). \ea\]
			With $A:=E^{(6)}_{n,K}, B:=E_{\ref{claim:bg-to-eone}}, C:=A_{k_n}$, this estimate turns into\cz
		\be\label{eq:gamma-in-bg} 
		\mathbb{P}(\cb\pi^1_{\mathrm{opt}}[v_n^q,\wt v_n^q(K)]\cz\not\subseteq \BGIRG) \leq 2\Pv\big((E_{\ref{claim:bg-to-eone}})^c\big)+2\Pv\big(A_{k_n}^c \big)
		 +\Pv\big((E^{(6)}_{n,K})^c \mid E_{\ref{claim:bg-to-eone}} \cap A_{k_n} \big).	\ee
		  Note that on the conditioning, the graph distance balls with radius $k_n\to \infty$ coincide in $\BGIRG$ and $\Eone$ as weighted graphs. Thus, by the translation invariance of $\Eone$, on $A_{k_n}$, $G_K(v_n^q)\le k_n$ has the same distribution as $G_K(0)\le k_n$ conditioned on $0\in \CC_\infty^1$, i.e. the dependence of $G_K(v_n^q)$ on $n$ is not apparent.  
		Thus the proof is finished with 
		\be e_{\ref{claim:ray-in-box}}(n,K):=2e_{\ref{claim:bg-to-eone}}(n)+\sum_{i\le 5} 2e_i(n)+\P{G_K(0)\geq k_n \mid  0\in \CC_\infty^1 }.\ee\end{proof}

	\begin{proof}[Proof of Proposition \ref{Prop:upperbound}] By Prop.\ \ref{Prop:existinfcomp} and Claim \ref{claim:explosive-ray}, the infinite component of $\mathcal{C}_\infty^1$ in $\Eone$ is explosive, and by Claim \ref{claim:bg-to-eone}, $v_n^q\in\mathcal{C}_\infty^1,q\in\{1,2\}$ whp. Hence, $\cb\pi_{\text{opt}}^1(v_n^q)$, the optimal infinite ray from $v_n^q$ in $\Eone$ with finite total $L$-length exists whp (see \eqref{eq:piopt}). Recall $\wt v_n^q(K)$ as the first vertex on $\pi_{\text{opt}}^1(v_n^q)$ with weight exceeding $K>0$, that exists  by Corollary \ref{lem:expgen}, and that $T_K(v_n^q)=d^1_L(v_n^q,\wt v_n^q(K))$.
		On the event $\wt A_{n,K}$ in \eqref{eq:ank},  the segment $\pi_{\text{opt}}^1[v_n^q,\wt v_n^q(K)]$ is present in $\BGIRG$, and $\wt A_{n,K}$ holds with probability at least $1-f(n,K)$ by Lemma \ref{lem:ank}. We shortly write $\wit v_n^q:=\wit v_n^q(K)$ from now on in this proof.
		
		Our aim is \cb to prove \cz \eqref{upperboundGIRG}, i.e. that on $\wt A_{n,K}$ we can connect $\wt{v}^n_1$ to $\wt{v}^n_2$ with a short path within $\BGIRG$. 
		\cb We establish this connection \emph{in the percolated graph} $\cb G^{\mathrm{thr}}$. More precisely, we do the following: Fix $c\le \alpha$ and any $\wit \gamma\in (0,1)$, and  set the threshold function to be
		\be\label{eq:thr-set} \mathrm{thr}(w_1,w_2):= F_L^{(-1)}\Big(\exp\{-c (\log w_u)^{\wit \gamma} -c(\log w_u)^{\wit \gamma}\}\Big).\ee Following the argumentation in \eqref{percolationprob}, we keep each edge $e=(u,v)$ if and only if its edge-weight $L_e\le \mathrm{thr}(w_u, w_v)$. This gives us a graph $\cb G^{\mathrm{thr}}$, that is the outcome of a weight-dependent percolation on $\BGIRG$, as in Definition \ref{def:percolationprob}, with 
		\be\label{eq:p-set} p(w_u,w_v)=F_L(\mathrm{thr}(w_u, w_v))\ge \exp\{-c (\log w_u)^{\wt\gamma} -c(\log w_u)^{\wt\gamma}\}.\ee
		The last inequality holds
		by the right-continuity of the distribution function $F_L$ and the definition of $F_L^{(-1)}$. 
		Hence, the conditions of  Claim \ref{claim:percolation} are satisfied with this choice of $\mathrm{thr}$, and  $\cb G^{\mathrm{thr}}$ contains a subgraph $\underline G^{\mathrm p}$ that is an instance of  a $\cb \mathrm{BGIRG}$, that still satisfies Assumptions \ref{assu:GIRGgen}, \ref{assu:weight} and \ref{assu:extendable}. Recall from the proof of Claim \ref{claim:percolation}, from \eqref{eq:new-weights}, that in $\cb \underline G^{\mathrm{p}}$ a vertex has weight $\cb W^{(n)}_{\mathrm{p},i}\cz=m(W^{(n)}_{i})$. Recall also that the function $m$ is increasing eventually, in particular, we assume that $K$ is such that it is increasing on  $[K,\infty)$.

	Now we return to $\wit v_n^q$. For $q\in\{1,2\}$, we construct two boxing systems as in \eqref{radii}--\eqref{Boxing}, centred at $\wt{v}^n_q$ with parameter $\mu^{(q)}:=\cb m(W^{(n)}_{\wit v_n^q})\cz\ge m(K)$, $q\in\{1,2\}$, respectively, as shown in Figure \ref{Fig:boxing}. 
	We apply Lemma \ref{lemma:weightconnectioncenters} to the boxing system \emph{within the subgraph of the percolated graph} $\underline G^{\mathrm{p}}$. This is an important point of the argument, since then we have a deterministic upper bound (given by $\mathrm{thr}$) on the edge-weight of an edge that is part of $\underline G^{\mathrm{p}}$.
	Then, with the error probabilities given in Lemma \ref{lemma:weightconnectioncenters}, each subbox centre has at least one connection within $\underline G^{\mathrm{p}}$ to a centre in the next annulus, and hence we can create a path $\pi^q$ from $\wt{v}_n^q$ through the centres of sub-boxes in consecutive annuli, as in Figure \ref{Figure:overviewupperbound}. 
		\cz
		We would like to connect two such paths by connecting a centre in one of the systems to a centre in the other system. This may not be directly possible, hence we \emph{merge} the two boxing systems in what follows.
		
		Without loss of generality, we assume $\mu^{(1)}\le \mu^{(2)}$. We refer to the boxing system around $\wt{v}_n^q$ as System-$q$ and denote all radii and structures as in \eqref{radii} and \eqref{Boxing} and the centres $\cki$ with a superscript or argument $q$ to show which system they belong to. Recall $C=C(\ve)$ from \eqref{CD} and set 
		\be\label{modify-point} {r}:=\lceil (\log\log \mu^{(2)}-\log\log \mu^{(1)})/\log C \rceil = (\log\log \mu^{(2)}-\log\log \mu^{(1)})/\log C + a_{1,2},\ee with $a_{1,2}\in[0,1)$ an upper fractional part, hence $r\ge 1$. \cb In words, $r$ is the first index $k$ when the annulus $k$ in System-$1$ becomes larger then annulus $1$ in System-$2$\cz.
		Let us modify System-$1$ as follows: define $(\text{Box}_{k}^1, \Gamma_k^1)_{k\le r-1}$ as given by \eqref{Boxing} with parameter $\mu^{(1)}$, and then let, for $j\ge 0$, 
		\be \label{Boxing-mod}{\text{Box}}_{r+j}^1:=\{x\in \R^d: \| x-\wit v_n^1\| \le D_j^{(2)}\}, \ee 
		i.e. these boxes have radius $D_j^{(2)}=(\mu^{(2)})^{DC^{j}/d}$ instead of $D^{(1)}_{r+j}=(\mu^{(1)})^{D C^{r+j}}=(\mu^{(2)})^{DC^{j+a_{1,2}}/d}$, and, since $a_{1,2}\geq 0$, $D_{r}^{(1)}>D_0^{(2)}$. \cb In words, this means that we \emph{shrink} the size of the annuli with index at least $r$ somewhat, to have the same base $\mu^{(2)}$ as System-$2$\cz.  Note that ${\text{Box}}_{r}^{(1)}$ is still bigger than $\text{Box}_{r-1}^{(1)}$.
		As usual, $\Gamma_i^{(1)}:=\text{Box}_{i}^{(1)} \setminus \text{Box}_{i-1}^{(1)}$.  
		Let us further use subboxes of diameter $R_0^{(2)}=(\mu^{(2)})^{1/d}$ instead of $R_r^{(1)}$ in annulus $\Gamma_r^{(1)}$.
		We claim that the error probabilities in \eqref{boundsweight-error} remain valid in Lemma \ref{lemma:weightconnectioncenters} with $\mu$ replaced by $m(K)$ for the modified System-$1$. For this it is enough to investigate $\Pv(F_{r}^{(1)}(\wit v_n^1))$ and $\Pv(F_{r-1}^{(2)}(\wit v_n^1)| \cap_{k\ge 1} F_{k}^{(1)}(\wit v_n^1))$, since for the other indices, the error estimates remain valid with $\mu$ either replaced by $\mu^{(1)}$ or $\mu^{(2)}$, both at least $m(K)$. 

		Using the definition of $r$, we rewrite $F_{r-1}^{(1)}(\wit v_n^1)$ in terms of $\mu^{(2)}$, and modify $F_r^{(1)}(\wit v_n^1)$ to accommodate the fact that the subbox sizes have volume $\mu^{(2)}$ as follows, with $\de=\de(\ve)$ as in \eqref{CD}:
		\be\ba\label{Etildeweights}
		{F}_{r-1}^{(1)}(\wit v_n^1)=\Big\{& W^{(n)}_{c_{r-1}^{\sss(i)}(1)}\in\Big[(\mu^{(2)})^{(1-\delta)/(\tau-1) C^{-\left(1-a_{1,2}\right)}},(\mu^{(2)})^{(1+\delta)/(\tau-1) C^{-\left(1-a_{1,2}\right)}}\Big]\Big\},\\
		\wt{F}_{r}^{(1)}(\wit v_n^1):=\Big\{
		& W^{(n)}_{{c}_{r}^{\sss(i)}(1)}\in\Big[(\mu^{(2)})^{(1-\delta)/(\tau-1)}, (\mu^{(2)})^{(1+\delta)/(\tau-1)} \Big]\Big\}.
		\ea\ee
		Then, the estimate on $F_0^{(1)}(\wit v_n^2)$ with parameter $\mu^{(2)}$ is valid for $\Pv(\wit F_r^{(1)}(\wit v_n^1))$. Further, analogously to the arguments between \eqref{eq:dist-center1} - \eqref{eq:connectprob}, the connection probabilities of individual centres are minimised by the first argument, i.e. with $l(w)=\exp\{-\wih c (\log w)^{\gamma^{\max}} \}$, for some $\wih c>0$ and some $\gamma^{\max}\in(0,1)$,
		\be
		\ba\label{probexceeddifferentboxing}
		\mathbb{P}&\left(c^{\sss(i)}_{r-1}(1)\leftrightarrow {c}^{\sss{(j)}}_{r}(1)\mid  F_{r-1}^{(1)}\cap\wt{F}_{r}^{(1)} \right)\geq \underline c_1 l\left((\mu^{(2)})^{(1+\delta)/(\tau-1)}\right)l\Big((\mu^{(2)})^{(1+\delta)/(\tau-1)C^{-\left(1-a_{1,2}\right)}}\Big)=:\wit a_{\mu, r}, 
		\ea
		\ee 
		Let us write $\wt{b}_{r}^1$ for the number of subboxes in ${\Gamma}_{r}^1$. Then, by \eqref{Boxing} and  \eqref{Boxing-mod},
		\be \wt{b}_{r}^1=((\mu^{(2)})^D-(\mu^{(2)})^{DC^{-(1-a_{1,2})}})/\mu^{(2)}\geq (\mu^{(2)})^{D-1}/4\ee for all sufficiently large $K$, since $\mu^{(2)}>m(K)$. It might happen that (some of) the annuli are not completely within $\Xdn$. This happens when either of the $\wt{v}_n^q$ is close to the boundary of $\Xdn$. However, Lemma \ref{lemma:weightconnectioncenters} still applies since only at most half of the volume of each annulus is outside $\Xdn$ and therefore the part of the annulus $\Gamma_k^q ,k\geq 0$ that lies within $\Xdn$ still grows doubly exponentially with $k$, so the lower bound on $b_k$ in \eqref{boundbk} could be modified by a factor $1/2$, and the results of Lemma \ref{lemma:weightconnectioncenters} still hold.	
		Recall $N_{k+1}(\cki)$ from Lemma \ref{lemma:weightconnectioncenters}.
		Then, analogous to \eqref{NBin} -\eqref{Nboundbinomial}, by the independence of the presence of edges conditioned on the weights, 
		\be \ba \Big( {N}_{r}(c^{\sss(i)}_{r-1}(1)) \mid  W_{c_{r}^{(i)}}, (W_{c_{r}^{(j)}})_{j\le \wit b_{r}^1}\Big))&\overset{d}{\geq}\text{Bin}(\wt{b}_{r}^1,\wit a_{\mu,r}),\\
		\mathbb{P}\left({N}_{r}(c^{\sss(i)}_{r-1}(1))\leq \wt{b}^1_{r} \wit a_{\mu,r}/2\right)&\leq 2\exp\left\{-\wt{b}_{r}^1 \wit a_{\mu,r}/8\right\}\leq 2\exp\left\{-(\mu^{(2)})^{\left(D-1\right)/2}/16\right\},
		\ea
		\ee 	
		where the second inequality follows for sufficiently large $\mu^{(2)}$ (and thus $K$), using the lower bound on $\wt b_{m}^1$ and bounding $\wit a_{\mu,r}$ from below by using that $\gamma^{\max}\in(0,1)$.
		Finally, noting that $\wit b_{r}^1 \wit a_{\mu,r} \ge (\mu^{(2)})^{(D-1)/2}/2 =\exp\{ (D-1) (\log \mu^{(2)})/2\}$ for all sufficiently large $\mu^{(2)}$ (and thus $K$) finishes the proof of the statement. 
		
		To conclude, we can build paths from both vertices $\wit v_n^1(K), \wit v_n^1(K)$ through centres of boxes in the modified boxing system, with error probability as in Lemma \ref{lemma:weightconnectioncenters} with $\mu $ replaced by $m(K)$. 
		Next we \emph{connect} these paths. 
		We introduce 
		\be\label{eq:kstar} k^\star:=\max\{k\in\mathbb{N} \;|\; D_k^{(2)}<\|\wt{v}_n^1- \wt{v}_n^2\|/(2\sqrt{d})\},\ee
		the largest (random) $k$ such that $\text{Box}_{r+k}^{(1)}$ and $\text{Box}_k^{(2)}$ are disjoint. \cb Note that per definition, the next boxes, with index $k^\star+1, r+k^\star+1$, respectively, would be centred around $\wit v_n^2$, $\wit v_n^1$. Instead, we now define $\text{Box}_{k^\star+1}^{(1,2)}$ as a box \emph{centred at} $0\in \R^d$, \cz
		containing \emph{both} $\text{Box}_{r+k^\star}^{(1)}\left(\wt{v}_n^1\right)$ and $\text{Box}_{k^\star}^{(2)}\left(\wt{v}_n^2\right)$. So, we let
		\be\label{boxing}
		\ba
		&D_{k^\star+1}:=\left((\sqrt{d}/2+1)(\mu^{(2)})^{DC^{k^\star+1}}\right)^{1/d},\qquad  R_{k^\star+1}:=\left((\mu^{(2)})^{C^{k^\star+1}}\right)^{1/d},\\
		&\text{Box}_{k^\star+1}^{(1,2)}:=  \left\{x\in\mathbb{R}^d\;\big|\;\|x\|_\infty\leq D_{k^\star+1}/2\right\},\\ &\Gamma_{k^\star+1}^{(1,2)}:=\text{Box}_{k^\star+1}(\wt{v}_n^1,\wt{v}_n^2)\backslash \left(\text{Box}_{m+k^\star}^{(1)}\cup \text{Box}_{k^\star}^{(2)}\right).
		\ea
		\ee
		\cb Observe that the change is not in the exponent, just in the prefactor of the size of the box\cz,  hence, the error bounds on $F_{r+k^\star}^{(2)}(\wit v_n^1)$ and $F_{k^\star}^{(2)}(\wit v_n^2)$ in Lemma \ref{lemma:weightconnectioncenters} hold in System-$1$ and $2$, respectively. 
		Thus, with error probability twice the rhs of \eqref{boundsweight-error}, there exist two paths \emph{in the percolated graph} $\cb\underline G^{\mathrm p}\subseteq G^{\mathrm{thr}}$  
		\be\label{eq:pi12} \pi^{(1)}:=\{\wt v_n^1,c_1^{\sss (i_1)}(1),c_2^{\sss (i_2)}(1),\ldots, c_{r+k^\star+1}^{\sss (i_{r+k^\star+1})}(1)\},\qquad 
		\pi^{(2)}:=\{\wt v_n^2,c_1^{\sss (j_1)}(2),c_2^{\sss (j_2)}(2),\ldots, c_{k^\star+1}^{\sss (j_{k^\star+1})}(2)\}.
		\ee through the annuli $\Gamma_k^{(q)}(\wt{v}_n^q)$. 
		Note that there are $b_{k^\star+1}\ge (\mu^{(2)})^{(D-1)C^{k^\star+1}}/2$ many centres in $\Gamma_{k^\star+1}^{(1,2)}$.
		The connection probability between any two centres within $\Gamma_{k^\star+1}^{(1,2)}$ is at least 
		\be\label{ERlowerbound}
		\ba
		\P{c_{k^\star+1}^{\left(i\right)}\leftrightarrow c_{k^\star+1}^{\left(j\right)}\mid  F^{(1)}_{k^\star+1}} & \geq c_1 l\Big((\mu^{(2)})^{(1+\delta)/(\tau-1)C^{k^\star+1}}\Big)^2\geq c_1(\mu^{(2)})^{-2a_2\left((1+\delta)/(\tau-1) C^{k^\star+1}\right)^\gamma }=:a_{\mu^{(2)}},
		\ea
		\ee
		where $F^{(1)}_{k^\star+1}$ is the event defined in \eqref{boundsweight},	as in the proof of Lemma \ref{lemma:weightconnectioncenters}. When $c_{r+k^\star+1}^{\sss (i_{r+k^\star+1})}(1)=c_{k^\star+1}^{\sss (j_{k^\star+1})}(2)$ or there is an edge between these two vertices in $\cb G^{\mathrm{p}}$, a connection is established. If these are not the case, then we find a third centre that connects them.
		Since we condition on the environment, edges are present independently with probability at least $a_{\mu^{(2)}}$. Hence,
		\be\label{eq:piconn}\ba  \Pv\big(\nexists i\le b_{k^\star+1}, c_{r+k^\star+1}^{\sss (i_{r+k^\star+1})}(1) &\leftrightarrow c_{k^\star+1}^{\sss{(i)}} \leftrightarrow c_{k^\star+1}^{\sss (j_{k^\star+1})}(2) \mid F_{k^\star+1}^{(1)}\big) \le (1-a_{\mu^{(2)}})^{2b_{k^\star+1}}\\
		&\le \exp\{-2 b_{k^\star+1}a_{\mu^{(2)}} \} \le \exp\{ - (\mu^{(2)})^{(D-1)C^{k^\star+1}}/2  \}
		\ea  \ee
		for all large enough $\mu^{(2)}$. As a result, this error term can be merged in with the other error terms in Lemma \ref{lemma:weightconnectioncenters}. Let us denote this connecting piece (of either $0,1$ or $2$ edges) by $\pi^{\text{conn}}$. What is left is to bound the $L$-length of the constructed path. \cb For this we use the definition of $\mathrm {thr}$ in \eqref{eq:thr-set} and its monotonicity in $w_u, w_v$, and the fact that we know a lower bound on the weights of the centres along the constructed path, due to the events $\cap_{k\ge 1} (F_k^{(1)}(\wit v_n^1) \cap F_k^{(1)}(\wit v_n^2) ) $. Observe that the first $r$ vertices of the path $\pi^{(1)}$ are in System-$1$ while the rest is in the merged system, with base $\mu^{(2)}$. All combined yields\cz
				\be\label{lengthboxingpath}
		\ba
		|\pi^{(1)}|_L+	|\pi^{(2)}|_L&+ |\pi^{\text{conn}}|_L\leq \\
		&	 \sum_{k=0}^{r}F^{\sss\left(-1\right)}_L\left(\exp\{-c \log \big((\mu^{(1)
		})^{(1-\delta)/(\tau-1)C^k}\big)^{\wt \gamma} -c \log \big((\mu^{(1)})^{(1-\delta)/(\tau-1)C^{k+1}}\big)^{\wt \gamma}\}\right)\\
		&\ +2\sum_{k=0}^{k^\star
			+1}F^{\sss(-1)}_L\left(\exp\{-c \log \big((\mu^{(2)})^{(1-\delta)/(\tau-1)C^k}\big)^{\wt \gamma}-c \log \big((\mu^{(2)})^{(1-\delta)/(\tau-1)C^{k+1}}\big)^{\wt \gamma}\}\right)\\
		&\leq   3\sum_{k=0}^\infty F^{\sss(-1)}_L\left(\exp\left\{- C^{\gamma k}\left((1-\delta)/(\tau-1) \log K\right)^{\wt\gamma}\right\}\right)=:\ve_K,
		\ea
		\ee 
		where we have used that $\mu^{(q)}\ge m(K)\ge K^{1/2}$, $q\in\{1,2\}$, for all sufficiently large $K$ in the last inequality. The integrability condition in \eqref{Lassumption} is equivalent to the summability of the final expression in \eqref{lengthboxingpath} by \cite[Claim 4.5]{HofKom2017}. 
		Further, $\ve_K\to 0$ as $K\to \infty$.
		Thus, set $\ve_K$ as in \eqref{lengthboxingpath} and error probability $\eta(K)$ as twice the rhs of \eqref{boundsweight-error}, with $\mu$ replaced by $m(K)\ge K^{1/2}$. This finishes the proof.
	\end{proof}
	
		\section{Extension to Scale Free Percolation}\label{sec:extensionSFP}
	In this section, we give a sketch of the proof of Theorem \ref{Th:SFPexplosive}. The proof of Theorem \ref{Th:SFPexplosive} is somewhat simpler than that of Theorem \ref{Th:GIRGexplosive}, since the model is already defined on an infinite space, using translation invariant connection probabilities. As a result, there is no need to extend the underlying space, and there is also no need to couple the model to one with limiting connection probabilities. 
	\begin{proof}[Proof of Theorem \ref{Th:SFPexplosive}]
		In \cite[Theorem 1.7]{HofKom2017}, a lower bound for the $L$-distance is already proved, stating that 	
		\be\lim_{m\to \infty}d_L\left(0,\lfloor m\underline{e}\rfloor\right) - Y^{\mathrm{S}}(0)-Y^{\mathrm{S}}(\lfloor m \underline e\rfloor)  \ge 0.\ee
		The proof is analogous to the proof of Proposition \ref{Prop:lowerbound}, i.e. it uses a BRW upper bound to bound the maximal displacement of the clusters around $0$ and $\lfloor m\underline{e}\rfloor$, and a lower bound as \eqref{eq:lower-distr}, the sum of the shortest path leading to the boundary of vertices within graph distance $k_n$. Then, it can again be shown that these variables tend to the explosion time of the two vertices, as was done between \eqref{eq:lower-distr} -- \eqref{eq:intersect}.
		
		To establish the corresponding upper bound, we should connect the two shortest explosive paths starting from $0, \lfloor m \underline e\rfloor$ in $\mathrm{SFP}_{W,L}$ by an arbitrarily short path. This is the missing direction from \cite[Theorem 1.7]{HofKom2017}. To do this, we use weight-dependent percolation as for GIRG. The connection probabilities of vertices in $\mathrm{SFP}_{W,L}$ are too specific to be able to say that the percolated graph - as in Definition \ref{def:percolationprob} - is still a scale-free percolation model. Hence, let us generalise the connection probability in \eqref{SFP} in a similar manner as for GIRG. Instead of \eqref{SFP}, let us assume that, for some function $h_{\mathrm{SFP}}$,
		\be \Pv(u\leftrightarrow v \mid  \| u-v\|\ge 1, (W_i)_{i\in \Z^d})=h_{\text{SFP}}(\|u-v\|, W_u, W_v),\ee
		that satisfies for some $\wit \alpha>d, \la>0$, and with $l(w)=\exp\{-a_2 \log (w)^{\gamma}\}$ with $a_2>0, \gamma\in(0,1)$,
		\be\label{eq:connection-new}
		\underline c_1\left(l(w_u)l(w_v)\wedge\lambda \|u-v\|^{-\wit \alpha}w_u w_v\right)\leq h_{\text{SFP}}(\|u-v\|, w_u, w_v) \leq \overline C_1\left(1\wedge\lambda \|u-v\|^{-\wit \alpha}w_u w_v\right),
		\ee
		for some constants $\underline c_1, \overline C_1\in(0, 1]$. Note that with this new assumption, Claim \ref{claim:percolation} remains valid: whenever we  apply weight-dependent percolation to non-nearest neighbour edges of an SFP model  satisfying \eqref{eq:connection-new}, with percolation function $p(w_u,w_v)$ satisfying the inequality in \eqref{eq:pw}, the obtained percolated subgraph $\cb G^{\mathrm{p}}$ contains $\underline G^{\mathrm p}$, an instance of a scale-free percolation model again satisfying \eqref{eq:connection-new}, with the same parameters $\wit \alpha, \wit \tau$ as before. Further, the boxing method developed in  Lemma \ref{lemma:weightconnectioncenters} remains valid for the SFP model under \eqref{eq:connection-new}, since we only use the upper and lower bound on the conditional probabilities of edges on the environment. The proof even becomes easier since the number of vertices in each subbox is deterministic in SFP.

		We argue that a result similar to Proposition \ref{Prop:upperbound} is valid for SFP as well. For SFP, the event $\wit A_{n,K}$ holds with probability $1$, since it is not necessary to couple the model to its limit.  
		Let us set $0:=v^1, \lfloor m \underline e \rfloor:=v^2$. Similar to the proof for GIRG, we denote by $\cb\pi^q_{\mathrm{opt}}$ the shortest infinite path emanating from $v^q$. We fix a $K>0$ and denote by $\cb\wit v^q(K)$ the first vertex on $\cb\pi^q_{\mathrm{opt}}$ with weight at least $K$, $\cb q\in\{1,2\}$. These vertices exists a.s.\ since the subgraph spanned by vertices with weight $<K$ cannot be explosive, see Corollary \ref{lem:expgen}. 
	Then, as in the proof of Proposition \ref{Prop:upperbound}, we apply weight-dependent percolation on $\mathrm{SFP}_{W,L}$ with \cb threshold function  $\mathrm{thr}(w_u,w_w)$ given in \eqref{eq:thr-set}, and use the edge-lengths $L_e$ to realise this percolation, as in \eqref{percolationprob}. This yields $p(w_u, w_v)$ as in \eqref{eq:p-set} above, satisfying the conditions of Claim \ref{claim:percolation}. So, the percolated graph $G^{\mathrm{thr}}$ contains a subgraph $\cb \underline G^{\mathrm{p}}$ that is still an SFP, with new edge weights $\cb W^{\mathrm{p}}=W\exp\{-c \log(W)^{\wit \gamma} \}$, i.e.  here $\al=1$. \cz
		
		We connect the two vertices $\wit v^1(K), \wit v^2(K) $ by a path of length at most $\ve_K$ in $\cb \underline G^{\mathrm p}\subseteq  G^{\mathrm{thr}}$. For this we can word-by-word follow the proof of Proposition \ref{Prop:upperbound} in Section \ref{sec:bestexplosive}: we build two boxing systems, centred around $\wit v^q(K)$ with parameter 
		$\mu^{(q)}:=\cb W_{\wit v^q(K)}^{\mathrm{p}}\cz\ge m(K)$. When $\mu^{(1)}\le \mu^{(2)}$, we modify the boxing system around $\mu^{(1)}$ at $r$ given in \eqref{modify-point}, so that the $r$th box is somewhat smaller and has parameter $\mu^{(2)}$. The diameters of boxes and subboxes with index $r+j, j\ge 0$ in system $1$ are equal to those with index $j\ge 0$ in System-$2$. Finally we define $k^\star$ in \eqref{eq:kstar} as the last index $k$ that $\mathrm{Box}_{r+k}^{(1)}\cap \mathrm{Box}^{(2)}_k=\emptyset$ and then we centre $\mathrm{Box}_{k^\star+1}^{(1,2)}$ at $\lfloor m/2\cdot \underline e\rfloor$  to contain both $\mathrm{Box}_{r+k^\star}^{(1)}, \mathrm{Box}_{k^\star}^{(2)}$.
		Then, the paths $\pi^{(q)}$ and their connection $\pi^{\mathrm{conn}}$ described in \eqref{eq:pi12}, \eqref{eq:piconn} exists with probability $\eta(K)$, that is at most twice the rhs of \eqref{boundsweight-error}, with $\mu$ replaced by $m(K)$. The length of this connection is at most $\ve_K$, the rhs of \eqref{lengthboxingpath}. 
		
		\cb This argument, for a fixed $K$, establishes a connection between the best explosive paths $\pi_{\mathrm{opt}}(0)$ and $\pi_{\mathrm{opt}}(\lfloor m\underline e\rfloor)$ that has length at most $|\pi_{\mathrm{opt}}(0)|_L+ \pi_{\mathrm{opt}}(\lfloor m\underline e\rfloor)+\ve_K$, and is not successful with probability $\eta(K)$. \cz
		To obtain an \emph{almost sure} connection, we repeat this procedure for an (increasing) sequence $(K_i)_{i\ge 1}$ so that the error probability $\eta(K_i)$ is summable. Then, by the Borel-Cantelli lemma, it will happen only finitely many times that a connection between $\wit v^1(K), \wit v^2(K)$ with length at most $\ve_{K_i}$ can \emph{not} be established. 
		Since $\cb|\pi_{\mathrm{opt}}^q|_L\cz=Y^{\mathrm{S}}(v^q)$,  the length of the total path constructed is at most $Y^S(0)+Y^S(\lfloor m \underline e\rfloor )+\ve_{K_i}$, for arbitrarily small $\ve_{K_i}>0$.
		Asymptotic independence of $Y^S(0),Y^S(\lfloor m \underline e\rfloor )$ as $m\to \infty$ was already part of \cite[Theorem 1.7]{HofKom2017}, and can be established in a similar manner as we did for the GIRG, see the proof of \eqref{eq:dl-2} between \eqref{eq:indep-111} -  \eqref{eq:end-indep}.
	\end{proof}
\section{Hyperbolic random graphs and threshold assumption}\label{s:hyperbolic}
	In this section we discuss the relation between hyperbolic random graphs and GIRGs, and show that Theorem \ref{Th:GIRGexplosive} is valid for hyperbolic random graphs. 
	For the model description of hyperbolic random graphs, we follow \cite{GugPanaPeter12} and \cite{KriPap10}.
Let us denote by $(\phi_v,r_v)$ the (hyperbolic) angle and radius of a location of a point within a disk of radius $R$, and  define the hyperbolic distance $d_H^{(n)}(u,v)$ of two vertices at $(\phi_u, r_u ), (\phi_v, r_v)$ by the equation
	\be \cosh(d_H^{(n)}(u,v)) := \cosh(r_u) \cosh(r_v) - \sinh(r_u) \sinh(r_v) \cos(\phi_u - \phi_v).\ee
	
	\begin{definition}[Hyperbolic random graphs]
	For parameters $C_H, \al_H, T_H>0$, set $R_n=2\log n +C_H$, and sample $n$ vertices independently from a circle of radius $R_n$ so that for each $v\in[n]$, $\phi_v$ is uniform in $[0, 2\pi]$, and $r_v\in[0,R_n]$ follows a density $f_n(r):=\al_H \sinh(\al_H r)/(\cosh(\al_H R_n)-1) $, \emph{independently} of $\phi_v$.
	\cb Then two vertices are connected in the \emph{threshold} hyperbolic random graphs \cz whenever $d_H^{(n)}(u,v)\le R_n$, while in a parametrised version \cite[Section VI]{KriPap10} they are connected independently of everything else, with probability 
	\be \label{eq:phdh} p_H^{(n)}(d_H^{(n)}(u,v)):=\big( 1+ \exp\{ (d_H^{(n)}(u,v)-R_n)/ 2T_H\}\big)^{-1}.\ee 
	Let us denote the obtained random graphs by $\mathrm{HG}_{\al_H, C_H, T_H}(n)$ when \eqref{eq:phdh} applies and $\mathrm{HG}_{\al_H, C_H}(n)$ when the threshold $d_H^{(n)}(u,v)\le R_n$ is applied. 
\end{definition}

	\cb  We show below that the power-law exponent of the degree distribution in HRG is $\tau_H=2\al_H+1$.
	The most important theorem of this section is the following: 
		\begin{theorem}\label{Th:hyperbolic} Consider the models $\mathrm{HG}_{\al_H, C_H, T_H}(n)$ and $\mathrm{HG}_{\al_H, C_H}(n)$ with $\al_H\in(1/2, 1)$. Assign to each present edge $e$ and edge-length $L_e$, and i.i.d. copy of the random variable $L$. Then the results in Theorem \ref{Th:GIRGexplosive} are valid for these models, i.e. the weighted distance between two uniformly chosen vertices in the giant component converges in distribution.
	\end{theorem}\cz
	We show this theorem by deterministically transforming HRGs into a GIRG model, and showing that the assumptions in Theorem \ref{Th:GIRGexplosive} all hold for the obtained GIRG. 
	The connection to GIRGs is derived as follows: set $d:=1, \mathcal X_1:=[-1/2, 1/2]$, and let, for a vertex $v=(\phi_v,r_v)$,  
	\be\label{mapping} x_v:=(\phi_v-\pi)/(2\pi),\quad W_v^{(n)}:=\exp\{(R_n-r_v)/2\}. \ee
For threshold HRGs, we need to modify Assumptions \ref{assu:GIRGgen}, \ref{assu:extendable}. So, let us  introduce two functions, for not necessarily equal $\underline a_1, \overline a_1>0, a_2>0$ and $\gamma>0$,
	\be\ba  \label{g-inf}
	\overline g_\infty(x,w_1,w_2)&:= \ind_{\{\|x\|\leq \overline a_1 (w_1 w_2)^{1/d}\}}, \\
	  \underline g_\infty(x,w_1,w_2)&:=e^{-a_2((\log w_1)^\gamma+(\log w_2)^\gamma)}\ind_{\{\|x\|\leq \underline a_1 (w_1 w_2)^{1/d}\}}.
	\ea\ee 	
	and consider the following two assumptions (in place of Assumptions \ref{assu:GIRGgen}, \ref{assu:extendable}):
	\begin{assumption}\label{assu:GIRG-inf}There exist parameters $\gamma\in(0,1)$, $a_1, a_2 \in \R_+$ and $0<c_1\leq C_1 <1$, such that for all $n$, $g_n$ in \eqref{eq:gn-intro} satisfies 
		\be\label{GIRGgeneral-inf}
		c_1 \underline{g}_\infty\big(n^{1/d}(x_u-x_v),W_u^{(n)},W_v^{(n)}\big) \le g_n\big(x_u,x_v,(W_i^{(n)})_{i\in [n]}\big)\le   C_1 \overline{g}_\infty\big(n^{1/d}(x_u-x_v),W_u^{(n)},W_v^{(n)}\big).
		\ee
	\end{assumption} We call a model satisfying Assumption \ref{assu:GIRG-inf} a \emph{threshold GIRG}.
	Assumption \ref{assu:extendable} cannot hold directly for threshold GIRGs, since whenever the limiting function contains an indicator that happens to be $0$, the error bound in \eqref{eq:h-intro} cannot be satisfied. 
	As a result, we modify Assumption \ref{assu:extendable} as follows:
	\begin{assumption}[Limiting probabilities for threshold GIRG]\label{assu:extendable-2}
	\cb Recall Definition \ref{def:GIRG} \cz and set $(\mathcal{X}_d, \nu):=([-1/2, 1/2]^d, \mathcal{L}eb)$. We assume that on some event $\CL_n$ that satisfies $\lim_{n\to \infty}\Pv(\CL_n) =1$, there exists a function $h:  \R^d\! \times\! \R_+^2\to [0,1]$ and intervals $I_\Delta(n)\subseteq \R_+, I_w(n)\subseteq [1,\infty)$  such that $g_n$ in \eqref{eq:gn-intro} satisfies for $\nu$-almost every $x \in  \mathcal{X}_d$, and all fixed $\Delta,w_u,w_v$ with $\|\Delta\|\in I_\Delta(n), w_u, w_v\in I_w(n)$, that for some set $ \CI_{w_u, w_v}(n)$ with measure at most $\epsilon(n)$, 
		\be \label{eq:h-intro-2}
		|\wit g_n(x, x+\Delta/n^{1/d}, w_u, w_v, (W_i^{(n)})_{i\in [n]\setminus\{u,v\}})- h(\Delta, w_u, w_v)|
		\le    \ind\{\Delta\in (w_u w_v)^{1/d}\CI_{w_u, w_v}(n) \},
		\ee
		with   
		$\epsilon(n)\to 0, I_\Delta(n)\to (0, \infty), I_w(n)\to [1,\infty)$ as $n\to \infty$.	
		\end{assumption}
The next two theorems tell us that weighted distances in threshold GIRGs still converge, and that  weighted distances in $\mathrm{HG}_{\al_H, C_H, T_H}(n)$ and $\mathrm{HG}_{\al_H, C_H}(n)$ converge in distribution:		
	\begin{proposition}\label{Th:threshold}
		Theorem \ref{Th:GIRGexplosive} remains valid when Assumptions \ref{assu:GIRGgen} and \ref{assu:extendable} are switched to Assumption \ref{assu:GIRG-inf} and \ref{assu:extendable-2}, respectively.
	\end{proposition}

		\begin{theorem}\label{Th:hyperbolic2} With the mapping in \eqref{mapping}, $\mathrm{HG}_{\al_H, C_H, T_H}(n)$ becomes a $\mathrm{GIRG}_{W,L}(n)$ model that satisfies Assumptions \ref{assu:GIRGgen}, \ref{assu:weight}, \ref{assu:extendable} with parameters $\al:=1/T_H, \tau:=2\al_H+1$ and $d=1$. Further, $\mathrm{HG}_{\al_H, C_H}(n)$ satisfies Assumptions \ref{assu:GIRG-inf}, \ref{assu:weight}, \ref{assu:extendable-2} with parameters $\tau:=2\al_H+1$ and $d=1$.
		In particular, Theorem \ref{Th:GIRGexplosive} is valid for $\mathrm{HG}_{\al_H, C_H, T_H}(n)$ and $\mathrm{HG}_{\al_H, C_H}(n)$.
	\end{theorem}
	Observe that Theorem \ref{Th:hyperbolic} follows directly from Proposition \ref{Th:threshold} and Theorem \ref{Th:hyperbolic2}.
	\begin{proof}[Proof of Theorem \ref{Th:hyperbolic2} subject to Proposition \ref{Th:threshold}]
		We start showing that $\mathrm{HG}_{\al_H, C_H, T_H}(n), \mathrm{HG}_{\al_H, C_H}(n)$ satisfy Assumption \ref{assu:weight}.
		In \cite[Lemma 7.2]{BriKeuLen15} it is shown that when $\al_H>1/2$, the weight distribution follows a (weak) power-law with $\tau=2\al_H+1$ in the limit. This is a similar result to showing that the degree distribution follows a power law in the limit, a result in \cite{GugPanaPeter12}. Assumption \ref{assu:weight} also requires an investigation on the $n$-dependent weight distribution, that we show now. The condition $r_v\in[0, R_n]$ implies that $W_v^{(n)}\in[1,\mathrm{e}^{C_H/2}n]$. For any $1\leq x\le \mathrm{e}^{C_H/2} n$,
		\be\ba \Pv(W_v^{(n)}> x) &= \Pv( r_v< R_n-2\log x) = \int_0^{R_n-2\log x} f_n(t)\mathrm dt = \frac{\cosh(\al_H (R_n- 2\log x)) -1}{\cosh(\al_H R_n)-1}\\
		&=  \frac{ x^{-2\al_H}(n^{2} \e^{C_H})^{\al_H} (1+ (x\e^{-C_H/2}/n)^{4\al_H})/2-1}{(n^{2} \e^{C_H})^{\al_H} (1+ (\e^{-C_H/2}/n)^{4\al_H})/2-1}.\ea\ee
		So, set $[1,M_n]:=[1,e^{C_H/2}n]$ in Assumption \ref{assu:weight}, and then
		\be\ba \Pv(W_v^{(n)}> x) &= x^{-2\al_H}\ell_n(x), \ea\ee
		\cb where for all $x\in[1,M_n]$ $1/2\le\ell_n(x) \le 2+6\mathrm{e}^{-C_H\alpha_H/2}$ holds. Hence, \eqref{powerlaw-n} is satisfied with $\underline \ell \equiv 1/2, \overline \ell \equiv 2+6\mathrm{e}^{-C_H\alpha_H/2}$\cz. 
		For fixed $x\in [1,\infty)$,  $\Pv(W_v^{(n)}>x)$ converges to $\Pv(W> x)=x^{-2\al_H}$, thus \eqref{powerlaw} is also satisfied with $\tau:=2\al_H+1$. 
		Next we show that $\mathrm{HG}_{\al_H, C_H}(n)$ satisfies Assumption \ref{assu:GIRG-inf}. 
		Let 
		\be \theta_n(r_u,r_v):=\argmax_{\phi\le \pi} \{d_H^{(n)}((0,r_u),(\phi,r_v))\le R_n\}=\arccos\Big\{ \frac{\cosh(r_u)\cosh(r_v)-\cosh(R_n)}{\sinh(r_u)\sinh(r_v)}\Big\}\ee
		be the largest angle of a vertex with radius $r_v$ that a vertex at $(0,r_u)$ is connected to.
		In \cite[Lemma 3.1]{GugPanaPeter12}, it is shown that, when $r_u+r_v\ge R_n$,
		\be\label{angle-estimate} 2\e^{(R_n-r_u-r_v)/2}(1+\Theta(\e^{R_n-r_u-r_v})) \le \theta_n(r_u,r_v) \le 2\e^{(R_n-r_u-r_v)/2}(1+\Theta(\e^{-2r_u}+ \e^{-2r_u})). \ee
		On the other hand, when $r_u+r_v\le R_n$, then $ \theta_n(r_u,r_v)=\pi$ due to the triangle inequality.
		Note that $r_u+r_v\le R_n$ is equivalent to $w_u w_v/n\ge \e^{C_H/2}$. Using \eqref{mapping}, this corresponds to that for some $ c_1,  C_1, \underline a_1, \overline a_1>0$,
		\be \ba \|x_u-x_v\| &\le c_1\min\{1, \underline a_1 w_u w_v/n\} \quad \Rightarrow\quad  d_H(u,v)\le R_n, \\
		\|x_u-x_v\| &\ge  C_1\min\{1, \overline a_1 w_u w_v/n \}\quad \Rightarrow \quad d_H(u,v)\ge R_n.\ea \ee
		which is also shown in \cite[Lemma 7.5]{BriKeuLen15}. 
		Hence, Assumption \ref{assu:GIRG-inf} is satisfied. Next we show that Assumption \ref{assu:GIRGgen} is satisfied for $\mathrm{HG}_{\al_H, C_H, T_H}(n)$.
Let $s_n:=d_H^{(n)}((x_u,w_u), (x_v,w_v) )-R_n.$ Note that 
		\be\label{eq:esn} \e^{s_n}=2\cosh(d_H^{(n)})\e^{-R_n}+\e^{-s_n}\e^{-R_n}=2\cosh(d_H^{(n)})\e^{-R_n} + \e^{-s_n} \Theta(n^{-4}).\ee
		 By noting that $r_q=R_n-2\log w_q, \phi_q=2\pi x_q+\pi$ for $q\in\{u,v\}$, using a trigonometric identity $\cosh(x-y)=\cosh(x)\cosh(y)-\sinh(x)\sinh(y)$,
		\be\ba\label{ch-dh} \cosh(d_H^{(n)})&\e^{-R_n}=\cosh(r_u-r_v)\e^{-R_n} + (1-\cos(2\pi(x_u-x_v))) \sinh(r_u)\sinh(r_v)\e^{-R_n}\\
		&=\frac{(w_u/w_v)^2+(w_v/w_u)^2}{n^2}\e^{-C_H}\\
		&\ +2\pi^2\cb \|x_u-x_v\|^2\cz\Big(1+\Theta(\cb\|x_u-x_v\|^2\cz)\Big)\frac{\e^{C_H}n^2}{4w_u^2 w_v^2} \Big(1-\frac{\e^{-2C_H}w_u^4}{n^4} \Big)\Big(1-\frac{\e^{-2C_H}w_v^4}{n^4} \Big). \ea\ee
		Since $p^{(n)}_H(d_H^{(n)})=(1+\e^{s_n/(2T_H)})^{-1}$, from a combination of \eqref{eq:esn} and \eqref{ch-dh}, using that $(1\wedge x^{-1})/2 \le(1+x)^{-1}\le 1\wedge x^{-1}$, it follows that for any angle difference \cb$\|x_u-x_v\|$\cz, 
		\be c_1 \Big(1 \wedge  \underline a_1\Big(\frac{w_u w_v}{n\cb\|x_u-x_v\|\cz}\Big)^{1/T_H}\Big)\le  p^{(n)}_H(d_H^{(n)}) \le C_1 \Big(1 \wedge  \overline a_1\Big(\frac{w_u w_v}{n\cb\|x_u-x_v\|\cz}\Big)^{1/T_H}\Big), \ee
		that is, $\mathrm{HG}_{\al_H, C_H, T_H}(n)$ satisfies Assumption \ref{assu:GIRGgen} with $\tau=2 \al_H+1, \al=1/T_H$ as well. This has appeared in   \cite[Lemma 7.5]{BriKeuLen15}. 
		It remains to show that Assumption \ref{assu:extendable} is satisfied for $\mathrm{HG}_{\al_H, C_H, T_H}(n)$ and Assumption \ref{assu:extendable-2} is satisfied for $\mathrm{HG}_{\al_H, C_H}(n)$.
		Note that the original GIRG model allows for the connection probabilities to depend on the whole collection of weights $(W_i)_{i\in[n]}$, while in $\mathrm{HG}_{\al_H, C_H, T_H}(n)$ and $\mathrm{HG}_{\al_H, C_H}(n)$ this is actually not the case, connection probabilities only depend on $w_u, w_v$ and $x_u-x_v$ per definition.
		We need to show that a limiting connection probability $h(\Delta, w_u,w_v)$ between vertices with $u=(x_u,w_u), v=(x_u+\Delta/n,w_v)$ exists when we set $w_u,w_v$ to be fixed \emph{constants}, and that \eqref{eq:h-intro} holds. 
		
		We continue showing that the limiting function $h(\Delta, w_u, w_v)$ exists when we set $x_u-x_v=\Delta/n$. Indeed, combining \eqref{eq:esn} with twice the rhs of \eqref{ch-dh} equals then, (since $w_u, w_v\ge 1$)
		\be \label{eq:ch-dh-2}\ba 
		&\e^{s_n}= \frac{\Delta^2\e^{C_H}\pi^2}{w_u^2w_v^2}\Big(1+\Theta(\Delta^2/n^2) + \Theta((w_u^4+w_v^4)/n^4)\Big)+ \Theta\big(\frac{w_u^2+w_v^2}{n^2}\big)+\Theta\big(\frac{w_u^2w_v^2}{\Delta^2n^{4}}\big). \ea\ee
		Thus, its limit exists for every fixed $\Delta, w_u, w_v$  with $\e^s=\Delta^2\e^{C_H}\pi^2/(w_u^2w_v^2)$.
		Hence,
		\be\label{ph-limit} p^{(n)}_H(d_H^{(n)}) \to \big(1+\e^{s/(2T_H)}\big)^{-1}= (1+(\e^{C_H/2}\cb\|\Delta\|\cz\pi /(w_u w_v) )^{1/T_H} )^{-1}=:h(\Delta, w_u, w_v),\ee
		while for the threshold case,
				\be\label{ind-limit} \ind\{d_H^{(n)}\le R_n\}\to \ind\{s(\Delta, w_u, w_v)\le 0\}= \ind\{ \cb\|\Delta\|\cz \le \e^{-C_H/2} w_uw_v/\pi\}=:h_\infty(\Delta, w_u,w_v).
		\ee
		From \eqref{ph-limit} and \eqref{ind-limit}, it follows that the limiting functions $h,h_\infty$ also satisfy Assumptions \ref{assu:GIRGgen} and \ref{assu:GIRG-inf}, respectively.
		What is left is to show that \eqref{eq:h-intro} in Assumption \ref{assu:extendable} is satisfied. For this we need to estimate the relative difference between the connection probabilities and their limit.
		Let $p(x):=(1+x^{1/2T_H})^{-1}$. Then $p_H^{(n)}(d_H^{(n)})=p(\e^{s_n}), h(\Delta, w_u, w_v)=p(\e^s)$. So,
		\be\label{error-est-3}\ba    |p_H^{(n)}(d_H^{(n)})-p(\e^s)|/p(\e^s)&=|p'(\e^s) (\e^{s_n}-\e^s)|/p(\e^s)+o(\e^{s_n}-\e^s)/p(\e^s)\\
		&= p(\e^s) \e^{s/(2T_H)} |\e^{s_n}/\e^s-1|+o((\e^{s_n}-\e^s)|/p(\e^s)).\ea\ee
		since $p'(\e^s)=-\e^{s(1/(2T_H)-1)} p(\e^s)^2/(2T_H)$. Using  
		\eqref{eq:ch-dh-2},	and that $\e^s=\Theta(\Delta^2/(w_u^2w_v^2)$,
		\be |\e^{s_n}/\e^s -1| \le \Theta(\Delta^2/n^2 + (w_u^4+w_v^4)/n^4+ w_u^2 w_v^2(w_u^2+w_v^2)/(\Delta^2 n^2) + w_u^4 w_v^4/ \Delta^4 n^4).  \ee
		Finally, returning to \eqref{error-est-3}, and recalling that $1/T_H=:\al$, $\e^{s/(2T_H)}= \Theta(\Delta/w_u w_v)^{\al}$ and so $p(\e^s)\le  \min\{1, \Theta(w_uw_v/\Delta)^{\al}\}$, 
		\be \ba  |p_H^{(n)}(d_H^{(n)})-p(\e^s)|/p(\e^s)& \le \min\{1, \Theta(w_uw_v/\Delta)^{\al}\} \Theta((\Delta/w_u w_v)^{\al})\\
		&\cdot \Theta(\Delta^2/n^2 + (w_u^4+w_v^4)/n^4+ w_u^2w_v^2(w_u^2+w_v^2)/(\Delta^2 n^2) + w_u^4 w_v^4/ (\Delta^4 n^4))\\
		&=\Theta(\Delta^2/n^2+(w_u w_v/\Delta)^{4-\al}/n^{4} + (w_uw_v/\Delta)^{2-\al} (w_u^{2}+ w_v^{2})/n^2)\\
		\ea\ee
Note that the error terms are at most $\Theta(n^{-1})$ whenever \cb$\|\Delta\|\cz\in [n^{-\beta(\al)}, n^{\beta(\al)}], w_u,w_v\in [1,n^{\beta(\al)}]$, where 
\be\ba \beta(\al)&=\ind\{\al<2\} \frac{1}{3(2-\al)+2}+\ind\{\al\ge 2\} \frac{1}{\alpha}. \ea \ee
		Thus, Assumption \ref{assu:extendable} is satisfied with $\epsilon(n)=n^{-1}$ and $I_\Delta(n)=[n^{-\beta(\al)}, n^{\beta(\al)}], I_w(n):=[1,n^{\beta(\al)}]$.		
		We continue showing that Assumption \ref{assu:extendable-2} holds for $\mathrm{HG}_{\al_H, C_H}(n)$. Combining \eqref{eq:ch-dh-2} and \eqref{ind-limit},
\be \ba\label{eq:inderror} \ind\{d_H^{(n)}\le R_n\}= \ind\{e^{s_n}\le 1\}=\ind\Big\{ \cb\|\Delta\|\cz \le \frac{w_uw_v}{\e^{C_H/2}\pi} \cdot \frac{1+ \Theta(w_u^2w_v^2/(\Delta^2n^4)+(w_u^2+w_v^2)/n^2\big)}{1+\Theta(\Delta^2/n^2) + \Theta((w_u^4+w_v^4)/n^4)}\Big\},\ea \ee
and hence
\be  \ind\{d_H^{(n)}\le R_n\}-\ind\{ \cb\|\Delta\|\cz \le \frac{w_uw_v}{\e^{C_H/2}\pi}\} \le\ind\{ \cb\| \Delta\|\cz\in w_u w_v \CI_{w_u, w_v}(n)\}, \ee
where $\CI_{w_u, w_v}(n)$ is an interval of length
\be |\CI_{w_u, w_v}(n)|=
 \frac{1}{\e^{C_H/2}\pi} \Theta( w_u^2w_v^2/(\cb\|\Delta\|^2\cz n^4)+ (w_u^2+w_v^2)/n^2+\cb\|\Delta\|^2\cz /n^2+(w_u^4+w_v^4)/n^4 ).\ee
 \cz Note that the length of this interval is at most $1/n$ whenever \cb $\|\Delta\|\cz\in [n^{-1/2}, n^{1/2}], w_u, w_v\in [1, n^{1/2}].$
 \cz Thus, Assumption \ref{assu:extendable-2} is satisfied with $\epsilon(n):=1/n, I_\Delta(n):= [n^{-1/2}, n^{1/2}], I_w(n):=[1,n^{1/2}]$.
 We comment that for HRGs the underlying space is the one-dimensional \emph{torus} $\cb[-1/2, 1/2)$ instead of $[-1/2,1/2]$. This is not an issue since in the proofs we made sure the paths constructed does not touch the boundary of $\Xdn$, that is equivalent to not wrapping around the torus.
	\end{proof}
		The content of the following claim is the modified version of Claim \ref{claim:edge-containment} for threshold GIRGs.  

	\begin{claim}\label{claim:edge-containment-2} Suppose that Assumption \ref{assu:extendable-2} holds instead of Assumption \ref{assu:extendable}. Then, Claim \ref{claim:edge-containment} remains valid with a different proof.
	\end{claim}
		\begin{proof} By the coupling described in Claim \ref{cl:vertex-containment}, the vertex set of $\BGIRG$ is a subset of $\CV_{1+\xi_n}$. We describe the new coupling of the edges. Following the proof of Claim \ref{claim:edge-containment}, the vertex-weights are not equal in $\BGIRG$ and in $\Eplus$ with probability at most $\epsilon_{\mathrm{TV}}(n)$, i.e.
		\be\label{eq:weightcouplingprob-2} \Pv(W^{(n)}_{i_1}\neq W_{i_1} \mbox{ or }  W^{(n)}_{i_2}\neq W_{i_2})=2\epsilon_{\mathrm{TV}}(n).\ee	
		Given that the two pairs of weights are equal and equal to $W_{i_1}, W_{i_2}$, the edge is present in $\overline \CE_{\la}$ when  $\|x_{i_1}-x_{i_2}\|\le \overline a_1 (W_{i_1} W_{i_2})^{1/d}$, and, by \eqref{GIRGgeneral-inf}
		that bounds the presence of the same edge in $\CE_B(n)$ (note that the $n$-dependence in \eqref{GIRGgeneral-inf} disappears when moving from $\mathrm{GIRG}_{W,L}(n)$ to $\BGIRG$). Hence, $\CE_B(n)\subseteq \overline \CE_{1+\xi_n}$ holds. We continue coupling $\CE_B(n)$ and $\CE_{1+\xi_n}$.
		A de-coupling happens precisely when 
		 for a possible edge $e$ connecting vertices with  locations and weights $(x_{i_1}, W_{i_1}), (x_{i_2}, W_{i_2})$, $\|x_{i_1}-x_{i_2}\|\in (W_{i_1} W_{i_2})^{1/d}\CI_{W_{i_1} W_{i_2}}(n)$. Thus, conditioned on the vertex weights and on the location $x_{i_1}$, 
\be\ba  \Pv&( x_{i_2}-x_{i_1}\in (W_{i_1}W_{i_2})^{1/d}\CI_{W_{i_1} W_{i_2}}(n) \mid   \|x_{i_2}-x_{i_1}\|\le \overline a_1 (W_{i_1}W_{i_2})^{1/d}, E_{\Delta,W}, W_{i_1}, W_{i_2}, x_{i_1})\\
&\le (W_{i_1}W_{i_2} \epsilon(n))/(\overline a_1^d W_{i_1} W_{i_2}) \le \epsilon(n)/\overline a_1^d. \ea\ee
 The condition that the event $E_{\Delta,W}$ holds (see Claim \ref{claim:edge-containment}) guarantees that this bound is true. As a result, \eqref{eq:conditional} remains valid with $2\epsilon_{\mathrm{TV}}(n)+\epsilon(n)/\overline a_1^d$ instead of $\epsilon(n)$ and a union bound finishes the proof.
\end{proof}

	\begin{proof}[Proof of Theorem \ref{Th:threshold}] We discuss how the lemmas, claims and propositions in this paper need to be adjusted when Assumptions \ref{assu:GIRG-inf}, \ref{assu:extendable-2} hold instead of Assumptions \ref{assu:GIRGgen}, \ref{assu:extendable}.
Let us equip the $\Elambda$ model with edge probability $h_\infty$, such that $\wt g_n$ satisfies Assumption \ref{assu:extendable} with $h:=h_\infty$. This implies that Claim \ref{claim:hg} holds for $h_\infty$, i.e. there exist constants $\underline c_1,\overline C_1$, with $0<\underline c_1\leq c_1\leq C_1\leq \overline C_1<\infty$, such that
		\be 
		\underline c_1 \underline g_\infty(\Delta, w_u,w_v)\leq h_\infty(\Delta, w_u,w_v)\leq \overline C_1 \overline g_\infty(\Delta,w_u,w_v).
		\ee 
		Then, equip the $\Eupper$ model with the edge probability $\overline C_1 \overline g_\infty$. We have seen in Claim \ref{claim:edge-containment-2} that the coupling of edges in \eqref{claim:edge-containment} remains valid. 		 When changing the definition of the Bernoulli BRW in \eqref{berbrw-edge-prob} to $N_x^B(\underline i)\overset{d}{=}\overline C_1 \overline g_\infty(x_{\underline i}-x,W_{\underline i},W_x)$, Lemma \ref{lemma:ballsbrw} and its extension in \eqref{eq:ball-containment} hold, since in the coupling it is only used that \eqref{upperboundingBRW} holds between $\mathrm{BerBRW}_{\la}$ and  $\Eupper$, and the connection probability is at least as much in $\Eupper$ as in $\Elambda$ between two vertices with given weights and locations, which stays true under Assumption \ref{assu:GIRG-inf}. Similarly, in the proof of Theorem \ref{thm:exp-charact}, if one replaces the rv $\mathrm{Ber}(\overline C_1\min\{1,\|y\|^{-\alpha d}(ws)^\alpha\})$ in \eqref{eq:kernel-1} by $\ind_{\{\|y\|\leq \overline a_1(ws)^{1/d}\}}$, and in \eqref{eq:mu-ber} drops the second sum and changes the constraint in the first sum to $y\in \CV_\lambda: \|y\|\leq \overline a_1(ws)^{1/d}$, then
		\be 
		\mu(w,ds)\leq F_W(ds)\lambda \overline a_1^d \mathrm{Vol}_d ws, 
		\ee
		where $\mathrm{Vol}_d$ is the volume of the d-dimensional unit ball. When we adjust $c_\lambda$ in \eqref{eq:muber-2} to $\lambda \overline a_1^d \mathrm{Vol}_d$, the rest of the proof follows analogously. Hence, Theorem \ref{thm:exp-charact} and thus Corollary \ref{lem:expgen} follow.
	Likewise, Proposition \ref{Prop:BerBRW} holds, as Claim \ref{claim:Nw} is satisfied. In the proof of Claim \ref{claim:Nw}, \eqref{NwSupperbound} can be reduced to the second integral, changing the constraint in the sum to $\cb \|x\|\leq \overline a_1 (ws)^{1/d}$, which yields
		\be 
		\E{N_w(\geq 1,\geq S)}\leq \overline a_1^d\lambda \mathrm{Vol}_d C_\tau w^{\tau-1} S^{-d(\tau-2)}\ell(S^d/w),
		\ee
		and similarly,
		\be 
		\E{N_w(\geq u,\geq 0)}\leq \overline a_1^d \lambda \mathrm{Vol}_d C_\tau wu^{-(\tau-2)}\ell(u).
		\ee 
		These yield that Proposition \ref{Prop:lowerbound} holds. 
		For the weight-dependent percolation, as defined in Definition \ref{def:percolationprob}, we show that Claim \ref{claim:percolation} holds, with $\cb(W^{(n)}_{\mathrm{p},i})_{i\in[n]}\cz:= (W^{(n)}_i)_{i\in[n]}$, i.e. unchanged weights. By definition, an edge $e=(u,v)$ is present with probability
		\be \ba
		\Pv\big((u,v)\in \CE(\cb G^{\mathrm{thr}}\cz) \mid (x_i, W_i^{(n)})_{i\in [n]} \big) & =p(W_u^{(n)},W_v^{(n)})\wt g_n^{\mathrm B}(x_u,x_v,W^{(n)}_u,W^{(n)}_v,(W^{(n)}_i)_{i\in[n]\backslash\{u,v\}})\\
		& =: \cb\wt g_n^{\mathrm{thr}}(x_u,x_v,W^{(n)}_u,W^{(n)}_v,(W^{(n)}_i)_{i\in[n]\backslash\{u,v\}})\cz,
		\ea\ee
		which we can bound from above and below using Assumption \ref{assu:GIRG-inf}, by
		\be\ba  \label{eq:percupper}
		\cb g_n^{\mathrm{thr}}(x_u,x_v,W^{(n)}_u,W^{(n)}_v,(W^{(n)}_i)_{i\in[n]\backslash\{u,v\}}) & \leq  C_1 \overline g_\infty(n^{1/d}(x_u-x_v),W^{(n)}_u,W^{(n)}_v),\\
		\cb g_n^{\mathrm{thr}}(x_u,x_v,W^{(n)}_u,W^{(n)}_v,(W^{(n)}_i)_{i\in[n]\backslash\{u,v\}})& \geq  c_1 p(W_u^{(n)}, W_v^{(n)})\underline g_\infty(n^{1/d}(x_u-x_v),W_u^{(n)},W^{(n)}_v).
		\ea\ee 
		When $p(w_u, w_v)$ satisfies \eqref{eq:pw}, 
		\[ p(w_u, w_v) \exp\{-a_2 (\log w_u)^{\gamma}-a_2 (\log w_v)^{\gamma}\}\ge  \exp\{-(a_2+c) (\log w_u)^{\gamma_{\max}}-(a_2+c) (\log w_v)^{\gamma_{\max}}\}, \]
		with $\gamma_{\max}=\max\{\gamma, \wit\gamma\}$. So, \eqref{g-inf} is satisfied again with new constant $a_2+c$ in place of $a_2$, and new exponent $\gamma_{\max}\in(0,1)$ in place of the original $\gamma$. Hence we can define $\underline G^{\mathrm{p}} \subseteq G^{\mathrm{thr}}$ as a new threshold GIRG with the edge probabilities
		\be \ba
		\cb g_n^{\mathrm{p}}(x_u,x_v,W^{(n)}_u,W^{(n)}_v)\cz:= &\exp\{-(a_2+c) (\log W_u^{(n)})^{\gamma_{\max}}-(a_2+c) (\log W_u^{(n)})^{\gamma_{\max}}\} \\ &\cdot c_1\cz \ind_{\{\|x\|\leq \underline a_1(W_u^{(n)}W_v^{(n)})^{1/d}\}},
		\ea\ee
		proving together with \eqref{eq:percupper} that the percolated graph contains a subgraph that is again an instance of the threshold GIRG, with the same parameters and the same weights $(W^{(n)}_i)_{i\in[n]}$. So, Claim \ref{claim:percolation} is valid.
	Turning to  Lemma \ref{lemma:weightconnectioncenters},	the proof of the events $F_k^{(1)}$ only concerns weights, so it holds unchanged. Investigating the events $F_k^{(2)}$, using \eqref{radii} and the events $F_k^{(1)}$,  the lower bound on the edge probabilities in \eqref{eq:connectprob} is satisfied for this model as well, as
		\be 
		\|\cki-c_{k+1}^{\sss(j)}\|^d/(W^{(n)}_{\cki}W^{(n)}_{c_{k+1}^{\sss\left(j\right)}})\leq d^{d/2}\mu^{C^k(CD-(1-\delta)(1+C)/(\tau-1))},
		\ee
		which decreases with $\mu$. From here the proof could be followed word-by-word. Thus, Lemma \ref{lemma:weightconnectioncenters} holds as well. It directly follows that Proposition \ref{Prop:existinfcomp} is satisfied as well. Finally, since \cite[Lemma 5.5]{BriKeuLen17} holds for the threshold GIRG model as well, as shown in said paper, the proofs of Claims \ref{claim:bg-to-eone}, \ref{claim:explosive-ray} and \ref{claim:ray-in-box} follow for threshold GIRGs. Hence, Lemma \ref{lem:ank} and Proposition \ref{Prop:upperbound} are satisfied, finishing the proof.
	\end{proof}
{\bf Acknowledgements.} The work of JK is supported by the Netherlands Organisation for Scientific Research (NWO)
through VENI grant 639.031.447 (JK). We thank the unknown reviewers for their comments that helped to improve the presentation of the paper. 

	\bibliographystyle{abbrv}
	\bibliography{refexplosion}	
\end{document}